\documentclass[11pt,reqno]{amsart}
\usepackage[T1]{fontenc}
\usepackage{lmodern}
\usepackage[expansion=false]{microtype}
\usepackage[margin=1in]{geometry}
\usepackage{amsmath,amssymb,amsthm,mathtools,bm,mathrsfs}
\usepackage{enumitem,booktabs,xcolor}
\usepackage[colorlinks=true,linkcolor=blue!40!black,citecolor=blue!40!black,urlcolor=blue!40!black]{hyperref}
\usepackage{zref-clever}
\zcsetup{cap=true}
\hypersetup{pdfauthor={P. M. Aronow and Patrick Lopatto}}
\hypersetup{pdftitle={Quantitative Parisi formulas and fluctuations in the Sherrington--Kirkpatrick model}}
\numberwithin{equation}{section}
\newtheorem{theorem}{Theorem}[section]
\newtheorem{proposition}[theorem]{Proposition}
\newtheorem{lemma}[theorem]{Lemma}
\newtheorem{corollary}[theorem]{Corollary}
\theoremstyle{remark}
\newtheorem{remark}[theorem]{Remark}
\AddToHook{env/theorem/begin}{\zcsetup{reftype=theorem}}
\AddToHook{env/lemma/begin}{\zcsetup{reftype=lemma}}
\AddToHook{env/proposition/begin}{\zcsetup{reftype=proposition}}
\AddToHook{env/corollary/begin}{\zcsetup{reftype=corollary}}
\AddToHook{env/remark/begin}{\zcsetup{reftype=remark}}
\newcommand{\E}{\mathbb E}
\newcommand{\Prob}{\mathbb P}
\newcommand{\R}{\mathbb R}
\DeclareMathOperator{\Var}{Var}
\DeclareMathOperator{\Cov}{Cov}
\DeclareMathOperator{\Tr}{Tr}
\DeclareMathOperator{\supp}{supp}

\newcommand{\dd}{\,\mathrm d}
\newcommand{\1}{\mathbf 1}
\newcommand{\cP}{\mathcal P}

\DeclareMathOperator*{\argmin}{argmin}
\setlist{itemsep=3pt,topsep=5pt}
\allowdisplaybreaks[2]
\title[Quantitative Parisi formulas and SK fluctuations]{Quantitative Parisi formulas and fluctuations in the Sherrington--Kirkpatrick model}
\author{P. M. Aronow}
\author{Patrick Lopatto}
\date{}
\begin{document}
\begin{abstract}
We establish quantitative Parisi formulas and fluctuation bounds for the Sherrington--Kirkpatrick model at zero external field. In particular, using  that the Parisi measure is supported on an interval, we show that for every fixed inverse temperature greater than
one, the variance of the logarithmic partition function lies between orders
$N^{4/15}$ and $N^{7/15}$. The same exponents hold for the ground-state
variance. We also obtain upper and lower bounds on the finite-size bias of the mean free energy and ground-state energy, together with two-sided upper-tail estimates with exponent $6/5$ at both positive and zero temperature.
\end{abstract}
\maketitle
\setcounter{tocdepth}{1}
\tableofcontents

\section{Introduction}\zlabel{sec:main}
The Sherrington--Kirkpatrick (SK) model \cite{SherringtonKirkpatrick1975}
is a mean-field spin glass with random pair interactions. Parisi's
replica-symmetry-breaking theory
\cite{Parisi1979,Parisi1980OrderParameter} describes its limiting free
energy through a variational problem over probability measures. At
finite system size, two further questions arise: how accurately does
the variational value approximate the expected free energy, and how
large are the fluctuations around that expectation? We study these
questions for the SK model at zero external field, both
at positive temperature and for the ground-state energy.

The limiting variational formulas rest on several developments.
Guerra and Toninelli established the thermodynamic limit by
interpolation \cite{GuerraToninelli2002Thermodynamic}. Guerra proved the
Parisi upper bound for the pressure \cite{Guerra2003}, and Aizenman,
Sims, and Starr developed an extended variational principle based on
random overlap structures \cite{AizenmanSimsStarr2003}. Talagrand then
proved the Parisi formula \cite[Theorem~1.1]{Talagrand}. Panchenko
extended the formula to general mixed $p$-spin models, including odd
interactions \cite{Panchenko2014Parisi}, and Auffinger and Chen
established its zero-temperature counterpart
\cite[Theorem~1]{AC}. These results identify the limiting pressure and
ground-state energy. Our aim is to control their finite-size
approximations with enough uniformity to obtain fluctuation bounds and
to pass to zero temperature.

The scale of the fluctuations depends strongly on temperature. In the
zero-field SK model, Aizenman, Lebowitz, and Ruelle proved a Gaussian
limit for the order-one fluctuations of the logarithm of the partition function
at fixed inverse temperature $0<\beta<1$
\cite{AizenmanLebowitzRuelle1987}. At the critical point $\beta=1$,
Chen and Lam obtained a variance upper bound of order $(\log N)^2$
\cite{ChenLam2019Critical}. More recently, Dey and Kang established
variance asymptotics and a central limit theorem on approaching
criticality from the high-temperature side, including at distance
of order $N^{-1/3}$ \cite{DeyKang2026Critical}. At the critical point
itself, Du and Huang proved that the variance is
$\frac16\log N+O(1)$ and established a Gaussian limit
\cite{DuHuang2026Critical}.

The low-temperature, zero-field regime has a different fluctuation
picture. Chatterjee proved superconcentration of the SK free energy,
with a variance bound of order $N/\log N$ for the logarithmic partition function
at each fixed inverse temperature, and related superconcentration to
disorder chaos \cite{Chatterjee2009Chaos}. Chen and Lam extended
superconcentration to non-Gaussian disorder under moment assumptions
\cite{ChenLam2024Universality}. At zero temperature, Chen, Handschy,
and Lerman proved that the ground-state variance is $o(N)$ in mixed
even $p$-spin models at zero field \cite{ChenHandschyLerman2018}.
In the other direction, Chatterjee obtained constant-scale
nonconcentration bounds for both the SK logarithmic partition function and
the ground-state energy \cite{Chatterjee2019LowerBounds}. Additionally, Chen, Dey, and Panchenko
proved variance of order $N$ in the presence of a nonzero external
field, together with a Gaussian limit for mixed even $p$-spin models
\cite{ChenDeyPanchenko2017}.

A complementary approach to fluctuations comes from positive moments
of the partition function. In the replica calculation, Kondor found a
fifth-order dependence of the fractional pressure on the replica
parameter near the critical temperature \cite{Kondor1983}.
Crisanti, Paladin, Sommers, and Vulpiani studied the relation between
this parameter and finite-size fluctuations through a tilted ensemble
of disorder realizations \cite{CrisantiPaladinSommersVulpiani1992}.
Parisi and Rizzo extended the fifth-order calculation throughout the
low-temperature phase and to zero temperature, obtaining the associated
six-fifths exponent for the large-deviation tail
\cite{ParisiRizzo2008,PR}. Aspelmeier developed a different connection
between fluctuations and disorder sensitivity, deriving an exact
relation between the free-energy variance and bond chaos and studying
its consequences through replica calculations
\cite{Aspelmeier2008Chaos}.

Finite-size corrections to the mean have also received substantial
attention. Using Guerra--Toninelli interpolation, Billoire found
numerical evidence for corrections of order $N^{-2/3}$ to the
low-temperature free-energy density \cite{Billoire2006FiniteSize}.
Aspelmeier, Billoire, Marinari, and Moore developed a finite-size
replica-symmetry-breaking analysis predicting corrections of order
$N^{-2/3}$ to the mean energy density and fluctuations of order
$N^{1/6}$ in the total internal energy \cite{AspelmeierEtAl2008FiniteSize}.
Ground-state large deviations and sample-to-sample fluctuations were
investigated numerically by Andreanov, Barbieri, and Martin
\cite{AndreanovBarbieriMartin2004} and by Palassini
\cite{Palassini2008}. The distinction between these questions matters, as
an exponent for a large-deviation tail does not by itself determine
the scale of typical fluctuations, and a bound on the mean correction
need not be sharp enough to recover either scale.

The rigorous study of fractional moments has received significant attention. 
Talagrand developed large-deviation and variational methods for these
quantities \cite{TalagrandLD}. For each fixed $0<s<1$, the limiting
normalized log fractional moment is covered by the one-level
specialization of Contucci and Mingione's multiscale formula
\cite[Theorem~2.1]{CM}. In that specialization, the moment parameter
becomes a lower bound on the cumulative variational parameter.
Chen, Guionnet, Ko, Lacroix-A-Chez-Toine, and Mourrat formulate the
fractional-moment limit as a Parisi minimum over measures with a
prescribed minimum mass at zero \cite[Theorem~2.1]{CGKLM}. Their
ground-state formula expresses the log-Laplace limit as an infimum
over nondecreasing integrable functions bounded below by the Laplace
parameter \cite[Theorem~3.1]{CGKLM}; they also establish a one-sided
ground-state large-deviation principle
\cite[Theorems~1.2 and~5.1]{CGKLM}.

At zero external field and the quadratic SK covariance, these are the
constrained variational problems considered here. Our contribution is a quantitative comparison. 
Two regimes require particular care: the moment parameter may tend to
zero with the number of spins, and the inverse temperature may tend to
infinity. The first regime probes fluctuations near the expected free
energy, while the second replaces the partition function by the
ground-state energy.

To implement our comparison argument, we add an independent centered Gaussian
variable of variance $1/2$ to every spin energy, using the same
variable for all configurations. We call this the completed model and
write $F^c_{N,\beta}$ for its logarithmic partition function. For
$0<s\le1/2$, we consider the normalized logarithmic fractional moment
\[
 \frac{1}{Ns}\log \E e^{sF^c_{N,\beta}},
\]
where $s$ is the fractional-moment parameter.  For the completed model,
this quantity differs from its constrained Parisi minimum by at most
\[
 C\min\{\beta N^{-1/2},
          \beta^{2/3}N^{-2/3}s^{-1/3}\}.
\]
Here $C$ is absolute, and the estimate holds simultaneously for all
system sizes $N$, inverse temperatures $\beta>0$, and moment parameters
$0<s\le1/2$. Division by $\beta$, with $s=\lambda/\beta$, gives a
comparison for the ground-state log-Laplace transform at every fixed
$\lambda>0$. The interpolation error is estimated at the
finite-dimensional functional's own minimizing measure. This comparison
does not require information about the support of the scalar Parisi
measure.

Sharper low-temperature estimates depend on the structure of that
measure. Auffinger and Chen proved uniqueness of the Parisi minimizer
through strict convexity \cite{AuffingerChen2015Uniqueness}. They also
established structural properties of Parisi measures, including the
presence of zero in the zero-field support and regularity on intervals
contained in the support \cite{AuffingerChen2015Properties}.
Variational representations and optimality conditions were developed
by Chen \cite{Chen2017Variational}, while Jagannath and Tobasco gave a
dynamic-programming treatment of the Parisi functional
\cite{JagannathTobasco2016}. These analytic descriptions make it
possible to relate perturbations of the variational problem to
properties of its minimizing profile.

The geometry of the support is a sharper question than the failure
of replica symmetry. Toninelli established the failure of the
replica-symmetric formula below the Almeida--Thouless line
\cite{Toninelli2002AT}. Auffinger, Chen, and Zeng proved that the SK
model has infinitely many levels of replica symmetry breaking at zero
temperature \cite{AuffingerChenZeng2020}. Zhou proved interval support
for inverse temperatures sufficiently close to and greater than one
\cite{Zhou2026FRSB}. For every $\beta>1$, a recent preprint identifies the support of the scalar Parisi measure as
$[0,q_\beta]$, with $0<q_\beta<1$
\cite[Theorem~1.1]{Lopatto}.

The fractional-moment variational problem imposes an artificial mass
constraint at zero. Using the interval support of the unconstrained
Parisi measure, we show that requiring mass at least $s$ at zero
raises the scalar minimum by order $s^5$ as $s\downarrow0$ for each
fixed $\beta>1$. Equivalently, under the rescaling
$\lambda=\beta s$, the rescaled scalar cost has order $\lambda^5$.
This fifth-order variational increase agrees with the order predicted
by Kondor and Parisi--Rizzo \cite{Kondor1983,ParisiRizzo2008,PR} and
is the mechanism that converts the fractional-moment formula into
fluctuation bounds. 
A refined finite-size comparison then yields, for every fixed
$\beta>1$, a variance lower bound of order $N^{4/15}$ and a variance
upper bound of order $N^{7/15}$ for the logarithmic partition function.
The difference of its expected value per spin from the limiting pressure
is bounded below by order $N^{-11/12}$ and above by order $N^{-2/3}$.
For the variance upper bound, we combine the mean-difference estimate
with Chen's convexity theorem for Gaussian logarithmic moments
\cite{ChenConvexity}. We also obtain bounds on tilted moments and
variances.

The same comparison gives upper and lower estimates for upper-tail
probabilities with exponent $6/5$. The deviations are measured from
the thermodynamic value, and the estimates hold when the deviation
per spin ranges from a sufficiently large multiple of $N^{-2/3}$ to
a sufficiently small fixed constant. They therefore make the
six-fifths prediction quantitative on an explicit shrinking deviation
range. For deviations of $F_{N,\beta}/(N\beta)$ from $P_\beta/\beta$, these
estimates are uniform for $\beta\ge\beta_0>1$, with constants depending
only on $\beta_0$. 
Taking $\beta_0=2$ and passing to zero temperature gives the same
variance exponents, mean-deficit bounds, and upper-tail estimates for
the ground-state energy. 

\subsection{Main results}
Let $(g_{ij})_{1\le i<j\le N}$ be independent standard Gaussian variables.
For $\sigma\in\{-1,1\}^N$, define
\begin{equation}\label{eq:model}
 \begin{gathered}
 H_N(\sigma)=\frac1{\sqrt N}\sum_{i<j}g_{ij}\sigma_i\sigma_j,
 \qquad F_{N,\beta}=\log\sum_\sigma e^{\beta H_N(\sigma)},\\
 V_{N,\beta}=\Var F_{N,\beta},\qquad
 p_{N,\beta}=\frac{\E F_{N,\beta}}N.
 \end{gathered}
\end{equation}
The variance is over the Gaussian disorder, and all logarithms are
natural logarithms. When $\beta$ is fixed, we suppress it and write $F_N,V_N,p_N$.

For a probability measure $\mu$ on $[0,1]$, write $a(q)=\mu([0,q])$.
The scalar Parisi functional and its constrained minimum are
\begin{equation}\label{eq:scalar-functional}
 \begin{gathered}
 \partial_q u_a=-\frac{\beta^2}{2}
       (u_{a,xx}+a(q)u_{a,x}^2),\qquad u_a(1,x)=\log(2\cosh x),\\
 \cP_\beta(a)=u_a(0,0)-\frac{\beta^2}{2}\int_0^1q a(q)\dd q,
 \qquad
 P_\beta(s)=\min_{\mu(\{0\})\ge s}\cP_\beta(a),\quad
 P_\beta=P_\beta(0).
 \end{gathered}
\end{equation}
We likewise write $\cP,P(s),P$ at fixed $\beta$. Existence and
continuity follow from the finite-dimensional construction in
\zcref{sec:finite-n} at $t=0$. 

Let $W$ be a Gaussian orthogonal ensemble (GOE) matrix with off-diagonal
entries $g_{ij}/\sqrt N$ and independent diagonal entries
$\sqrt{2/N}z_i$, where the $z_i$ are standard Gaussian. Set
\[
 D_N=\frac12\Tr W,\qquad
 F^c_{N,\beta}=\log\sum_\sigma e^{\beta\sigma^TW\sigma/2}
             =F_{N,\beta}+\beta D_N.
\]
The variable $D_N$ is independent of $F_{N,\beta}$ and has law
$N(0,1/2)$. For $s>0$, define
\begin{equation}\label{eq:completion}
 \phi^c_{N,\beta}(s)=\frac1{Ns}\log\E e^{sF^c_{N,\beta}}
  =p_{N,\beta}+\frac{K_N(s)}{Ns}+\frac{\beta^2s}{4N},\qquad K_N(s)=\log\E e^{s(F_{N,\beta}-\E F_{N,\beta})}.
\end{equation}
At $s=0$ set $\phi^c_{N,\beta}(0)=p_{N,\beta}$, and at fixed $\beta$
write $F_N^c$ and $\phi_N^c(s)$. The completed Hamiltonian
$\sigma^TW\sigma/2$ has covariance $NR(\sigma,\tau)^2/2$, where
$R(\sigma,\tau)=N^{-1}\sigma\cdot\tau$.

\begin{theorem}\zlabel{thm:uniform-parisi}
There exists a constant $C>0$ such that, for every
$N\ge1$, $\beta>0$, and $0< s\le1/2$,
\begin{equation}\label{eq:uniform-parisi}
 0\le P_\beta(s)-\phi^c_{N,\beta}(s)
 \le C\min\{\beta^2,\beta N^{-1/2},
                 \beta^{2/3}N^{-2/3}s^{-1/3}\}.
\end{equation}
For the expected pressure,
\begin{equation}\label{eq:quantitative-pressure}
 0\le P_\beta-p_{N,\beta}\le C\beta N^{-1/2}.
\end{equation}
\end{theorem}
The uniformity in
$s$ also permits moment parameters that depend on $N$. The bound includes $C\beta N^{-1/2}$, so the error remains bounded after division by
$\beta$ and passage to zero temperature.

To state the zero-temperature result, let $\mathcal M_\lambda$ be the
nondecreasing, right-continuous, integrable functions
$\gamma:[0,1)\to[\lambda,\infty)$, for $\lambda\ge0$. Define
\begin{equation}\label{eq:zero-functional}
 \begin{gathered}
 \partial_q\Psi_\gamma=-\tfrac12
       (\Psi_{\gamma,xx}+\gamma(q)\Psi_{\gamma,x}^2),
 \qquad \Psi_\gamma(1,x)=|x|,\\
 \mathcal E(\lambda)=\inf_{\gamma\in\mathcal M_\lambda}
 \left[\Psi_\gamma(0,0)-\frac12\int_0^1q\gamma(q)\dd q\right].
 \end{gathered}
\end{equation}
For integrable, possibly unbounded $\gamma$, the solution is defined
by bounded approximation, equivalently by the control representation
in \zcref{subsec:sec-zero-temperature}. Put
\[
 G_N=\max_\sigma H_N(\sigma),\qquad G_N^c=G_N+D_N.
\]
\begin{theorem}\zlabel{thm:ground-state}
There exists a  constant $C<\infty$ such that, for every $N\ge1$,
\begin{equation}\label{eq:ground-mean}
 0\le\mathcal E(0)-\frac{\E G_N}{N}\le C N^{-1/2}.
\end{equation}
For every $\lambda>0$, the same constant satisfies
\begin{equation}\label{eq:ground-laplace}
 0\le\mathcal E(\lambda)
       -\frac1{N\lambda}\log\E e^{\lambda G_N^c}
 \le C\min\{N^{-1/2},N^{-2/3}\lambda^{-1/3}\}.
\end{equation}
In the original SK model, the exact correction is
\begin{equation}\label{eq:ground-offdiagonal}
 \frac{\lambda}{4N}\le\mathcal E(\lambda)
       -\frac1{N\lambda}\log\E e^{\lambda G_N}
 \le C\min\{N^{-1/2},N^{-2/3}\lambda^{-1/3}\}
                  +\frac{\lambda}{4N}.
\end{equation}
\end{theorem}
Here $\mathcal E(0)$ is the thermodynamic ground-state value. The
log-Laplace estimate gives a finite-size comparison with the limiting
functional for every positive Laplace parameter. Its uniformity in
$\lambda$ concerns the shifted ground state $G^c_N$. The correction in \eqref{eq:ground-offdiagonal} cannot
be dropped when $\lambda$ grows with $N$.

The remaining fluctuation and tail bounds use the interval-support
property of the Parisi measure from \cite[Theorem~1.1]{Lopatto}. 
\begin{theorem}\zlabel{thm:ground-fluctuations}
Put $e_*=\mathcal E(0)$. There exist  constants $c,C>0$
and $N_0<\infty$ such that, for every integer $N\ge N_0$,
\begin{equation}\label{eq:ground-four-bounds}
 \begin{gathered}
 cN^{4/15}\le\Var G_N\le CN^{7/15},\\
 cN^{-11/12}\le e_*-\frac{\E G_N}{N}\le CN^{-2/3}.
 \end{gathered}
\end{equation}
\end{theorem}
The variance is over all off-diagonal Gaussian couplings. The proof,
in \zcref{sec:ground-fluctuations}, first establishes estimates
uniformly for $\beta\ge2$ and then takes $\beta\to\infty$ at fixed
$N$. 

\begin{theorem}\zlabel{thm:main}
For every fixed $\beta>1$, there exist $c_\beta,C_\beta>0$ and an integer $N_0(\beta)\ge 1$ such that, for every $N\ge N_0(\beta)$,
\begin{equation}\label{eq:fixed-beta-four-bounds}
 \begin{gathered}
 c_\beta N^{4/15}\le V_{N,\beta}\le C_\beta N^{7/15},\\
 c_\beta N^{-11/12}\le P_\beta-p_{N,\beta}
                       \le C_\beta N^{-2/3}.
 \end{gathered}
\end{equation}
\end{theorem}
The scalar gap and finite-size comparison also
control rare upper deviations without the variance lower bound as an input.

\begin{theorem}\zlabel{thm:upper-tails}
For every fixed $\beta>1$, there exist constants
$A_\beta,x_\beta,c_\beta,C_\beta>0$ and $N_0(\beta)<\infty$ such that
\begin{equation}\label{eq:six-fifths-tail}
 e^{-C_\beta Nx^{6/5}}
 \le\Prob\{F_{N,\beta}\ge N(P_\beta+x)\}
 \le e^{-c_\beta Nx^{6/5}}
\end{equation}
whenever $N\ge N_0(\beta)$ and
$A_\beta N^{-2/3}\le x\le x_\beta$.
\end{theorem}
These are rare upper tails centered at $NP_\beta$, not at
$\E F_{N,\beta}$. The lower end of the range corresponds to total
upward deviations of order $N^{1/3}$ and probabilities
$\exp(-\Theta_\beta(N^{1/5}))$. 
A similar argument gives the following uniform
bound and its zero-temperature counterpart.

\begin{theorem}\zlabel{thm:uniform-upper-tails}
There exist numerical constants $A,x_*,c,C>0$ and $N_0<\infty$
such that, for every $N\ge N_0$, $\beta\ge2$, and
$AN^{-2/3}\le x\le x_*$,
\begin{equation}\label{eq:uniform-six-fifths-tail}
 e^{-CNx^{6/5}}
 \le\Prob\left\{\frac{F_{N,\beta}}\beta
                  \ge N\left(\frac{P_\beta}\beta+x\right)\right\}
 \le e^{-cNx^{6/5}}.
\end{equation}
The same constants give
\begin{equation}\label{eq:ground-six-fifths-tail}
 e^{-CNx^{6/5}}
 \le\Prob\{G_N\ge N(e_*+x)\}
 \le e^{-cNx^{6/5}}
\end{equation}
for the same $N$ and $x$.
\end{theorem}
We also prove bounds for the tilted mean and variance, concentration
at the finite-system mean, and quantitative comparison for negative
moment parameters; see \zcref{cor:fixed-beta-concentration,cor:fixed-beta-tilted,cor:uniform-energy-tilted,cor:uniform-energy-concentration}.

Throughout, $C$ denotes a numerical constant, while $C_\beta$ denotes a
finite constant depending only on $\beta$. In the uniform energy
argument we fix $\beta_0>1$ and allow constants to depend only on
that threshold. Their values may increase from one occurrence to the next. Unless explicitly stated otherwise,
these constants are independent of $N$, the interpolation parameter,
and the tilt in the stated range.

\subsection{Ideas of the proof}
The proof connects positive fractional moments to fluctuations through the
constrained Parisi problem. The fractional-moment parameter $s$ requires
the Parisi measure to place mass at least $s$ at zero. After rescaling
the cumulative Parisi profile by $\gamma=\beta a$ and setting
$\lambda=\beta s$, this becomes the lower bound
$\gamma\ge\lambda$. Using interval support of the unconstrained Parisi
measure, we show that imposing this constraint raises the scalar minimum
by order $\lambda^5$. To turn this variational increase into a
finite-$N$ fluctuation estimate, we need a comparison between the
finite-system logarithmic moment and the constrained Parisi value that
is accurate on the same scale. We take
$\lambda\asymp N^{-2/15}$, so that
$\lambda^5\asymp N^{-2/3}$. At this choice of $\lambda$, the remaining
terms in the finite-size comparison are smaller than order $N^{-2/3}$.
Gaussian moment inequalities then convert the resulting logarithmic
moment bounds into the variance and mean-deficit estimates. All of
these estimates are uniform for $\beta\ge\beta_0>1$, which allows us
to pass to the ground-state limit.

The finite-size comparison uses the Guerra interpolation employed by
Talagrand for positive fractional moments \cite{TalagrandLD}. Talagrand's
interpolation already identifies the corresponding limiting constrained
Parisi problem. Our additional task is quantitative: at each interpolation
parameter we minimize the finite-$N$ variational functional and use the
resulting optimality conditions to bound the overlap error explicitly,
with estimates that remain useful as $s\downarrow0$ and, after the
rescaling above, as $\beta\to\infty$.

We use an interpolation parameter $t\in[0,1]$. The parameter $t$
controls the strength of the SK interaction: the Gaussian coupling
matrix in the completed Hamiltonian is multiplied by $\sqrt t$, while
an auxiliary Gaussian field supplies the complementary variance
\[
 c=\beta^2(1-t).
\]
At $t=0$, the matrix interaction is absent and the variational problem
is the constrained scalar Parisi problem. At $t=1$, the auxiliary field
is absent and the interpolating value is the completed finite-$N$
fractional moment. For each $t$, we minimize the corresponding
finite-$N$ variational functional over measures satisfying the mass
constraint at zero. Differentiating the optimized value in $t$
expresses the difference between the two endpoints as an integral of
mean-square overlap deviations.

Write a minimizing measure as
\[
 \mu=s\delta_0+\nu,
 \qquad
 a(q)=\mu([0,q])=s+\nu([0,q]),
\]
where $q\in[0,1]$ is the overlap parameter and $\nu$ contains the
remaining mass; it may also place additional mass at zero. The
comparison must remain useful as $s\downarrow0$. The posterior variance
estimate used below contains the factor $a(q)^{-1}$. If $\nu$
accumulates little mass near zero, then $a(q)$ remains small over an
initial range of overlap values, and this estimate becomes weak there.

There is a second difficulty. The posterior estimate controls the
variance of an average over $\nu$, whereas the interpolation error
contains an average of the variances at individual overlap values.
We address both problems by choosing averaging intervals through the
quantiles of $\nu$. This fixes the amount of $\nu$-mass in each
interval, so the quantile level supplies a lower bound on $a(q)$.
Comparing conditional spin means at nearby overlap values then converts
the variance bound for an average into the pointwise variance bounds
needed in the interpolation.

\medskip
\emph{Step 1: Contact conditions for the interpolating minimizer.}
Fix $0<s\le1/2$ and $0\le t<1$, and set
\[
 c=\beta^2(1-t)>0.
\]
Let $\mu=s\delta_0+\nu$ be a minimizing measure for the finite-$N$
variational functional at interpolation parameter $t$, and continue
to write $a(q)=\mu([0,q])$. 

For this fixed $t$ and $\mu$, the variational recursion involves the
Gaussian SK coupling matrix and an $N$-dimensional auxiliary Gaussian
field. Conditional on the coupling matrix and the full auxiliary field,
sample $
 \sigma\in\{-1,1\}^N$ 
from the Gibbs weights appearing in the terminal log-sum-exp of the
recursion. We call this random configuration the terminal spin.

Let $X_q\in\mathbb R^N$ denote the auxiliary field revealed at overlap
level $q\in[0,1]$, and let $\mathcal G_q$ be the sigma-field generated
by the coupling matrix and the field history $(X_r)_{0\le r\le q}$.
The terminal spin itself is not included in $\mathcal G_q$. Define
\[
 m_q=\E[\sigma\mid\mathcal G_q],
 \qquad
 C_q=\Cov(\sigma\mid\mathcal G_q).
\]
Thus $m_q$ is the conditional mean spin, and
$(m_q)_{0\le q\le1}$ is a martingale.

Let $u_{t,\mu}(q,x)$ denote the $N$-dimensional Parisi solution for
this interpolating problem, and set
\[
 H_q=D_x^2u_{t,\mu}(q,X_q),
\]
where $D_x^2$ is the Hessian with respect to $x$. 
The recursion gives
\[
 a(q)C_q\preceq H_q\preceq C_q.
\]
Thus the Hessian of the Parisi solution controls the conditional spin
covariance, with the lower comparison deteriorating only through the
factor $a(q)$. This is the source of the factor $a(q)^{-1}$ that
appears in the posterior variance estimate of Step~2.

We next use optimality of $\mu$. The first variation with respect to
the remaining measure $\nu$ gives a contact condition on
$\supp\nu$. At every $q\in\supp\nu$, this condition and its
second-order consequence yield
\[
 \E|m_q|^2=Nq,\qquad
 \E\Tr H_q^2\le \frac{N}{c}.
\]
The first identity says that, on the support of $\nu$, the expected
squared norm of the conditional mean spin is exactly the overlap scale
$Nq$. The second provides a uniform bound on the Hessian at those same
overlap levels. Together with
$a(q)C_q\preceq H_q$, these estimates give the covariance control used
in the quantitative interpolation. These are finite-$N$ properties of the interpolating minimizer; see
\zcref{sec:finite-n}.

\medskip
\emph{Step 2: From posterior concentration to pointwise variance bounds.}
We now derive the averaged posterior variance estimate and convert it into the pointwise bounds required for the interpolation error. Fix the terminal spin $\sigma$ and condition on its value. Under this
conditional law, the Gaussian SK couplings and auxiliary Gaussian
fields from Step~1 have a density, relative to Lebesgue measure in
the underlying Gaussian coordinates, whose negative logarithm is
uniformly convex. The Hessian of this
negative log density contains an additional positive term obtained by
integrating the Hessians of the finite-$N$ Parisi recursion against
$\nu$. To exploit this curvature, define
\[
 L_q=\frac{\sigma\cdot m_q-|m_q|^2/2}{N}.
\]
The subtraction of $|m_q|^2/2$ is chosen so that differentiation of
$L_q$ with respect to the underlying Gaussian coordinates produces
the residual $\sigma-m_q$. Conditional on $\mathcal G_q$, this
residual has mean zero and covariance
\[
 C_q=\Cov(\sigma\mid\mathcal G_q).
\]
Applying the Brascamp--Lieb variance inequality
\cite{BrascampLieb1976}  under the fixed-spin conditional law, followed
by matrix Cauchy--Schwarz and the comparison
$a(q)C_q\preceq H_q$ from Step~1, gives
\[
 \Var\left(\int k(q)L_q\,\nu(\dd q)\right)
 \le\frac1{Nc}\int\frac{k(q)^2}{a(q)}\,\nu(\dd q)
\]
for every deterministic $k\in L^2(\nu)$.

The displayed inequality also holds under the full joint law from
Step~1. Indeed, gauge symmetry identifies the conditional law of the
observable $\int k(q)L_q\,\nu(\dd q)$ for every fixed value of
$\sigma$. Consequently its conditional mean and variance do not
depend on the chosen terminal spin, so averaging over $\sigma$ adds
no extra variance.

We must convert this estimate for averages of $L_q$ into a bound on
\[
 \int\Var(L_q)\,\nu(\dd q).
\]
For $q<r$, the martingale property of $m_q$ gives
\[
 L_r-L_q
 =
 \frac{(\sigma-m_r)\cdot(m_r-m_q)}{N}
 +\frac{|m_r-m_q|^2}{2N}.
\]
Conditional on $\mathcal G_r$, the first term has mean zero, while the
second is nonnegative. This allows us to bound $L_q$ by averages of
$L_r$ over overlap levels on either side of $q$. The centered term is
then controlled by the averaged posterior variance estimate above.

We choose the averaging intervals by quantiles of $\nu$ so that their
$\nu$-mass is fixed in advance. Let
\[
 Q(v)=\inf\{q\in[0,1]:\nu([0,q])\ge v\},
 \qquad 0<v<1-s,
\]
be the quantile function of $\nu$. Let $I$ be an interval in the
quantile variable with length at most $u$, and suppose
$a(Q(v))\ge u$ for $v\in I$. If
\[
 \ell=\operatorname*{ess\,sup}_{v,w\in I}|Q(v)-Q(w)|,
\]
then the interval estimate proved in \zcref{sec:residual} gives
\[
 \int_I \Var(L_{Q(v)})\,\dd v
 \le
 C\left[(Nc)^{-3/4}u^{-1/2}
       +(Nc)^{-2/3}u^{-1/3}\ell^{2/3}
       +(Ncu)^{-1}\right].
\]
The parameter $u$ controls the amount of mass available for averaging,
while $\ell$ measures how far the corresponding overlap values are
separated. Quantiles also allow atoms of $\nu$ to be divided without
changing the prescribed mass of an interval. Summing these estimates
over quantile intervals gives the bound needed for the interpolation
error.

\medskip
\emph{Step 3: The finite-size comparison and its zero-temperature limit.}
We now sum the quantile estimates from Step~2. The only difficulty is
the region where the cumulative mass $a(q)$ is very small. Suppose
first that $Nc>1$, and set
\[
 u_0=(Nc)^{-1/2}.
\]
We separate an initial quantile interval having $\nu$-mass at most
$u_0$. On this interval we use only the elementary bound
$\Var(L_q)\le1$, so its contribution is at most $u_0$.

On the remaining mass, divide the quantile variable into dyadic
intervals whose lower cumulative-mass levels are
$u_j=2^ju_0$. Then $a(q)\ge u_j$ on the interval at level $u_j$.
Using the estimate from Step~2 with the trivial diameter bound
$\ell\le1$ gives the three contributions
\[
 (Nc)^{-3/4}u_j^{-1/2},\qquad
 (Nc)^{-2/3}u_j^{-1/3},\qquad
 (Ncu_j)^{-1}.
\]
At the smallest level $u_0=(Nc)^{-1/2}$, each is at most a constant
times $(Nc)^{-1/2}$:
\[
 \begin{aligned}
 (Nc)^{-3/4}u_0^{-1/2}&=(Nc)^{-1/2},\\
 (Nc)^{-2/3}u_0^{-1/3}&=(Nc)^{-1/2},\\
 (Ncu_0)^{-1}&=(Nc)^{-1/2}.
 \end{aligned}
\]
At successive dyadic levels these terms decrease geometrically, so
their sum has the same order and no logarithmic factor appears. When
$Nc\le1$, the trivial overlap bound suffices. The trace contribution
to the interpolation error obeys the same estimate.

To state the resulting bound, write
\[
 R(\sigma^1,\sigma^2)=\frac1N\sigma^1\cdot\sigma^2
\]
for the overlap. For $c=\beta^2(1-t)$, let
\[
 D(c)=\int
 \E\left[(R(\sigma^1,\sigma^2)-q)^2\right]\,\nu(\dd q),
\]
where $\mu=s\delta_0+\nu$ minimizes the interpolating functional at
$t=1-c/\beta^2$, and, conditional on the disorder and auxiliary field
revealed through overlap level $q$, the spins $\sigma^1,\sigma^2$ are
independent samples from the same conditional Gibbs law. The preceding
estimates give
\[
 D(c)\le C\min\{1,(Nc)^{-1/2}\}.
\]
The derivative formula for the optimized interpolation, followed by
the change of variables $c=\beta^2(1-t)$, yields
\[
 P_\beta(s)-\phi^c_{N,\beta}(s)
 =\frac14\int_0^{\beta^2}D(c)\,\dd c.
\]
Since $c^{-1/2}$ is integrable at zero,
\[
 \int_0^{\beta^2}\min\{1,(Nc)^{-1/2}\}\,\dd c
 \le C\min\{\beta^2,\beta N^{-1/2}\}.
\]
This gives the temperature-uniform comparison in
\zcref{thm:uniform-parisi}; the sharper dependence on $s$ is obtained
by starting the dyadic decomposition at cumulative mass $s$ and
combining the resulting estimate with the moving-cutoff bound.

To pass to zero temperature, fix a Laplace parameter $\lambda>0$ and
set
\[
 s=\lambda/\beta,\qquad \gamma=\beta a.
\]
Then the constraint $a\ge s$ becomes $\gamma\ge\lambda$. The
rescaled terminal function satisfies
\[
 \sup_{x\in\mathbb R}
 \left|
 \frac1\beta\log(2\cosh(\beta x))-|x|
 \right|
 \le\frac{\log2}{\beta},
\]
and the rescaled positive-temperature variational problems converge
to the corresponding zero-temperature problem. On the finite-system
side,
\[
 \frac1\beta\phi^c_{N,\beta}(\lambda/\beta)
 =
 \frac1{N\lambda}
 \log\E\exp\left\{\lambda\frac{F^c_{N,\beta}}{\beta}\right\},
\]
and $F^c_{N,\beta}/\beta$ converges, for fixed $N$, to the completed
ground-state energy. Dividing the finite-temperature comparison by
$\beta$ and letting $\beta\to\infty$ therefore gives the
ground-state log-Laplace comparison.  See \zcref{subsec:sec-zero-temperature} for details.

\medskip
\emph{Step 4: The cost of imposing the mass constraint at zero.}
Fix $\beta_0>1$ and consider $\beta\ge\beta_0$. Rescale the scalar
Parisi profile by
\[
 \gamma=\beta a,\qquad \lambda=\beta s,
\]
and write
\[
 E_\beta(\lambda)=\beta^{-1}P_\beta(\lambda/\beta).
\]
Let $\gamma_\beta$ be an unconstrained minimizer of the rescaled scalar
Parisi functional, whose minimum value is $E_\beta(0)$. In these
variables, requiring the original Parisi measure to have mass at least
$s$ at zero is equivalent to requiring
\[
 \gamma(q)\ge\lambda,\qquad 0\le q<1.
\]
Thus $E_\beta(\lambda)-E_\beta(0)$ is the increase in the scalar
minimum caused by this artificial constraint.

The interval-support theorem gives
$\supp(\beta^{-1}\dd\gamma_\beta)=[0,q_\beta]$. Let $u_\beta$ denote
the rescaled scalar Parisi PDE solution associated with
$\gamma_\beta$, with terminal condition
\[
 u_\beta(1,x)=\frac1\beta\log(2\cosh(\beta x)).
\]
Let $Y_q^\beta$ denote the corresponding one-dimensional Parisi
diffusion, and write $\E_\beta$ for expectation under this diffusion.
On $[0,q_\beta]$, the first-variation contact condition is
\[
 S(q):=\E_\beta u_{\beta,x}(q,Y_q^\beta)^2=q.
\]
If
\[
 b(q)=\E_\beta u_{\beta,xx}(q,Y_q^\beta),
\]
then the scalar diffusion identities give
\[
 b'=-\gamma_\beta S',\qquad
 b(0)=1,\qquad
 b(1)=\beta(1-S(1)).
\]
Since $S'=1$ on $[0,q_\beta]$ and
$\gamma_\beta=\beta$ on $(q_\beta,1]$, integration yields
\[
 \int_0^1\gamma_\beta(q)\,\dd q=1.
\]
Consequently
\[
 \beta(1-q_\beta)\le1,
 \qquad
 q_\beta\ge1-\beta^{-1}\ge1-\beta_0^{-1}>0.
\]
Moreover, on an interval $[0,T_*]$ whose length depends only on
$\beta_0$, parabolic smoothing and a boundary comparison give
\[
 kq\le\gamma_\beta(q)\le Cq,
 \qquad 0\le q\le T_*,
\]
with $k,C>0$ depending only on $\beta_0$.

These linear bounds already identify the scale of the perturbation.
For the upper bound, use the admissible trial profile
\[
 \gamma^{\mathrm{tr}}(q)=\max\{\lambda,\gamma_\beta(q)\}.
\]
Because $\gamma_\beta(q)\ge kq$, this profile differs from
$\gamma_\beta$ only for $q\le C\lambda$, and the pointwise difference
is at most $\lambda$. On the support interval, the first-variation
density of the unconstrained minimizer is constant, so the linear
term in this perturbation vanishes. The second variation along the
segment joining $\gamma_\beta$ to $\gamma^{\mathrm{tr}}$ is bounded
above by
\[
 C\lambda^2T^3
\]
when the perturbation is supported on an interval of length $T$.
Taking $T\asymp\lambda$ therefore gives
\[
 E_\beta(\lambda)-E_\beta(0)\le C\lambda^5.
\]

For the lower bound, the upper estimate
$\gamma_\beta(q)\le Cq$ implies that every admissible profile
$\gamma\ge\lambda$ satisfies
\[
 \gamma(q)-\gamma_\beta(q)\ge\lambda/2
\]
throughout an interval $[0,T]$ with $T\asymp\lambda$. It is enough to
consider competitors whose scalar value is at most one above
$E_\beta(0)$, since all others already satisfy the desired lower
bound for sufficiently small $\lambda$. For such competitors, the
control representation gives a uniform bound on the profile on a
fixed initial interval.

Along the segment from $\gamma_\beta$ to a competitor $\gamma$, we
test the linearized scalar PDE with a function supported in
$[0,T]$. The perturbation has size at least $\lambda/2$ there, while
the response estimate in \zcref{app:response-estimates} gives a
positive lower bound, uniform in $\beta\ge\beta_0$, for the averaged
spatial curvature under the normalized law associated with each
profile on the segment. These estimates yield the lower second-variation bound
\[
 c_0\lambda^2T^3,
\]
where $c_0>0$ depends only on $\beta_0$. Since
$T\asymp\lambda$, we obtain
\begin{equation}\label{eq:gap-outline}
 c_1\lambda^5
 \le E_\beta(\lambda)-E_\beta(0)
 \le C_1\lambda^5.
\end{equation}
The constants $c_1,C_1>0$ and the range of admissible $\lambda$
depend only on $\beta_0$; see
\zcref{lem:gs-scalar-root,prop:scalar-lower-via-fisher}.

The curvature estimate needed for the lower bound must hold under
each profile's own normalized law, not merely under the law of the
unconstrained minimizer. In \zcref{app:response-estimates}, we prove
this directly by differentiating the finite-spin recursion with
respect to the terminal spin energies and then applying a conditional
Gaussian estimate. This avoids any comparison between the normalized
densities associated with different profiles.

\medskip
\emph{Step 5: Refining the comparison by local regularization.}
At the fluctuation scale $\lambda\asymp N^{-2/15}$, the scalar
increase from Step~4 has size
\[
 \lambda^5\asymp N^{-2/3}.
\]
The comparison from Step~3 is not accurate enough by itself to
distinguish a contribution on this scale. We therefore return to the
quantile estimate of Step~2 and retain the diameter of each quantile
interval instead of replacing it by the trivial bound $\ell\le1$.

Fix an interpolation parameter $r=1-t>0$. After the rescaling
$\gamma=\beta a$, write a minimizing profile as
\[
 \dd\gamma=\lambda\delta_0+\rho,
 \qquad
 \rho([0,1])=\beta-\lambda.
\]
Here $\rho$ is the rescaled measure remaining after the prescribed
mass $\lambda$ at zero is separated. To control the location of a low
quantile of $\rho$, we regularize only the mass below that quantile.
For each location $q$ in this part of the measure, replace its mass by
the average of its images under
\[
 q\mapsto |q+U|,
 \qquad U\sim\operatorname{Unif}[-h,h].
\]
This defines a deterministic smoothed measure. The prescribed atom at
zero and all mass above the chosen quantile are left unchanged.

If $m$ is the amount of mass moved, the smoothing raises the
interpolating variational value by at most a constant times $rmh^2$.
The smoothed part of the measure has density at most $m/h$ on the
modified interval.

These two properties allow us to use the response estimate from
Step~4. Comparing the smoothed profile with the scalar trial
$\max\{\lambda,\gamma_\beta\}$ bounds the overlap location of the
quantile in terms of the amount by which the scalar trial lies above
the interpolating minimum. Thus small variational excess forces low
quantiles of $\rho$ to lie close to zero, which in
turn makes the diameter term in the covariance estimate small.

To close this estimate, define the completed comparison error by
\[
 e_{N,\beta}(\lambda)
 =E_\beta(\lambda)-\frac1{N\lambda}
   \log\E e^{\lambda(F_{N,\beta}/\beta+D_N)}.
\]
The scalar trial from Step~4 lies at most
$C\lambda^5$ above $E_\beta(\lambda)$, so its excess above the
interpolating minimum is bounded by
\[
 \eta=e_{N,\beta}(\lambda)+C\lambda^5.
\]
Substituting the resulting quantile bounds into the interpolation
error and integrating in $r$ produces terms involving fractional
powers of $\eta$, in addition to explicit finite-size terms. Since
all of these powers are strictly smaller than one, Young's inequality
bounds the terms containing $e_{N,\beta}(\lambda)$ by a fixed fraction
of the left-hand side plus explicit functions of $N$ and $\lambda$.
Moving that fraction to the left gives
\begin{equation}\label{eq:error-outline}
 0\le e_{N,\beta}(\lambda)
 \le C\left[N^{-2/3}+N^{-3/4}\lambda^{-1/2}
                    +\frac{\log(2+N)}{N\lambda}\right].
\end{equation}
The quantile decomposition is performed in the cumulative-mass
variable, so the constants in the sums do not grow with the total mass
$\beta$. The estimate is therefore uniform for
$\beta\ge\beta_0>1$; see
\zcref{lem:p2-quantile,prop:abs-selfconsistent-comparison}.

\medskip
\emph{Step 6: From logarithmic moments to fluctuations and tails.}
We now convert the variational estimates into probabilistic bounds for
the original off-diagonal model. Set
\[
 X_{N,\beta}=\frac{F_{N,\beta}}{\beta},
\]
and define its centered logarithmic moment generating function by
\[
 K_{N,\beta}(\lambda)
 =\log\E\exp\{\lambda(X_{N,\beta}-\E X_{N,\beta})\}
 =K_N(\lambda/\beta).
\]

We first need an upper bound on this logarithmic moment in terms of the
variance. Gauge symmetry of the SK Hamiltonian allows us to change the
Gaussian disorder measure without changing the distribution of the
logarithmic partition function. Under the resulting measure, adjoining
one independent exponential random variable produces a one-dimensional
log-concave density. The corresponding tail estimate implies that there
are numerical constants $B,b>0$ such that
\begin{equation}\label{eq:mgf-outline}
 K_{N,\beta}(\lambda)
 \le
 B\lambda^2\bigl(\Var X_{N,\beta}+\beta^{-2}\bigr)
\end{equation}
whenever
\[
 0\le
 \lambda\sqrt{\Var X_{N,\beta}+\beta^{-2}}
 \le b.
\]
Thus, on its natural small-tilt range, a large centered logarithmic
moment forces a large variance; see \zcref{sec:moments}.

We apply this at
\[
 \lambda=A N^{-2/15},
\]
where $A$ is a sufficiently large constant depending only on
$\beta_0$. By Step~4, the constrained scalar minimum exceeds the
unconstrained minimum by at least $c_1\lambda^5$.  Multiplying by the factor $N\lambda$ appearing
in the logarithmic moment gives a contribution of order
\[
 N\lambda^6
 =A^6N^{1-12/15}
 =A^6N^{1/5}.
\]
The refined comparison from Step~5 has smaller size at this tilt once
$A$ is chosen sufficiently large. In particular, for all sufficiently
large $N$,
\[
 K_{N,\beta}(\lambda)\ge1.
\]

There are now two cases. If
\[
 \lambda\sqrt{\Var X_{N,\beta}+\beta^{-2}}>b,
\]
then
\[
 \Var X_{N,\beta}+\beta^{-2}
 >b^2\lambda^{-2}
 =b^2A^{-2}N^{4/15}.
\]
Otherwise the small-tilt condition in \eqref{eq:mgf-outline} holds,
and hence
\[
 1
 \le K_{N,\beta}(\lambda)
 \le B\lambda^2
       \bigl(\Var X_{N,\beta}+\beta^{-2}\bigr).
\]
This again gives
\[
 \Var X_{N,\beta}+\beta^{-2}\ge c_2N^{4/15}.
\]
Since $\beta\ge\beta_0>1$, the term $\beta^{-2}$ is uniformly bounded.
After increasing the threshold in $N$, we conclude that
\[
 \Var X_{N,\beta}\ge c_3N^{4/15}.
\]

We next control the mean. Let
\[
 M_{N,\beta}
 =E_\beta(0)-\frac{\E X_{N,\beta}}{N}
\]
be the finite-size deficit from the limiting energy density. Convexity
of the normalized Gaussian logarithmic moment, applied at two positive
tilts, together with the scalar upper bound from Step~4 and the
comparison from Step~5, gives
\[
 M_{N,\beta}\le CN^{-2/3}.
\]
A second use of convexity, now comparing the tangent at zero with the
scalar upper cost, gives
\[
 M_{N,\beta}
 \ge
 c_3\left(
 \frac{\Var X_{N,\beta}+1/2}{N}
 \right)^{5/4}.
\]
The term $1/2$ is the variance of the independent Gaussian diagonal
completion. Combining the last two displays gives
\[
 \Var X_{N,\beta}\le CN^{7/15}.
\]
Combining the same lower bound on $M_{N,\beta}$ with
$\Var X_{N,\beta}\ge c_3N^{4/15}$ gives
\[
 M_{N,\beta}\ge c_4N^{-11/12}.
\]

The same moment comparison gives the upper-tail exponent. For Laplace
parameters satisfying
\[
 A N^{-2/15}\le\lambda\le\lambda_0,
\]
with $A$ sufficiently large and $\lambda_0>0$ sufficiently small, we
obtain two-sided bounds of the form
\[
 c_5N\lambda^6
 \le
 \log\E
 \exp\{\lambda(X_{N,\beta}-NE_\beta(0))\}
 \le
 CN\lambda^6.
\]
The sixth power comes from the fifth-order scalar increase multiplied
by the prefactor $N\lambda$ in the logarithmic moment.

For an upward deviation of size $Nx$, choose
\[
 \lambda\asymp x^{1/5}.
\]
Then
\[
 N\lambda x\asymp N\lambda^6\asymp Nx^{6/5}.
\]
Chernoff's inequality therefore gives the upper-tail bound with exponent
$6/5$, while Paley--Zygmund applied to the exponential tilt at
$\lambda$ and $2\lambda$ gives the matching lower bound. The lower
restriction $\lambda\gtrsim N^{-2/15}$ becomes
\[
 x\gtrsim N^{-2/3}.
\]
This proves the two-sided tail estimates on the shrinking deviation
range described in \zcref{sec:tails}.

For a fixed inverse temperature $\beta>1$, the corresponding bounds for
$F_{N,\beta}$ follow from $F_{N,\beta}=\beta X_{N,\beta}$.
For the ground state, fix $\beta_0=2$ and let $\beta\to\infty$ at fixed
$N$. The soft maximum and the scalar variational values converge to
their zero-temperature limits, so the same variance, mean-deficit,
and tail estimates pass to the ground state.

\subsection{Acknowledgments}
P.L. was partially supported by NSF grant DMS-2450004. This paper was written
by the authors with the assistance of large language models, which included
suggesting arguments, contributing to drafting and revision, and writing code
for computational checks.

\section{Finite-dimensional interpolation and covariance estimates}
\zlabel{sec:finite-n}

This section supplies three inputs for the quantitative comparison
between $P_\beta(s)$ and $\phi^c_{N,\beta}(s)$.
\zcref{prop:fn-contact} gives contact identities and covariance
bounds at a minimizing measure.
\zcref{prop:fn-integrated-variance} controls the variance of an
average over overlap levels.
\zcref{prop:fn-envelope} expresses the difference between the
variational minimum and the finite-$N$ normalized logarithmic
moment as an integral of mean-square overlap deviations.

We first derive the fixed-profile identities for step profiles.
\zcref{subsec:fn-approximation} extends them to arbitrary admissible
measures, after which the limiting first variation is applied directly
to a minimizer.

Fix $\beta>0$, $N\ge1$, $0<s\le1/2$, and use the completed GOE
matrix $W$ and pressure $\phi_N^c(s)$ from \zcref{sec:main}.

For a probability measure $\mu$ on $[0,1]$, write
$a(q)=\mu([0,q])$. For $0\le t\le1$, set
$c=\beta^2(1-t)$ and define
\begin{align*}
 U_\mu(1,x;W)&=\log\sum_\sigma
  \exp\left(\frac{\beta\sqrt t}{2}\sigma^TW\sigma+x\cdot\sigma\right),\\
 \partial_q U_\mu&=-\frac c2
   \left(\Delta U_\mu+a(q)|\nabla U_\mu|^2\right),\\
 \mathcal F_{t,s}(\mu)&=\frac1{Ns}\log\mathbb E_W
  e^{sU_\mu(0,0;W)}-\frac c2\int_0^1qa(q)\dd q .
\end{align*}
We minimize only over deterministic measures of the form
\[
 \mu=s\delta_0+\nu,\qquad
 \nu\ge0,\qquad
 \nu([0,1])=1-s.
\]
Let
\[
 f(t)=\min_\mu\mathcal F_{t,s}(\mu).
\]
At $t=0$, the matrix interaction vanishes and the vector equation
separates into scalar equations, giving $f(0)=P(s)$. At $t=1$, the
auxiliary Gaussian field vanishes and $f(1)=\phi_N^c(s)$. The
interpolation identity and the passage from a fixed measure to the
minimized value are proved below.

\subsection{The normalized joint law and its posterior Hessian}

We first construct the probability law under which the conditional spin
means and covariances will be estimated. For the moment, assume that the
cumulative profile $a$ is a step function:
\[
 0=q_0<q_1<\cdots<q_L=1,
 \qquad
 a(q)=a_j\quad\text{for }q\in[q_j,q_{j+1}).
\]
Let $z_1,\ldots,z_L$ be independent standard Gaussian vectors in
$\mathbb R^N$, and define the auxiliary field at the grid points by
\[
 X_{q_j}
 =
 x_0+\sum_{k=1}^j\sqrt{c(q_k-q_{k-1})}\,z_k.
\]
Write
\[
 U_j=U_\mu(q_j,X_{q_j};W).
\]
If $\gamma_N$ denotes standard Gaussian measure on $\mathbb R^N$, the
Parisi recursion on the interval $[q_j,q_{j+1}]$ is
\[
 U_j
 =
 \frac1{a_j}
 \log\E_{z_{j+1}}e^{a_jU_{j+1}}.
\]
Consequently, conditional on the variables already revealed through
$q_j$, the normalized law of $z_{j+1}$ has density
\[
 e^{a_j(U_{j+1}-U_j)}
\]
with respect to $\gamma_N$. When $a_j=0$, the recursion is ordinary
Gaussian expectation and this conditional law is simply $\gamma_N$.

The factor $e^{sU_0}$ in the finite-$N$ fractional moment tilts the
Gaussian law of the coupling matrix. Accordingly, sample $W$ with
density
\[
 \frac{e^{sU_0}}{\E_W e^{sU_0}}
\]
relative to its original Gaussian law. Conditional on $W$, sample the
successive auxiliary Gaussian field increments using the normalized
transition laws above. After the full auxiliary field has been
revealed, sample
\[
 \sigma\in\{-1,1\}^N
\]
from the terminal Gibbs weights appearing in $U_L$. Denote the
resulting joint probability law by $\mathbb P$ and expectation under
this law by $\mathbb E$.

For $q\in[0,1]$, let $\mathcal G_q$ be the sigma-field generated by
$W$ and the auxiliary field revealed through overlap level $q$. The
terminal spin $\sigma$ is not included in $\mathcal G_q$.

To apply Brascamp--Lieb, we need a common set of Gaussian coordinates
with respect to which all conditional means can be differentiated.
Introduce an independent standard Gaussian vector $z_0$ and replace
the initial field by
\[
 x_0=\sqrt\varepsilon\,z_0.
\]
Include this variable in the outer normalization:
\[
 Z_{s,\varepsilon}
 =
 \E_{W,z_0}
 e^{sU_\mu(0,\sqrt\varepsilon z_0;W)}.
\]
Let $Z$ denote the collection of all independent standard Gaussian
coordinates used to generate $W$, $z_0$, and the field increments
$z_1,\ldots,z_L$.

We now compute the joint density of $(Z,\sigma)$. Multiplying the outer
tilt, the normalized transition densities, and the terminal Gibbs
weight causes the intermediate normalization factors to telescope.
The result is
\begin{equation}
 \frac1{Z_{s,\varepsilon}}
 \exp\left(
 H_\sigma(Z)
 -
 \int_0^1 U_\mu(q,X_q;W)\,\nu(\dd q)
 \right),
 \label{eq:fn-joint-law}
\end{equation}
relative to the product standard Gaussian law of $Z$ and counting
measure on $\{-1,1\}^N$. Here $H_\sigma(Z)$ is the terminal exponent
corresponding to the spin configuration $\sigma$.

For a step profile, the coefficients of
$U_0,U_1,\ldots,U_L$ in the integral against $\nu$ are
\[
 a_0-s,\quad
 a_1-a_0,\quad\ldots,\quad
 a_{L-1}-a_{L-2},\quad
 1-a_{L-1}.
\]
They are nonnegative and sum to $1-s=\nu([0,1])$. Thus every
intermediate value $U_j$ enters the logarithm of the density with a
nonpositive coefficient.

The joint law also has the usual spin-gauge symmetry. For
$\tau\in\{-1,1\}^N$, simultaneously transform
\[
 \sigma_i\mapsto\tau_i\sigma_i,\qquad
 W_{ij}\mapsto\tau_i\tau_jW_{ij},\qquad
 z_{k,i}\mapsto\tau_i z_{k,i}.
\]
These transformations preserve the underlying Gaussian laws and all
values $U_j$, while acting transitively on spin configurations.
Therefore $
 \mathbb P(\sigma)=2^{-N}$. 
Moreover, any observable invariant under this simultaneous gauge
transformation has the same conditional law for every fixed value of
$\sigma$. This is the symmetry used below to pass between fixed-spin
posteriors and the full joint law.

We next identify the curvature of a fixed-spin posterior. For each
$q$, regard
\[
 U_q(Z)=U_\mu(q,X_q;W)
\]
as a function of the Gaussian coordinates revealed through $q$, and
extend it constantly in coordinates revealed later. At the grid points
this agrees with $U_j$. Each $U_q$ is convex in $Z$, because the terminal value
is a log-sum-exp of affine functions, and the recursion preserves
convexity.

Fix $\sigma$. Including the Gaussian reference density, the negative
logarithm of the conditional density of $Z$ given $\sigma$, relative
to Lebesgue measure on the coordinate space, is, up to a constant,
\begin{equation}
 V(Z)
 =
 \frac{|Z|^2}{2}
 -H_\sigma(Z)
 +\int_0^1U_q(Z)\,\nu(\dd q).
 \label{eq:fn-posterior-potential}
\end{equation}
The terminal exponent $H_\sigma$ is affine in $Z$, so it contributes
nothing to the Hessian. Hence
\begin{equation}
 \begin{aligned}
 Q(Z)
 :=
 \nabla_Z^2V(Z)
 &=
 I+\int_0^1A_q(Z)\,\nu(\dd q),\\
 A_q(Z)
 &:=
 \nabla_Z^2U_q(Z)\succeq0,
 \end{aligned}
 \label{eq:fn-posterior-metric}
\end{equation}
and therefore
\[
 Q(Z)\succeq I.
\]
Thus the fixed-spin posterior is uniformly log-concave in the original
Gaussian coordinates. The Hessian in
\eqref{eq:fn-posterior-metric} is taken with respect to the independent
Gaussian coordinates used to generate both $W$ and the auxiliary
fields, with $X_q$ understood as their accumulated field at level $q$.
This distinction is important because, after the outer tilt, the
marginal distribution of $W$ under $\mathbb P$ is generally not
Gaussian.

\subsection{Martingale identities and contact conditions at a minimizing measure}

We now derive the identities that relate the conditional spin covariance
to the Hessian of the finite-$N$ Parisi solution, and then use
variational optimality to strengthen these bounds on the support of a
minimizing measure.

Work under the joint law constructed in the preceding subsection, with
$x_0=0$. Conditional on the coupling matrix $W$, the auxiliary field
process satisfies
\begin{equation}
 \dd X_q
 =c\,a(q)m_q\,\dd q+\sqrt c\,\dd B_q,
 \qquad
 m_q=\nabla_xU_\mu(q,X_q;W),
 \qquad
 H_q=\nabla_x^2U_\mu(q,X_q;W),
 \label{eq:fn-ancestral-sde}
\end{equation}
where $B$ is standard $N$-dimensional Brownian motion under the
conditional law given $W$. The tilt used to define the joint law
changes the distribution of $W$, but not this conditional Brownian
law.

Set
\[
 C_q=\Cov(\sigma\mid\mathcal G_q),
 \qquad
 \mathbb E_q[\cdot]=\mathbb E[\cdot\mid\mathcal G_q].
\]

\begin{proposition}
\zlabel{prop:fn-contact}
Fix $N\ge1$, $\beta>0$, $0<s\le1/2$, and $0\le t<1$.
Let $\mu=s\delta_0+\nu$ be a deterministic admissible measure, and
write
\[
 c=\beta^2(1-t),\qquad a(q)=\mu([0,q]).
\]
Under the normalized joint law with $x_0=0$, let $m_q$, $C_q$,
and $H_q$ be the conditional mean spin, conditional spin covariance,
and spatial Hessian defined above. Then, for every $q\in[0,1]$,
\[
 a(q)C_q\preceq H_q\preceq C_q
 \qquad\text{almost surely}.
\]
If $\mu$ minimizes $\mathcal F_{t,s}$, then
$K=\supp\nu\subset[0,1)$ and, for every $q\in K$,
\[
 \E|m_q|^2=Nq,\qquad
 \E\Tr H_q^2\le\frac{N}{c}.
\]
These conclusions include $q=0$ whenever $0\in K$.
\end{proposition}

\begin{proof}
Differentiating the Parisi PDE in the spatial variable and applying
It\^o's formula along $X_q$ gives
\begin{equation}
 \dd m_q=\sqrt c\,H_q\,\dd B_q,
 \qquad
 \dd H_q=-c\,a(q)H_q^2\,\dd q
           +\sqrt c\,\nabla_xH_q\,\dd B_q.
 \label{eq:fn-martingale-sdes}
\end{equation}
Thus $m_q$ is a martingale. Since $m_1$ is the terminal Gibbs mean,
the tower property gives
\[
 m_q=\mathbb E[\sigma\mid\mathcal G_q].
\]
At $q=1$, differentiation of the terminal log-sum-exp gives
$C_1=H_1$. The martingale isometry applied to $m_q$, together with
\eqref{eq:fn-martingale-sdes}, therefore gives 
\[
 C_q
 =\mathbb E_qH_1
   +c\int_q^1\mathbb E_qH_u^2\,\dd u.
\]
Taking conditional expectation in the equation for $H_q$ gives
similarly
\[
 H_q
 =\mathbb E_qH_1
   +c\int_q^1a(u)\mathbb E_qH_u^2\,\dd u.
\]
Since $a$ is nondecreasing and $a(q)\le a(u)\le1$ for $u\ge q$,
all matrices in these integrals are positive semidefinite, and hence
\begin{equation}
 a(q)C_q\preceq H_q\preceq C_q.
 \label{eq:fn-covariance-order}
\end{equation}

We also need the evolution of the conditional mean spin. Define
\[
 \Gamma(q)=\frac1N\mathbb E|m_q|^2.
\]
The first equation in \eqref{eq:fn-martingale-sdes} and It\^o's
isometry give
\begin{equation}
 \Gamma'(q)=\frac cN\mathbb E\Tr H_q^2.
 \label{eq:fn-gamma-derivative}
\end{equation}

We next use the fact that $\mu=s\delta_0+\nu$ minimizes the
interpolating functional. Assume $c>0$. Let $\nu'$ be any
nonnegative measure on $[0,1]$ with
$\nu'([0,1])=1-s$, and let
\[
 \delta a(q)=(\nu'-\nu)([0,q])
\]
be the corresponding variation of the cumulative profile.
Linearizing the PDE in this direction gives
\[
 \delta U_\mu(0,0;W)
 =
 \frac c2\,
 \mathbb E\left[
   \int_0^1\delta a(q)|m_q|^2\,\dd q
   \,\middle|\,W
 \right].
\]
Differentiating the logarithmic outer moment averages this identity
under the joint law. After exchanging the order of integration, the
first variation of $\mathcal F_{t,s}$ is
\begin{equation}
 D\mathcal F_{t,s}(\mu)[\nu'-\nu]
 =
 \frac c2\int_0^1 g(q)\,(\nu'-\nu)(\dd q),
 \qquad
 g(q)=\int_q^1(\Gamma(u)-u)\,\dd u.
 \label{eq:fn-first-variation}
\end{equation}

Because $\mu$ is a minimizer, the expression in
\eqref{eq:fn-first-variation} is nonnegative for every admissible
$\nu'$. In particular, taking
$\nu'=(1-s)\delta_x$ gives
\[
 (1-s)g(x)\ge\int g(q)\,\nu(\dd q)
 \qquad\text{for every }x\in[0,1].
\]
The right-hand side is itself at least
$(1-s)\min_{[0,1]}g$. Hence equality must hold, and
\[
 \supp\nu\subset\argmin_{[0,1]}g.
\]

Let $K=\supp\nu$. The identities above and the fixed-$N$ regularity
proved below imply that $\Gamma'$ is continuous and
\[
 g'(q)=q-\Gamma(q),
 \qquad
 g''(q)=1-\Gamma'(q).
\]
If $q\in K\cap(0,1)$, then $q$ is a minimum point of $g$, so
\[
 g'(q)=0,\qquad g''(q)\ge0.
\]
Therefore
\[
 \Gamma(q)=q,\qquad \Gamma'(q)\le1.
\]

The endpoints require separate checks. At $q=0$, global spin-flip
symmetry gives $m_0(W)=0$ for every $W$, and hence
$\Gamma(0)=0$ and $g'(0)=0$. If $0\in K$, the one-sided minimum
condition together with Taylor's formula gives $g''(0)\ge0$.
At $q=1$, the finite-temperature terminal Gibbs distribution assigns
positive probability to every spin configuration. Thus
$|m_1|^2<N$, so
\[
 \Gamma(1)<1,\qquad g'(1)=1-\Gamma(1)>0.
\]
A minimum at the right endpoint would require the left derivative to
be nonpositive, so $1\notin K$.

Combining $\Gamma(q)=q$ and $\Gamma'(q)\le1$ on $K$ with
\eqref{eq:fn-gamma-derivative}, we obtain, for every $q\in K$,
\begin{equation}
 \mathbb E|m_q|^2=Nq,
 \qquad
 \mathbb E\Tr H_q^2\le\frac Nc.
 \label{eq:fn-contact}
\end{equation}
These conclusions also apply when $\nu$ has an atom at zero. If
$0\notin K$, then only the prescribed mass $s\delta_0$ is present
there, and minimization with respect to the remaining measure $\nu$
does not provide the second-order bound at $q=0$.
\end{proof}

\subsection{The fixed-profile interpolation identity}

Fix a deterministic admissible measure $\mu$. We compute the derivative
of $\mathcal F_{t,s}(\mu)$ without assuming that $\mu$ minimizes the
functional. The matrix interaction, auxiliary-field variance, and
penalty contribute separately; their sum is a mean-square overlap
deviation.

We first record a differentiation identity for the finite-step
recursion. Let $b(\sigma)$ be any function of the terminal spin, add
$yb(\sigma)$ to the terminal exponent, and differentiate with respect
to $y$ at $y=0$. Define
\[
 b_q=\mathbb E[b(\sigma)\mid\mathcal G_q].
\]
At a grid point $q_j$, write $\mathbb E_j$ and $\Var_j$ for
conditional expectation and variance given $\mathcal G_{q_j}$ under
the normalized law. If primes denote derivatives with respect to $y$,
then one recursion step gives
\[
 U_j'=\mathbb E_jU_{j+1}',
 \qquad
 U_j''=\mathbb E_jU_{j+1}''
       +a_j\Var_j(U_{j+1}').
\]
Iterating these identities from the terminal level to $q=0$ yields
\begin{equation}
 U_0''
 =
 \mathbb E[b(\sigma)^2\mid\mathcal G_0]
 -
 \int_0^1
 \mathbb E[b_q^2\mid\mathcal G_0]\,\mu(\dd q).
 \label{eq:fn-terminal-variation}
\end{equation}

We first differentiate the matrix interaction while keeping the
auxiliary-field variance $c$ fixed. Write the completed SK contribution
to the terminal exponent as
\[
 \sqrt t\sum_e g_ev_e(\sigma),
\]
where
\[
 v_{ij}(\sigma)=\frac{\beta}{\sqrt N}\sigma_i\sigma_j
 \quad(i<j),
 \qquad
 v_{ii}(\sigma)=\frac{\beta}{\sqrt{2N}}.
\]
These coefficients satisfy
\[
 \sum_ev_e(\sigma)v_e(\tau)
 =
 \frac{\beta^2N}{2}R(\sigma,\tau)^2,
 \qquad
 R(\sigma,\tau)=\frac{\sigma\cdot\tau}{N}.
\]

For each overlap level $q$, condition on $\mathcal G_q$ and sample
$\sigma^1,\sigma^2$ independently from the conditional Gibbs law of
the terminal spin. Define
\[
 B(q)=
 \mathbb E\bigl[R(\sigma^1,\sigma^2)^2\bigr].
\]
To differentiate in $t$, regard
\[
 y_e=\sqrt t\,g_e
\]
as Gaussian heat coordinates, and write $U_{0,e}$ and $U_{0,ee}$ for
the corresponding first and second derivatives of $U_0$ before
substituting $y_e=\sqrt t\,g_e$. Gaussian differentiation of the outer
logarithmic moment gives the matrix contribution
\[
 \frac1{2N}\mathbb E\sum_e
 \bigl(U_{0,ee}+sU_{0,e}^2\bigr).
\]
Applying \eqref{eq:fn-terminal-variation} with
$b(\sigma)=v_e(\sigma)$ and summing over $e$ gives
\begin{equation}
 \frac1{2N}\mathbb E\sum_e
 \bigl(U_{0,ee}+sU_{0,e}^2\bigr)
 =
 \frac{\beta^2}{4}
 \left(1-\int_0^1B(q)\,\nu(\dd q)\right).
 \label{eq:fn-matrix-derivative}
\end{equation}
The term $sU_{0,e}^2$ accounts exactly for the prescribed mass
$s\delta_0$: it removes that mass from the integral against
$\mu=s\delta_0+\nu$, leaving the integral against $\nu$.
The same heat-semigroup interpretation defines the derivative at
$t=0$.

We next differentiate the auxiliary-field variance $c$. Set
\[
 h(q)=\frac1N\mathbb E\Tr H_q,
 \qquad
 \Gamma(q)=\frac1N\mathbb E|m_q|^2.
\]
Differentiating the Gaussian field increments in the finite-step
recursion gives
\[
 \partial_c
 \left(
 \frac1{Ns}\log\mathbb E_W e^{sU_0}
 \right)
 =
 \frac12\int_0^1
 \bigl(h(q)+a(q)\Gamma(q)\bigr)\,\dd q.
\]
To simplify this expression, use the martingale identities from the
preceding subsection. They imply
\[
 h'=-a\Gamma',
 \qquad
 h(1)=1-\Gamma(1).
\]
Stieltjes integration then gives
\[
 h(q)+a(q)\Gamma(q)
 =
 1-\int_{(q,1]}\Gamma(u)\,\mu(\dd u).
\]
Integrating in $q$ and using
$\partial_t c=-\beta^2$, the auxiliary-field contribution to
$\partial_t\mathcal F_{t,s}$ is therefore
\[
 -\frac{\beta^2}{2}
 \left(
 1-\int_0^1q\Gamma(q)\,\mu(\dd q)
 \right).
\]

Finally, differentiating the penalty
\[
 -\frac c2\int_0^1qa(q)\,\dd q
\]
and using $\partial_t c=-\beta^2$ gives
\[
 \frac{\beta^2}{4}
 \left(
 1-\int_0^1q^2\,\mu(\dd q)
 \right).
\]
Because $q$ and $q^2$ vanish at zero, the last two integrals are
unchanged if $\mu$ is replaced by
$\nu=\mu-s\delta_0$.

Combining the matrix, auxiliary-field, and penalty contributions gives
\[
 \begin{aligned}
 \partial_t\mathcal F_{t,s}(\mu)
 &=
 -\frac{\beta^2}{4}
 \int_0^1
 \bigl(B(q)-2q\Gamma(q)+q^2\bigr)\,\nu(\dd q)\\
 &=
 -\frac{\beta^2}{4}
 \int_0^1
 \mathbb E\left[
   \bigl(R(\sigma^1,\sigma^2)-q\bigr)^2
 \right]\nu(\dd q).
 \end{aligned}
\]
Indeed, conditional independence of $\sigma^1,\sigma^2$ given
$\mathcal G_q$ gives
\[
 \mathbb E R(\sigma^1,\sigma^2)
 =
 \frac1N\mathbb E|m_q|^2
 =
 \Gamma(q).
\]
Thus, with
\[
 D_{t,s}(\mu)
 =
 \int_0^1
 \mathbb E\left[
   \bigl(R(\sigma^1,\sigma^2)-q\bigr)^2
 \right]\nu(\dd q),
\]
we obtain
\begin{equation}
 \partial_t\mathcal F_{t,s}(\mu)
 =
 -\frac{\beta^2}{4}D_{t,s}(\mu),
 \qquad
 0\le D_{t,s}(\mu)\le4(1-s).
 \label{eq:fn-fixed-interpolation}
\end{equation}
The fixed-profile derivative is therefore nonpositive. We justify
integration of this identity and passage to the minimum in
\zcref{subsec:fn-approximation}.

\subsection{The integrated variance bound}

We now apply Brascamp--Lieb to an average over overlap levels, retaining
the integrated curvature in \eqref{eq:fn-posterior-metric}. Using the
same terminal spin at every level, define
\[
 S_q=\frac{|m_q|^2}{N},\qquad
 L_q=\frac{\sigma\cdot m_q-|m_q|^2/2}{N}.
\]
The following proposition gives an averaged variance bound for every
admissible measure and a sharper bound when the measure minimizes
$\mathcal F_{t,s}$.

\begin{proposition}
\zlabel{prop:fn-integrated-variance}
Fix $N\ge1$, $\beta>0$, $0<s\le1/2$, and $0\le t<1$.
Let $\mu=s\delta_0+\nu$ be a deterministic admissible measure, with
$c=\beta^2(1-t)$ and $a(q)=\mu([0,q])$. All expectations and
variances below are under its normalized joint law with $x_0=0$.

For every deterministic $k\in L^2(\nu)$, the observable
\[
 Z_k=\int_0^1 k(q)L_q\,\nu(\dd q)
\]
satisfies
\[
 \Var Z_k
 \le
 \frac1{N^2}\int_0^1 k(q)^2
 \E\Tr(H_qC_q)\,\nu(\dd q).
\]
If $\mu$ minimizes $\mathcal F_{t,s}$, then
\[
 \E\Tr C_q^2\le\frac{N}{c\,a(q)^2}
 \qquad(q\in\supp\nu),
\]
and
\[
 \Var Z_k
 \le
 \frac1{Nc}\int_0^1\frac{k(q)^2}{a(q)}\,\nu(\dd q).
\]
\end{proposition}

\begin{proof}

First take a step profile and use the regularized joint law
\eqref{eq:fn-joint-law}, with $x_0=\sqrt\varepsilon\,z_0$ and
$\varepsilon>0$. Differentiating $L_q$ with the spin fixed produces
$\sigma-m_q$. The initial-field coordinates $z_0$ occur in every
$X_q$ and let us express this derivative using the matrices $A_q$.
For this coordinate block,
\[
 \nabla_{z_0}U_q=\sqrt\varepsilon\,m_q,
 \qquad
 (A_q)_{z_0,z_0}=\varepsilon H_q.
\]
Define a vector $v_q$ in the full Gaussian coordinate space by setting
all blocks except the $z_0$ block equal to zero and taking
\[
 (v_q)_{z_0}=\frac{\sigma-m_q}{\sqrt\varepsilon}.
\]
With the terminal spin held fixed, differentiation gives
\begin{equation}
 \nabla_Z Z_k
 =
 \frac1N\int_0^1 k(q)A_qv_q\,\nu(\dd q).
 \label{eq:fn-posterior-gradient}
\end{equation}
Thus the same matrices $A_q$ that appear in the posterior Hessian also
appear in the gradient of $Z_k$. This is the structural reason that
the integrated posterior curvature can control $\Var Z_k$.

Fix the terminal spin $\sigma$, and write $\Var_\sigma$ and
$\mathbb E_\sigma$ for variance and expectation under its posterior
law. If a probability density has negative log density $V$ with
\[
 Q=\nabla^2V\succ0,
\]
the Brascamp--Lieb inequality gives
\begin{equation}
 \Var_\sigma(F)
 \le
 \mathbb E_\sigma\!\left[
   \nabla F^TQ^{-1}\nabla F
 \right].
 \label{eq:fn-brascamp-lieb}
\end{equation}
For completeness, this follows by solving
$-\mathcal L u=F-\mathbb E_\sigma F$ for
$\mathcal L=\Delta-\nabla V\cdot\nabla$. Integration by parts and
weighted Cauchy--Schwarz give
\[
 \Var_\sigma(F)^2
 \le
 \mathbb E_\sigma[\nabla F^TQ^{-1}\nabla F]\,
 \mathbb E_\sigma[\nabla u^TQ\nabla u],
\]
while the integrated Bochner identity gives
\[
 \mathbb E_\sigma[\nabla u^TQ\nabla u]
 \le
 \mathbb E_\sigma[(\mathcal Lu)^2]
 =
 \Var_\sigma(F).
\]
Smooth cutoff and $L^2$ approximation extend the inequality to the
observables used below. The finite-step posterior potentials are smooth,
uniformly convex, and have Gaussian tails.

We apply \eqref{eq:fn-brascamp-lieb} to $Z_k$. Set
\[
 b=\int_0^1k(q)A_qv_q\,\nu(\dd q),
 \qquad
 w=Q^{-1}b.
\]
Because
\[
 \int_0^1A_q\,\nu(\dd q)=Q-I\preceq Q,
\]
matrix Cauchy--Schwarz gives
\[
 \begin{aligned}
 b^TQ^{-1}b
 &=\int_0^1 k(q)w^TA_qv_q\,\nu(\dd q)\\
 &\le
 \left(
   \int_0^1w^TA_qw\,\nu(\dd q)
 \right)^{1/2}
 \left(
   \int_0^1k(q)^2v_q^TA_qv_q\,\nu(\dd q)
 \right)^{1/2}\\
 &\le
 (b^TQ^{-1}b)^{1/2}
 \left(
   \int_0^1k(q)^2v_q^TA_qv_q\,\nu(\dd q)
 \right)^{1/2}.
 \end{aligned}
\]
Hence
\[
 b^TQ^{-1}b
 \le
 \int_0^1k(q)^2v_q^TA_qv_q\,\nu(\dd q).
\]
Using \eqref{eq:fn-posterior-gradient} and the identity
\[
 v_q^TA_qv_q
 =
 (\sigma-m_q)^TH_q(\sigma-m_q),
\]
we obtain
\[
 \Var_\sigma Z_k
 \le
 \frac1{N^2}
 \int_0^1k(q)^2
 \mathbb E_\sigma[
   (\sigma-m_q)^TH_q(\sigma-m_q)
 ]\,\nu(\dd q).
\]

Gauge symmetry shows that the conditional law of $Z_k$, and hence its
conditional mean and variance, does not depend on the fixed value of
$\sigma$. The law of total variance therefore gives the same estimate
under the full joint law. Conditioning on $\mathcal G_q$ then yields
\begin{equation}
 \begin{aligned}
 \Var Z_k
 &\le
 \frac1{N^2}\int_0^1k(q)^2
 \mathbb E[
   (\sigma-m_q)^TH_q(\sigma-m_q)
 ]\,\nu(\dd q)\\
 &=
 \frac1{N^2}\int_0^1k(q)^2
 \mathbb E\Tr(H_qC_q)\,\nu(\dd q).
 \end{aligned}
 \label{eq:fn-integrated-covariance}
\end{equation}

We next remove the auxiliary initial variance $\varepsilon$. At fixed
$N,t,s,\mu$, the recursion and transition laws depend continuously on
the initial field. Moreover,
\[
 |L_q|\le\frac32,\qquad
 0\preceq H_q\preceq C_q,\qquad
 \Tr C_q\le N,
\]
so the trace integrand in
\eqref{eq:fn-integrated-covariance} is bounded by $N^2$.
The outer value $U_0$ has at most linear Gaussian growth uniformly for
$0\le\varepsilon\le1$. Dominated convergence therefore gives
\eqref{eq:fn-integrated-covariance} at $\varepsilon=0$.

Finally suppose that $\mu=s\delta_0+\nu$ minimizes the interpolating
functional, and put
\[
 K=\supp\nu,\qquad M=Nc.
\]
For $q\in K$, the contact estimate
\eqref{eq:fn-contact} gives
\[
 \mathbb E\Tr H_q^2\le\frac{N}{c}=\frac{N^2}{M}.
\]
Together with
$a(q)C_q\preceq H_q$, this converts
\eqref{eq:fn-integrated-covariance} into a bound depending only on the
cumulative mass $a(q)$. Indeed,
\[
 H_q=a(q)C_q+D_q,\qquad D_q\succeq0,
\]
implies
\[
 \Tr H_q^2\ge a(q)^2\Tr C_q^2
\]
and
\[
 \Tr(H_qC_q)
 \le\frac1{a(q)}\Tr H_q^2.
\]
For the second inequality, take the trace against $H_q$ in
$H_q-a(q)C_q\succeq0$:
\[
 0\le\Tr\bigl(H_q(H_q-a(q)C_q)\bigr)
   =\Tr H_q^2-a(q)\Tr(H_qC_q).
\]
Consequently,
for $q\in K$,
\[
 \mathbb E\Tr C_q^2
 \le
 \frac{N^2}{M\,a(q)^2},
\]
and, for every deterministic $k\in L^2(\nu)$,
\begin{equation}
 \Var\left(\int_0^1k(q)L_q\,\nu(\dd q)\right)
 \le
 \frac1M
 \int_0^1\frac{k(q)^2}{a(q)}\,\nu(\dd q).
 \label{eq:fn-contact-covariance}
\end{equation}
This is the integrated posterior variance estimate used in the next
section. Its dependence on $a(q)^{-1}$ is the reason that the later
argument works with quantile intervals of $\nu$ when $s$ is small.
\end{proof}

\subsection{Decomposing the overlap error and passing to quantiles}

Fix a minimizing measure $\mu=s\delta_0+\nu$ and write
$D=D_{t,s}(\mu)$. At each $q$, let $\sigma^1,\sigma^2$ be
independent samples from the conditional law of the terminal spin
given $\mathcal G_q$. Using the observables $S_q,L_q$ defined above,
we separate $D$ into a trace term and a term involving $\Var(L_q)$.
The preceding proposition controls variances of averages of $L_q$, and
\zcref{sec:residual} will convert that control into the pointwise
variance estimate needed here.

Set
\[
 j_q=\frac1{N^2}\mathbb E[m_q^TC_qm_q],
 \qquad
 r_q=\Var(L_q).
\]
Since $\mathbb E[L_q\mid\mathcal G_q]=S_q/2$ and
$\Var(L_q\mid\mathcal G_q)=N^{-2}m_q^TC_qm_q$, the law of
total variance gives
\[
 r_q=j_q+\frac14\Var(S_q).
\]
Similarly, conditional independence of $\sigma^1$ and $\sigma^2$
given $\mathcal G_q$ yields
\[
 \mathbb E\bigl[(R(\sigma^1,\sigma^2)-q)^2\bigr]
 =
 \frac1{N^2}\mathbb E\Tr C_q^2
 +2j_q
 +\mathbb E(S_q-q)^2.
\]
On the support of $\nu$, the contact identity
$\mathbb ES_q=q$ from \eqref{eq:fn-contact} gives
\[
 \mathbb E(S_q-q)^2=\Var(S_q).
\]
Since
\[
 2j_q+\Var(S_q)
 \le
 4\left(j_q+\frac14\Var(S_q)\right)
 =4r_q,
\]
we obtain
\begin{equation}
 D\le A+4\mathcal W,
 \qquad
 A=\frac1{N^2}\int\mathbb E\Tr C_q^2\,\nu(\dd q),
 \qquad
 \mathcal W=\int r_q\,\nu(\dd q).
 \label{eq:fn-defect}
\end{equation}
The term $A$ will be controlled directly by the covariance estimate,
while $\mathcal W$ is the integrated pointwise variance that remains
to be estimated.

We first bound $A$. Let
\[
 Q_\nu(v)=\inf\{q\in[0,1]:\nu([0,q])\ge v\},
 \qquad 0<v<1-s,
\]
be the generalized quantile function of $\nu$. The pushforward of
Lebesgue measure on $(0,1-s)$ under $Q_\nu$ is $\nu$, and
\[
 a(Q_\nu(v))
 =
 s+\nu([0,Q_\nu(v)])
 \ge s+v.
\]
Moreover, $Q_\nu(v)\in K=\supp\nu$ for Lebesgue-almost every $v$.
Using \zcref{prop:fn-integrated-variance} therefore gives
\[
 \begin{aligned}
 A
 &\le
 \frac1M
 \int_0^{1-s}\frac{\dd v}{(s+v)^2}\\
 &=
 \frac1M\left(\frac1s-1\right)
 \le\frac1{Ms},
 \end{aligned}
\]
where $M=Nc$. Since also
$\Tr C_q^2\le(\Tr C_q)^2\le N^2$, we have the trivial bound
$A\le1$. Hence
\begin{equation}
 A\le\min\{1,(Ms)^{-1}\}.
 \label{eq:fn-trace-error}
\end{equation}

The remaining term $\mathcal W$ is treated in \zcref{sec:residual}.
There we apply \eqref{eq:fn-contact-covariance} to averages over
quantile intervals and compare those averages with the individual
variables $L_q$.

\subsection{Approximation and the optimized envelope}
\zlabel{subsec:fn-approximation}

We now extend the fixed-profile identities to arbitrary admissible
measures at fixed $N$. This completes the contact and variance
propositions and supplies the continuity needed to minimize the
interpolation and differentiate its value.

\subsubsection{Approximation for a fixed profile}
We first record estimates that are uniform over step approximations.
At the terminal time, every spatial derivative of fixed order is
uniformly bounded in $W$ and $x$, because these derivatives are
cumulants of spins taking values in $\{-1,1\}^N$. The maximum
principle gives
\[
 |\nabla_xU_\mu|\le\sqrt N.
\]
Successive differentiation of the PDE then bounds every further fixed
spatial derivative by a constant depending only on its order, $N$,
and $\beta$, uniformly over all step profiles with $0\le a\le1$.
Indeed, after the lower-order derivatives have been bounded, the
equation for a derivative of order $k$ is linear with bounded drift
and coefficients, and its forcing contains only derivatives of lower
order. We will use these bounds only at fixed $N$. In particular,
\[
 0\preceq H_q\preceq C_q\preceq NI
\]
gives a uniform Hessian bound.

We next quantify continuity in the cumulative profile. Let $a$ and
$\widetilde a$ be two cumulative functions, and let $U$ and
$\widetilde U$ be the corresponding solutions with the same terminal
condition. Set $w=U-\widetilde U$. Subtracting the two PDEs gives
\[
 \partial_q w+\frac c2\Delta w
 +\frac{ca}{2}(\nabla U+\nabla\widetilde U)\cdot\nabla w
 =-\frac c2(a-\widetilde a)|\nabla\widetilde U|^2,
 \qquad w(1,\cdot)=0.
\]
Since $|\nabla\widetilde U|^2\le N$, the backward representation gives
\begin{equation}
 \sup_{W,x}|U(q,x;W)-\widetilde U(q,x;W)|
 \le
 \frac{cN}{2}\int_q^1|a(u)-\widetilde a(u)|\,\dd u.
 \label{eq:fn-coefficient-continuity}
\end{equation}
Thus $L^1$ convergence of the cumulative functions implies uniform
convergence of the corresponding solutions. Applying the same argument
to the differentiated equations, together with the fixed-order bounds
above and Gronwall's inequality, gives uniform convergence of the first
and second spatial derivatives needed below.

To compare the stochastic processes for two profiles, place them on
one probability space using the same Gaussian matrix $W$ under its
original Gaussian reference law and the same Brownian motion $B$.
The drifts in \eqref{eq:fn-ancestral-sde} are uniformly bounded and
spatially Lipschitz at fixed $N$. Hence, if $a_n\to a$ in $L^1$, then
\[
 X^{(n)}\to X,\qquad
 m^{(n)}\to m,\qquad
 H^{(n)}\to H
\]
in probability, uniformly in $q\in[0,1]$. The field differences are
uniformly bounded at fixed $N$, and $m^{(n)}$ and $H^{(n)}$ are also
uniformly bounded. The convergence therefore holds in every finite
$L^p$ norm required below.

The normalized law also depends on the profile through the outer tilt.
Set
\[
 d_n=\frac{cN}{2}\|a_n-a\|_{L^1}.
\]
By \eqref{eq:fn-coefficient-continuity},
\[
 |U_n(0,0;W)-U(0,0;W)|\le d_n.
\]
Therefore the ratio of the two normalized matrix densities lies between
\[
 e^{-2sd_n}\quad\text{and}\quad e^{2sd_n}.
\]
Since $d_n\to0$, convergence proved under the common Gaussian reference
coupling also passes to the normalized matrix laws.

It remains to couple the terminal spins. Fix an ordering of
$\{-1,1\}^N$ and use the same independent uniform random variable to
sample from the terminal Gibbs probabilities for every profile. These
probabilities depend continuously on the terminal field, so the coupled
terminal spins agree with probability tending to one. Consequently the
observables involving the same terminal spin at several overlap levels,
including those in \eqref{eq:fn-integrated-covariance}, converge as
well. The first and second conditional moments of the terminal spin
solve linear backward equations with drift $ca\nabla U$ and bounded
smooth terminal data. Their solutions therefore converge under the same
approximation, and so does $C_q$.

\subsubsection{Passage from step profiles to arbitrary measures}
Let
\[
 \mu=s\delta_0+\nu
\]
be any admissible measure. Approximate $\nu$ by pushforwards onto
successively finer finite meshes, preserving its total mass and the
point $0$, and write
\[
 \mu_n=s\delta_0+\nu_n.
\]
Then $\mu_n\Rightarrow\mu$, and the corresponding cumulative functions
$a_n$ converge to $a$ in $L^1([0,1])$. The estimates above, together
with Gaussian domination of the outer tilts, allow us to pass
\eqref{eq:fn-martingale-sdes}--\eqref{eq:fn-first-variation},
\eqref{eq:fn-terminal-variation}--\eqref{eq:fn-fixed-interpolation},
and \eqref{eq:fn-integrated-covariance} to the limit.

If an identity is needed at a specified time $q$, we include $q$ in
every approximating mesh. The conditional moment functions and the
integrated second moments are continuous in $q$, uniformly along the
approximation, so the corresponding Stieltjes integrals converge as
well. For \eqref{eq:fn-integrated-covariance}, first take continuous
deterministic $k$. Since
\[
 |L_q|\le\frac32,
 \qquad
 \mathbb E\Tr(H_qC_q)\le N^2,
\]
density then extends the result to every $k\in L^2(\nu)$, including
the functions used to divide the mass of an atom between quantile
intervals.

If $\mu$ minimizes the functional, apply the limiting first variation
\eqref{eq:fn-first-variation} directly to $\mu$. Continuity of
$\Gamma'$ permits the same endpoint and support argument used in
\zcref{prop:fn-contact}. Together with the limiting integrated
variance bound, this yields the stated extensions of
\zcref{prop:fn-contact,prop:fn-integrated-variance}.

\subsubsection{The optimized interpolation}
The continuity estimates above allow us to minimize over all
admissible measures and integrate the fixed-profile derivative.
\begin{proposition}
\zlabel{prop:fn-envelope}
Fix $N\ge1$, $\beta>0$, and $0<s\le1/2$. For every deterministic
admissible measure $\mu$, the function
$t\mapsto\mathcal F_{t,s}(\mu)$ is Lipschitz on $[0,1]$ and
continuously differentiable on $(0,1)$, with
\[
 \partial_t\mathcal F_{t,s}(\mu)
 =-\frac{\beta^2}{4}D_{t,s}(\mu),
 \qquad
 0\le D_{t,s}(\mu)\le4(1-s).
\]

A minimizing measure exists for every $t\in[0,1]$. The optimized
value
\[
 f(t)=\min_\mu\mathcal F_{t,s}(\mu)
\]
is Lipschitz, with
\[
 f(0)=P(s),\qquad f(1)=\phi_N^c(s).
\]
At every differentiability point $t\in(0,1)$ of $f$, every
minimizer $\mu_t$ satisfies
\[
 f'(t)=-\frac{\beta^2}{4}D_{t,s}(\mu_t).
\]
Consequently,
\[
 P(s)-\phi_N^c(s)
 =\frac{\beta^2}{4}\int_0^1D_{t,s}(\mu_t)\,\dd t.
\]
The integrand is defined almost everywhere by $-4f'(t)/\beta^2$, and 
at every differentiability point this agrees with
$D_{t,s}(\mu_t)$ for every minimizer $\mu_t$.
\end{proposition}

\begin{proof}

For fixed $s$, the admissible measures
\[
 \mu=s\delta_0+\nu,\qquad
 \nu\ge0,\qquad
 \nu([0,1])=1-s,
\]
form a weakly compact set. On $[0,1]$, weak convergence of these
measures implies $L^1$ convergence of their cumulative functions.
Equation \eqref{eq:fn-coefficient-continuity} therefore makes
$\mathcal F_{t,s}$ continuous in $\mu$, so a deterministic minimizer
exists for every $t$.

Gaussian domination gives continuity at $t=0$ and $t=1$, uniformly
over admissible measures. On compact subintervals of $(0,1)$, the
common Brownian coupling above gives joint continuity of
\[
 (t,\mu)\longmapsto D_{t,s}(\mu).
\]
Passing the finite-step interpolation identity to the limit therefore
shows that, for every fixed admissible $\mu$,
$\mathcal F_{t,s}(\mu)$ is $C^1$ for $0<t<1$, with
\[
 \partial_t\mathcal F_{t,s}(\mu)
 =
 -\frac{\beta^2}{4}D_{t,s}(\mu).
\]
Since $0\le D_{t,s}(\mu)\le4(1-s)$, these derivatives are uniformly
bounded. Hence both the fixed-profile functionals and their minimum
$f$ are Lipschitz in $t$.

Let $t$ be a differentiability point of $f$, and let $\mu_t$ be any
minimizer at that time. Since
\[
 f(t+h)\le\mathcal F_{t+h,s}(\mu_t),
 \qquad
 f(t)=\mathcal F_{t,s}(\mu_t),
\]
the quotient for $h>0$ gives
\[
 f'(t)\le\partial_t\mathcal F_{t,s}(\mu_t),
\]
while the quotient for $h<0$ gives the reverse inequality. Therefore
\begin{equation}
 f'(t)
 =
 \partial_t\mathcal F_{t,s}(\mu_t)
 =
 -\frac{\beta^2}{4}D_{t,s}(\mu_t)
 \qquad\text{for almost every }t.
 \label{eq:fn-envelope}
\end{equation}
In particular, the derivative has the same value for every minimizer
at every differentiability point of $f$.

At $t=0$, the vector PDE separates into $N$ scalar equations, giving
$f(0)=P(s)$. At $t=1$, $c=0$ and the value is independent of
$\mu$, giving $f(1)=\phi_N^c(s)$. Integrating
\eqref{eq:fn-envelope} yields 
\begin{equation}
 P(s)-\phi_N^c(s)
 =
 \frac{\beta^2}{4}
 \int_0^1D_{t,s}(\mu_t)\,\dd t
 \ge0.
 \label{eq:fn-endpoints}
\end{equation}
At almost every $t$, the integrand is determined by
$D_{t,s}(\mu_t)=-4f'(t)/\beta^2$, independently of the minimizer.
\end{proof}
\subsubsection{The case $s=0$}
At $s=0$, replace the normalized logarithmic moment by
\[
 \frac1N\mathbb E_WU_\mu(0,0;W).
\]
The outer change of measure then disappears. At fixed $N$, Gaussian
domination allows the fixed-profile identities above to pass to the
limit $s\downarrow0$.

To obtain the ordinary finite-size Parisi bound, fix an arbitrary
probability measure $\mu$ with cumulative function $a$, and for
$s>0$ set
\[
 \mu_s=s\delta_0+(1-s)\mu.
\]
Its cumulative function is $s+(1-s)a$. Fixed-profile monotonicity gives
\[
 \phi_N^c(s)
 \le
 \mathcal P\bigl(s+(1-s)a\bigr).
\]
As $s\downarrow0$,
Gaussian domination gives
\[
 \phi_N^c(s)\longrightarrow p_N,
\]
while \eqref{eq:fn-coefficient-continuity} gives
\[
 \mathcal P\bigl(s+(1-s)a\bigr)\longrightarrow\mathcal P(a).
\]
Hence  $p_N\le\mathcal P(a)$ for every admissible $a$. Taking the minimum over $a$ yields $
 p_N\le P$.

\section{Covariance estimates on quantile intervals}\zlabel{sec:residual}

The decomposition \eqref{eq:fn-defect} reduces the interpolation error
to two terms. The trace term is bounded by
\eqref{eq:fn-trace-error}. It remains to bound
\[
 \int \Var(L_q)\,\nu(\dd q).
\]
The estimate \eqref{eq:fn-contact-covariance} controls the variance of
an average of the variables $L_q$, rather than the average of their
individual variances. In this section we bridge that gap by comparing
$L_q$ at nearby overlap levels and averaging over intervals containing
a prescribed amount of $\nu$-mass. Parametrizing these intervals by
quantiles of $\nu$ makes their mass deterministic and also records their
diameter in the overlap variable $q$. Using only the bound that this
diameter is at most one will give the comparison in
\zcref{sec:quantitative}; later, \zcref{lem:p2-quantile} will give a
smaller diameter and hence a sharper low-temperature estimate.

Fix $\beta>0$, $0<s\le1/2$, and $0\le t<1$, and set
\[
 M=N\beta^2(1-t)>0.
\]
Let
\[
 \mu=s\delta_0+\nu,
 \qquad
 a(q)=\mu([0,q]),
\]
be a deterministic minimizer of $\mathcal F_{t,s}$. Expectations and
variances in this section are taken under the normalized joint law
constructed in \zcref{sec:finite-n}, with initial field $x_0=0$.
Thus $W$ denotes the coupling matrix under this normalized law,
$(X_q)_{0\le q\le1}$ is the auxiliary field process, and $\sigma$ is
the single terminal spin sampled from the terminal Gibbs law. For
\[
 \mathcal G_q=\sigma\!\left(W,(X_r)_{0\le r\le q}\right),
\]
define
\[
 m_q=\E[\sigma\mid\mathcal G_q],
 \qquad
 C_q=\Cov(\sigma\mid\mathcal G_q),
 \qquad
 S_q=\frac{|m_q|^2}{N},
 \qquad
 L_q=\frac{\sigma\cdot m_q-|m_q|^2/2}{N}.
\]

By \zcref{prop:fn-contact,prop:fn-integrated-variance}, there is a
deterministic set $\mathcal K\subset[0,1]$ of full $\nu$-measure
such that, for every $q\in\mathcal K$,
\begin{equation}\label{eq:cr-contact-inputs}
 \E S_q=q,\qquad
 \E\Tr C_q^2\le\frac{N^2}{M a(q)^2}.
\end{equation}
Moreover, for every deterministic $k\in L^2(\nu)$,
\begin{equation}\label{eq:cr-averaging-input}
 \Var\left(\int k(q)L_q\,\nu(\dd q)\right)
 \le
 \frac1M\int\frac{k(q)^2}{a(q)}\,\nu(\dd q).
\end{equation}
If $\nu$ has an atom at zero, then $0\in\mathcal K$ and the contact
estimate applies there. By contrast, the prescribed mass
$s\delta_0$ alone gives no variational condition at zero. The point
$q=1$ is not in the support of $\nu$.

We next identify the pointwise variance that must be estimated. Set
\[
 r_q=\Var(L_q),
 \qquad
 j_q=\frac1{N^2}\E[m_q^TC_qm_q].
\]
Since $|m_{q,i}|\le1$ and $\sigma_i\in\{-1,1\}$,
\[
 -\frac32
 \le
 \sigma_i m_{q,i}-\frac12m_{q,i}^2
 \le
 \frac12.
\]
Averaging over $i$ gives
\[
 -\frac32\le L_q\le\frac12,
\]
and hence
\[
 0\le r_q\le1
\]
by the elementary variance bound
$\Var(Y)\le(b-a)^2/4$ for $Y\in[a,b]$.

Conditional on $\mathcal G_q$,
\[
 \E[L_q\mid\mathcal G_q]=\frac{S_q}{2},
\]
because $\E[\sigma\mid\mathcal G_q]=m_q$, while
\[
 \Var(L_q\mid\mathcal G_q)
 =
 \frac1{N^2}m_q^TC_qm_q.
\]
The law of total variance therefore gives
\begin{equation}\label{eq:cr-profile-decomposition}
 r_q
 =
 j_q+\frac14\Var(S_q).
\end{equation}
At every contact level $q\in\mathcal K$,
\eqref{eq:cr-contact-inputs} also gives
\[
 \E L_q=\frac12\E S_q=\frac q2.
\]

Estimate \eqref{eq:cr-averaging-input} controls averages of $L_q$ over
sets of overlap levels, but our target is the integral of the
pointwise quantities $r_q$. To relate the two, we compare $L_q$ and
$L_r$ for $q<r$. Their increment consists of a centered term and a
nonnegative quadratic term. The sign of the latter gives one-sided
comparisons with averages over earlier and later overlap levels, while
the centered term can again be bounded using
\eqref{eq:cr-averaging-input}. We carry out these comparisons on
quantile intervals of $\nu$, so that the amount of mass available for
averaging is fixed in advance.

\begin{lemma}\zlabel{lem:cr-oriented-residual}
Let $p,q\in\mathcal K$ with $p\le q$, and set
\[
 \Delta_{p,q}=m_q-m_p.
\]
Then
\[
 L_q-L_p
 =
 \eta_{p,q}+\frac{|\Delta_{p,q}|^2}{2N},
 \qquad
 \eta_{p,q}
 =
 \frac{(\sigma-m_q)\cdot\Delta_{p,q}}{N}.
\]
The first term is centered,
\[
 \E\eta_{p,q}=0,
\]
while the second is nonnegative. Moreover,
\begin{equation}\label{eq:cr-residual-refinement}
 \E\eta_{p,q}^2
 \le
 \frac1{a(q)\sqrt M}
 \sqrt{(q-p)^2+8r_q+32r_p}.
\end{equation}
\end{lemma}

\begin{proof}
The purpose of the estimate is to control the centered part of the
increment by the pointwise variances $r_p$ and $r_q$. Set
\[
 d=q-p,\qquad
 A=\frac{|\Delta_{p,q}|^2}{N},\qquad
 B=\frac{2m_p\cdot\Delta_{p,q}}{N}.
\]
Since $(m_q)$ is a martingale,
\[
 \E[B\mid\mathcal G_p]=0,
 \qquad\text{and hence}\qquad
 \E B=0.
\]
Also
\[
 S_q-S_p=A+B.
\]
Because $p,q\in\mathcal K$, the contact identities give
$\E S_p=p$ and $\E S_q=q$. Therefore
\[
 \E A
 =
 \E(S_q-S_p)-\E B
 =
 q-p=d.
\]

We first bound the fluctuation of $B$. Conditional covariance of the
martingale increment gives
\[
 C_p
 =
 \E[C_q\mid\mathcal G_p]
 +
 \E[\Delta_{p,q}\Delta_{p,q}^T\mid\mathcal G_p].
\]
Hence
\[
 \begin{aligned}
 \E B^2
 &=
 \frac4{N^2}
 \E\!\left[
   m_p^T
   \E[\Delta_{p,q}\Delta_{p,q}^T\mid\mathcal G_p]
   m_p
 \right]\\
 &\le
 \frac4{N^2}\E[m_p^TC_pm_p]
 =4j_p
 \le4r_p.
 \end{aligned}
\]

Next, from $S_q-S_p=A+B$,
\[
 A-d
 =
 (S_q-q)-(S_p-p)-B.
\]
Using
\[
 (x-y-z)^2\le2x^2+4y^2+4z^2,
\]
together with
\[
 \Var(S_u)\le4r_u
 \qquad (u=p,q)
\]
from \eqref{eq:cr-profile-decomposition}, we obtain
\[
 \begin{aligned}
 \E(A-d)^2
 &\le
 2\Var(S_q)+4\Var(S_p)+4\E B^2\\
 &\le
 8r_q+16r_p+16r_p\\
 &=8r_q+32r_p.
 \end{aligned}
\]
Since $\E A=d$,
\begin{equation}\label{eq:cr-increment-fourth}
 \E A^2
 =
 d^2+\E(A-d)^2
 \le
 d^2+8r_q+32r_p.
\end{equation}

We now identify the increment of $L_q$. Since
$m_q=m_p+\Delta_{p,q}$,
\[
 \begin{aligned}
 L_q-L_p
 &=
 \frac{\sigma\cdot\Delta_{p,q}}{N}
 -\frac{|m_q|^2-|m_p|^2}{2N}\\
 &=
 \frac{(\sigma-m_q)\cdot\Delta_{p,q}}{N}
 +\frac{|\Delta_{p,q}|^2}{2N},
 \end{aligned}
\]
which is the asserted decomposition.

Because $\Delta_{p,q}$ is $\mathcal G_q$-measurable and
$\E[\sigma-m_q\mid\mathcal G_q]=0$,
\[
 \E\eta_{p,q}=0
\]
and
\[
 \E\eta_{p,q}^2
 =
 \frac1{N^2}
 \E[\Delta_{p,q}^TC_q\Delta_{p,q}].
\]
Frobenius Cauchy--Schwarz and then scalar Cauchy--Schwarz give
\[
 \E\eta_{p,q}^2
 \le
 \frac1{N^2}
 \bigl(\E\Tr C_q^2\bigr)^{1/2}
 \bigl(\E|\Delta_{p,q}|^4\bigr)^{1/2}.
\]
Since $q\in\mathcal K$, \eqref{eq:cr-contact-inputs} gives
\[
 \E\Tr C_q^2
 \le
 \frac{N^2}{M a(q)^2},
\]
while
\[
 |\Delta_{p,q}|^4=N^2A^2.
\]
Combining these estimates with
\eqref{eq:cr-increment-fourth} yields
\[
 \E\eta_{p,q}^2
 \le
 \frac1{a(q)\sqrt M}
 \sqrt{(q-p)^2+8r_q+32r_p},
\]
as claimed.
\end{proof}

Quantiles let us average over a prescribed amount of $\nu$-mass even
when $\nu$ has atoms. Let
\[
 Q(v)=\inf\{q\in[0,1]:\nu([0,q])\ge v\},
 \qquad 0<v<1-s,
\]
be the generalized quantile function of $\nu$. Then $Q$ pushes
Lebesgue measure on $(0,1-s)$ forward to $\nu$.

Let $I\subset(0,1-s)$ be an interval of length
\[
 x=|I|>0.
\]
The pushforward of Lebesgue measure restricted to $I$ is the measure
\[
 \nu_I(B)
 =
 \bigl|\{v\in I:Q(v)\in B\}\bigr|,
 \qquad B\subset[0,1]\ \text{Borel}.
\]
Since Lebesgue measure restricted to $I$ is dominated by Lebesgue
measure on $(0,1-s)$, we have $\nu_I\le\nu$. Hence
\[
 \nu_I=f_I\nu
\]
for some measurable function $0\le f_I\le1$, with
\[
 \int f_I\,\dd\nu=x.
\]
This description remains valid when $I$ contains only part of the
quantile interval corresponding to an atom of $\nu$; in that case
$f_I$ takes a value strictly between zero and one on that atom.

Suppose now that
\[
 a(Q(v))\ge u
 \qquad\text{for almost every }v\in I.
\]
Then
\[
 \frac1x\int_I L_{Q(v)}\,\dd v
 =
 \frac1x\int L_q f_I(q)\,\nu(\dd q).
\]
Applying \eqref{eq:cr-averaging-input} with
$k=f_I/x$ gives
\begin{equation}\label{eq:cr-quantile-average}
 \begin{aligned}
 \Var\left(\frac1x\int_I L_{Q(v)}\,\dd v\right)
 &\le
 \frac1{Mx^2}\int\frac{f_I(q)^2}{a(q)}\,\nu(\dd q)\\
 &\le
 \frac1{Mux^2}\int f_I(q)\,\nu(\dd q)
 =\frac1{Mux}.
 \end{aligned}
\end{equation}
Here we used $f_I^2\le f_I$ and the fact that
$f_I=0$ for $\nu$-almost every $q$ with $a(q)<u$.

Because $\nu([0,1]\setminus\mathcal K)=0$ and $Q$ pushes Lebesgue
measure on $(0,1-s)$ to $\nu$,
we also have
\[
 Q(v)\in\mathcal K
\]
for Lebesgue-almost every $v\in(0,1-s)$. We may therefore apply the
contact identities at $Q(v)$ throughout the quantile calculations
below, outside a null set.

We next convert the averaged variance bound
\eqref{eq:cr-quantile-average} into a pointwise variance estimate.
For a quantile interval $I$, define the diameter of its image in the
overlap variable by
\[
 \ell(I)
 =
 \operatorname*{ess\,sup}_{v,w\in I}|Q(v)-Q(w)|.
\]
Retaining $\ell(I)$ in the estimate will be important later: the
trivial bound $\ell(I)\le1$ is enough for the comparison in
\zcref{sec:quantitative}, while the sharper control supplied by
\zcref{lem:p2-quantile} yields the refined low-temperature comparison
in \zcref{prop:abs-selfconsistent-comparison}. Before deriving that
quantile estimate, we prove an auxiliary interval bound for the
square-root variance terms that arise from
\zcref{lem:cr-oriented-residual}.

\begin{lemma}
\zlabel{lem:cr-fractional-interval}
Let $J\subset\mathbb R$ be an interval of length $m>0$, and let
$f:J\to[0,\infty)$ be integrable, with
\[
 F=\int_J f(w)\,\dd w.
\]
Let $E\subset J$ be measurable. Suppose that, for almost every
$v\in E$, we are given measurably an interval $I_v\subset J$ of
positive length such that $v\in\overline{I_v}$. Then
\[
 \int_E
 \left(
   \frac1{|I_v|}\int_{I_v}f(w)\,\dd w
 \right)^{1/2}\dd v
 \le 2\sqrt{3mF}.
\]
\end{lemma}

\begin{proof}
Set
\[
 A(v)=\frac1{|I_v|}\int_{I_v}f(w)\,\dd w,
 \qquad v\in E.
\]
We first prove the weak estimate
\begin{equation}\label{eq:cr-fractional-weak}
 \bigl|\{v\in E:A(v)>z\}\bigr|
 \le \frac{3F}{z},
 \qquad z>0.
\end{equation}

Fix $z>0$. The endpoints of $J$ form a null set, so we may restrict
to points $v$ in the interior of $J$. If $A(v)>z$, then, because the
inequality is strict, $I_v$ may be enlarged slightly inside $J$ so
that $v$ lies in its interior and its average of $f$ remains greater
than $z$. These enlarged intervals form an open cover of the level set
$\{A>z\}$.

Let $K$ be a compact subset of this level set and choose a finite
subcover. From this finite collection, select an interval of maximal
length, delete every interval that intersects it, and repeat. The
selected intervals are pairwise disjoint. Moreover, every deleted
interval has length no greater than the selected interval that deleted
it. Since the two intervals intersect, the deleted interval is
contained in the threefold enlargement of the selected one. Hence
\[
 |K|
 \le 3\sum_{\text{selected }I}|I|.
\]
Each selected interval has average greater than $z$, so
\[
 z|I|<\int_I f.
\]
Because the selected intervals are disjoint,
\[
 z\sum_{\text{selected }I}|I|
 \le \sum_{\text{selected }I}\int_I f
 \le F.
\]
Therefore $|K|\le3F/z$. Taking the supremum over compact subsets of
$\{A>z\}$ proves \eqref{eq:cr-fractional-weak} by inner regularity.

We now integrate the level-set estimate. By the layer-cake formula,
\[
 \begin{aligned}
 \int_E\sqrt{A(v)}\,\dd v
 &=\int_0^\infty
   \bigl|\{v\in E:A(v)>y^2\}\bigr|\,\dd y\\
 &\le
 \int_0^\infty
 \min\left\{m,\frac{3F}{y^2}\right\}\dd y.
 \end{aligned}
\]
If $F>0$, splitting the last integral at
\[
 y_0=\sqrt{\frac{3F}{m}}
\]
gives
\[
 my_0+\frac{3F}{y_0}
 =2\sqrt{3mF}.
\]
If $F=0$, then $f=0$ almost everywhere on $J$, so $A(v)=0$ for almost
every $v\in E$ and the conclusion is immediate.
\end{proof}

Combining this interval estimate with the one-sided residual bounds
and \eqref{eq:cr-quantile-average} gives the desired dependence on
the quantile image's diameter.

\begin{lemma}\zlabel{lem:cr-refined-shell}
Let $J\subset(0,1-s)$ be a quantile interval of length
$0<m\le u\le1$, and suppose
\[
 a(Q(v))\ge u
 \qquad\text{for almost every }v\in J.
\]
Let
\[
 \ell
 =
 \operatorname*{ess\,sup}_{v,w\in J}|Q(v)-Q(w)|
 \le1
\]
be the diameter of its image in the overlap variable. Then there is a
universal constant $C<\infty$ such that
\begin{equation}\label{eq:cr-shell-bound}
 \int_J r_{Q(v)}\,\dd v
 \le
 C\left[
   M^{-3/4}u^{-1/2}
   +M^{-2/3}u^{-1/3}\ell^{2/3}
   +\frac1{Mu}
 \right].
\end{equation}
\end{lemma}

\begin{proof}
Write
\[
 r(v)=r_{Q(v)},\qquad
 \varepsilon=\frac1{Mu},\qquad
 \alpha=\frac1{u\sqrt M},\qquad
 \mathcal R=\int_Jr(v)\,\dd v.
\]
The parameter $\varepsilon$ is the variance scale for averages over
quantile intervals, by \eqref{eq:cr-quantile-average}, while
$\alpha$ is the coefficient in the residual estimate
\eqref{eq:cr-residual-refinement} when $a\ge u$.

If $\ell=0$, then $Q$ is constant almost everywhere on $J$, so
$L_{Q(v)}$ is the same random variable for almost every $v\in J$.
Applying \eqref{eq:cr-quantile-average} to the whole interval gives
\[
 \mathcal R\le\varepsilon.
\]
We also have $\mathcal R\le m$ because $r\le1$. Hence
\[
 \mathcal R
 \le\min\{m,\varepsilon\}
 \le m^{1/4}\varepsilon^{3/4},
\]
which proves both \eqref{eq:cr-shell-bound} and the finer estimate
\eqref{eq:cr-fine-shell} below. We therefore assume $\ell>0$.

\medskip
\emph{Step 1: Compare a point with averages on either side.}
Fix an integer $K\ge1$. Partition the smallest interval containing
$Q(J)$ up to null sets into $K$ overlap bins of width
\[
 h=\ell/K.
\]
Use half-open bins, closing only the final bin at its right endpoint,
so that an atom lying on a bin boundary is assigned to a unique bin.
Since $Q$ is nondecreasing, the inverse image in $J$ of every
positive-mass bin is an interval.

Fix a point $v$ in the interior of one such inverse-image interval
$(b_-,b_+)$. Set
\[
 x_-=v-b_-,
 \qquad
 x_+=b_+-v,
\]
and
\[
 I_-=(b_-,v),
 \qquad
 I_+=(v,b_+).
\]
The endpoints arising from all integer values of $K$ form a null set,
which we exclude throughout.

Let
\[
 q=Q(v),
 \qquad
 Z_v=L_q-\frac q2.
\]
For the averages to the left and right of $v$, define
\[
 \overline Z_\pm
 =
 \frac1{x_\pm}\int_{I_\pm}
 \left(L_{Q(w)}-\frac{Q(w)}2\right)\dd w.
\]
We also average the centered residuals from
\zcref{lem:cr-oriented-residual}:
\[
 \overline\eta_-
 =
 \frac1{x_-}\int_{I_-}\eta_{Q(w),q}\,\dd w,
 \qquad
 \overline\eta_+
 =
 \frac1{x_+}\int_{I_+}\eta_{q,Q(w)}\,\dd w.
\]

For $w\in I_-$, the nonnegative quadratic term in
\zcref{lem:cr-oriented-residual} gives a lower bound on $L_q-L_{Q(w)}$;
for $w\in I_+$ it gives the corresponding upper bound.
Since all overlap values in one bin differ by at most $h$, averaging
these inequalities gives
\[
 Z_v
 \ge
 \overline Z_-+\overline\eta_- -\frac h2,
 \qquad
 Z_v
 \le
 \overline Z_+-\overline\eta_+ +\frac h2.
\]
At contact points, $\E Z_v=0$. Hence the negative and positive parts
of $Z_v$ satisfy
\[
 (Z_v)_-
 \le
 |\overline Z_-|+|\overline\eta_-|+\frac h2,
 \qquad
 (Z_v)_+
 \le
 |\overline Z_+|+|\overline\eta_+|+\frac h2.
\]
Squaring, taking expectations, and adding the two inequalities yields
\begin{equation}
 r(v)
 \le
 3\varepsilon(x_-^{-1}+x_+^{-1})
 +\frac32h^2
 +3\E\overline\eta_-^2
 +3\E\overline\eta_+^2.
 \label{eq:cr-bin-raw}
\end{equation}
Indeed, \eqref{eq:cr-quantile-average} gives
\[
 \E\overline Z_\pm^2
 \le\frac{\varepsilon}{x_\pm}.
\]

We next estimate the residual terms. Since $a(Q(w))\ge u$ on $J$,
\zcref{lem:cr-oriented-residual} and Jensen's inequality give
\[
 \begin{aligned}
 \E\overline\eta_-^2
 &\le
 \alpha\left[
   h+\sqrt8\,\sqrt{r(v)}
      +\sqrt{32}\,\sqrt{R_-(v)}
 \right],\\
 \E\overline\eta_+^2
 &\le
 \alpha\left[
   h+\sqrt8\,\sqrt{R_+(v)}
      +\sqrt{32}\,\sqrt{r(v)}
 \right],
 \end{aligned}
\]
where
\[
 R_\pm(v)
 =
 \frac1{x_\pm}\int_{I_\pm}r(w)\,\dd w.
\]
For example, on $I_-$ we use
\[
 \sqrt{h^2+8r(v)+32r(w)}
 \le
 h+\sqrt8\,\sqrt{r(v)}+\sqrt{32}\,\sqrt{r(w)},
\]
and then
\[
 \frac1{x_-}\int_{I_-}\sqrt{r(w)}\,\dd w
 \le\sqrt{R_-(v)}.
\]

Substituting these estimates into \eqref{eq:cr-bin-raw} gives
\begin{align}
 r(v)
 &\le
 d_K(v)
 +\alpha\left[
   18\sqrt2\,\sqrt{r(v)}
   +12\sqrt2\,\sqrt{R_-(v)}
   +6\sqrt2\,\sqrt{R_+(v)}
 \right],
 \label{eq:cr-scale-pairing}\\
 d_K(v)
 &=
 3\varepsilon(x_-^{-1}+x_+^{-1})
 +\frac32(\ell/K)^2
 +6\alpha\ell/K.
 \label{eq:cr-scale-envelope}
\end{align}
The three terms in $d_K$ have different origins: the first is the
variance cost of averaging over intervals of masses $x_\pm$, the
second is the change of the deterministic mean $q/2$ across one
overlap bin, and the third comes from the $h$ term in the residual
estimate.

\medskip
\emph{Step 2: Choose the bin scale pointwise and integrate the residuals.}
For every $v$ outside the common null set, define
\[
 d(v)=\inf_{K\ge1}d_K(v),
 \qquad
 e(v)=\min\{1,d(v)\}.
\]
Since $x_-+x_+\le m$,
\[
 x_-^{-1}+x_+^{-1}
 \ge\frac4{x_-+x_+}
 \ge\frac4m,
\]
and therefore
\[
 d_K(v)\ge\frac{12\varepsilon}{m}>0.
\]
Also $d_1(v)<\infty$ almost everywhere.

On the set
\[
 E=\{v\in J:d(v)<1\},
\]
choose the smallest integer $K(v)$ satisfying
\[
 d_{K(v)}(v)\le2d(v).
\]
This choice is measurable and depends only on the deterministic
quantile geometry. Apply \eqref{eq:cr-scale-pairing} using the two
intervals determined by this selected scale. On $J\setminus E$, use
only $r(v)\le1=e(v)$. Integrating gives
\[
 \begin{aligned}
 \mathcal R
 \le
 2\int_Je(v)\,\dd v
 +\alpha\biggl[
   18\sqrt2\int_J\sqrt{r(v)}\,\dd v
   +12\sqrt2\int_E\sqrt{R_-(v)}\,\dd v
   +6\sqrt2\int_E\sqrt{R_+(v)}\,\dd v
 \biggr].
 \end{aligned}
\]

Cauchy--Schwarz gives
\[
 \int_J\sqrt r\le\sqrt{m\mathcal R}.
\]
For each $v\in E$, the selected intervals $I_\pm(v)$ contain $v$ in
their closure. Applying \zcref{lem:cr-fractional-interval} with
$f=r$ therefore gives
\[
 \int_E\sqrt{R_\pm(v)}\,\dd v
 \le2\sqrt{3m\mathcal R}.
\]
Thus, with
\[
 B=18\sqrt2(1+2\sqrt3),
\]
we obtain
\[
 \mathcal R
 \le
 2\int_Je(v)\,\dd v
 +B\alpha\sqrt{m\mathcal R}.
\]
Using
\[
 B\alpha\sqrt{m\mathcal R}
 \le
 \frac{\mathcal R}{2}
 +\frac{B^2m\alpha^2}{2}
\]
and moving $\mathcal R/2$ to the left gives
\begin{equation}
 \mathcal R
 \le
 4\int_Je(v)\,\dd v+B^2m\alpha^2.
 \label{eq:cr-direct-l1}
\end{equation}
It remains to estimate the deterministic integral of $e$.

\medskip
\emph{Step 3: Bound the deterministic envelope.}
Fix $0<z<1$. We choose a single bin count $K_z$ for which both
deterministic bin-width terms in \eqref{eq:cr-scale-envelope} are at
most $z/8$. We also exclude short strips near the two endpoints of
each inverse-image interval so that the averaging term is at most
$z/8$.

These requirements are satisfied by
\[
 K_z
 =
 \left\lceil
 \max\left\{
   1,\sqrt{12}\,\ell z^{-1/2},
   48\alpha\ell z^{-1}
 \right\}
 \right\rceil,
 \qquad
 b_z=\frac{48\varepsilon}{z}.
\]
Indeed, away from endpoint strips of width $b_z$ we have
$x_-,x_+\ge b_z$, and hence
\[
 3\varepsilon(x_-^{-1}+x_+^{-1})\le\frac z8.
\]
The choices in $K_z$ similarly give
\[
 \frac32(\ell/K_z)^2\le\frac z8,
 \qquad
 6\alpha\ell/K_z\le\frac z8.
\]
Thus
\[
 d_{K_z}(v)\le\frac{3z}{8}<z
\]
outside the endpoint strips. Consequently $\{e>z\}$ is contained in
their union.

There are at most $K_z$ nonempty inverse-image intervals, and each
contributes at most two endpoint strips of length $b_z$. Therefore
\[
 |\{e>z\}|
 \le
 \min\left\{m,2K_zb_z\right\}.
\]
Since
\[
 K_z
 \le
 1+\sqrt{12}\,\ell z^{-1/2}
   +48\alpha\ell z^{-1},
\]
we obtain
\begin{equation}
 |\{e>z\}|
 \le
 \min\left\{
   m,\,
   z_0z^{-1}+z_1z^{-3/2}+z_2z^{-2}
 \right\},
 \label{eq:cr-envelope-tail}
\end{equation}
where
\[
 z_0=96\varepsilon,\qquad
 z_1=96\sqrt{12}\,\varepsilon\ell,\qquad
 z_2=4608\varepsilon\alpha\ell.
\]

We now integrate this level-set estimate. Since $0\le e\le1$,
\[
 \int_Je(v)\,\dd v
 =
 \int_0^1|\{e>z\}|\,\dd z.
\]
Use
\[
 \min\{m,A+B+D\}
 \le
 \min\{m,A\}+\min\{m,B\}+\min\{m,D\}.
\]
For the first contribution,
\[
 \min\{m,z_0z^{-1}\}
 \le
 m^{1/4}z_0^{3/4}z^{-3/4},
\]
whose integral over $(0,1)$ is at most
$4m^{1/4}z_0^{3/4}$. For $k>1$,
\[
 \int_0^\infty\min\{m,bz^{-k}\}\,\dd z
 =
 \frac{k}{k-1}b^{1/k}m^{1-1/k}.
\]
Applying this with $k=3/2$ and $k=2$ to the other two terms gives
\[
 \int_Je(v)\,\dd v
 \le
 C\left[
   m^{1/4}\varepsilon^{3/4}
   +(\varepsilon\ell)^{2/3}m^{1/3}
   +(\varepsilon\alpha\ell)^{1/2}m^{1/2}
 \right].
\]
Combining this with \eqref{eq:cr-direct-l1} proves the finer estimate
\begin{equation}
 \int_Jr(v)\,\dd v
 \le
 C\left[
   m^{1/4}\varepsilon^{3/4}
   +(\varepsilon\ell)^{2/3}m^{1/3}
   +(\varepsilon\alpha\ell)^{1/2}m^{1/2}
   +m\alpha^2
 \right].
 \label{eq:cr-fine-shell}
\end{equation}

Finally substitute
\[
 \varepsilon=(Mu)^{-1},
 \qquad
 \alpha=(u\sqrt M)^{-1},
\]
and use $m\le u$ and $\ell\le1$. The four terms on the right of
\eqref{eq:cr-fine-shell} are bounded respectively by
\[
 M^{-3/4}u^{-1/2},\qquad
 M^{-2/3}u^{-1/3}\ell^{2/3},\qquad
 M^{-3/4}u^{-1/2},\qquad
 (Mu)^{-1}.
\]
This proves \eqref{eq:cr-shell-bound}. 
\end{proof}

\section{Quantitative Parisi formulas}\zlabel{sec:quantitative}
In this section we turn the interpolation estimates from the preceding
sections into quantitative finite-size Parisi formulas. At positive
temperature, the derivative of the optimized interpolation is controlled
by the overlap defect. We combine the refined quantile-interval estimate  with a cutoff
adapted to \(N\beta^2(1-t)\), and then integrate over the interpolation
parameter to prove \zcref{thm:uniform-parisi}. The resulting uniformity in
\(s\) also permits the zero-temperature scaling \(s=\lambda/\beta\), which
yields the ground-state comparison in \zcref{thm:ground-state}.

\subsection{A cutoff adapted to the interpolation parameter}
\zlabel{subsec:moving-cutoff}

The bound from \zcref{lem:cr-refined-shell} deteriorates when the
cumulative mass $a(q)$ is small. We therefore discard a short initial
quantile interval and apply the quantile-interval estimate only after the
cumulative mass has reached a threshold $u_*$. The threshold is chosen
as a function of
\[
 M=N\beta^2(1-t),
\]
so that the mass discarded near zero and the largest remaining shell
error have the same order. This choice gives a bound of order
$M^{-1/2}$ without a logarithmic loss when the interpolation parameter
$t$ is later integrated.

\begin{lemma}\zlabel{lem:moving-cutoff}
Fix $N\ge1$, $\beta>0$, $0<s\le1/2$, and $0\le t<1$, and set
\[
 c=\beta^2(1-t),\qquad M=Nc.
\]
For every deterministic minimizer $\mu=s\delta_0+\nu$ of
$\mathcal F_{t,s}$,
\begin{equation}\label{eq:moving-defect}
 D_{t,s}(\mu)
 \le C\min\{1,M^{-1/2}\}
 =C\min\{1,(Nc)^{-1/2}\},
\end{equation}
where $C>0$ is a numerical constant.
\end{lemma}

\begin{proof}
Write, as in \eqref{eq:fn-defect},
\[
 D_{t,s}(\mu)\le A+4\mathcal W,
\]
where
\[
 A=\frac1{N^2}\int\E\Tr C_q^2\,\nu(\dd q),
 \qquad
 \mathcal W=\int r_q\,\nu(\dd q),
 \qquad
 r_q=\Var(L_q).
\]
The term $A$ measures the conditional-covariance contribution to the
overlap error, while $\mathcal W$ is the integrated pointwise variance
estimated by the quantile argument.

If $M\le1$, then the elementary bound
$D_{t,s}(\mu)\le4(1-s)\le4$ proves the result. We therefore assume
$M>1$.

We first bound $A$. Let
\[
 Q(v)=\inf\{q:\nu([0,q])\ge v\},
 \qquad 0<v<1-s,
\]
be the quantile function of $\nu$. On the support of $\nu$,
\zcref{prop:fn-integrated-variance} and the trivial bound
$\Tr C_q^2\le N^2$ give
\[
 \frac1{N^2}\E\Tr C_q^2
 \le
 \min\left\{1,\frac1{Ma(q)^2}\right\}.
\]
Since
\[
 a(Q(v))\ge s+v,
\]
we obtain
\[
 \begin{aligned}
 A
 &\le
 \int_0^{1-s}
 \min\left\{1,\frac1{M(s+v)^2}\right\}\dd v\\
 &\le
 \int_0^1
 \min\left\{1,\frac1{Mu^2}\right\}\dd u\\
 &=2M^{-1/2}-M^{-1}
 \le2M^{-1/2}.
 \end{aligned}
\]

It remains to bound $\mathcal W$. If we discard all quantile mass
below cumulative level $u$, the discarded contribution is at most
$u$, because $r_q\le1$. On the remaining mass, the last term in
\eqref{eq:cr-shell-bound} is of order $(Mu)^{-1}$. We therefore choose
$u$ so that
\[
 u\asymp(Mu)^{-1},
\]
namely
\[
 u_*=M^{-1/2}.
\]
Because the cumulative mass is already at least $s$, set
\[
 u_0=\max\{s,u_*\}.
\]
Since $M>1$ and $s\le1/2$, we have $u_0<1$.

The initial quantile interval
\[
 \{v: s+v<u_0\}
\]
has length
\[
 (u_0-s)_+\le u_*=M^{-1/2}.
\]
Using only $r_q\le1$, its contribution to $\mathcal W$ is therefore at
most $M^{-1/2}$.

We partition the remaining quantile mass into dyadic levels. For
integers $j\ge0$ with
\[
 u_j=2^ju_0<1,
\]
set
\[
 J_j
 =
 \left\{
 v\in(0,1-s):
 u_j\le s+v<\min\{2u_j,1\}
 \right\}.
\]
Every nonempty $J_j$ has length at most $u_j$, and
\[
 a(Q(v))\ge u_j
 \qquad (v\in J_j).
\]
Apply \zcref{lem:cr-refined-shell} to $J_j$, using only the trivial
diameter bound $\ell\le1$. This gives
\[
 \int_{J_j}r_{Q(v)}\,\dd v
 \le
 C\left[
 M^{-3/4}u_j^{-1/2}
 +M^{-2/3}u_j^{-1/3}
 +(Mu_j)^{-1}
 \right].
\]
Each term decreases geometrically as $j$ increases. Summing over $j$
therefore gives
\[
 \mathcal W
 \le
 M^{-1/2}
 +C\left[
 M^{-3/4}u_0^{-1/2}
 +M^{-2/3}u_0^{-1/3}
 +(Mu_0)^{-1}
 \right].
\]
Because $u_0\ge M^{-1/2}$,
\[
 \begin{aligned}
 M^{-3/4}u_0^{-1/2}&\le M^{-1/2},\\
 M^{-2/3}u_0^{-1/3}&\le M^{-1/2},\\
 (Mu_0)^{-1}&\le M^{-1/2}.
 \end{aligned}
\]
Hence
\[
 \mathcal W\le C M^{-1/2}.
\]

Combining
\[
 A\le2M^{-1/2},
 \qquad
 \mathcal W\le CM^{-1/2},
\]
with $D_{t,s}(\mu)\le A+4\mathcal W$ proves
\eqref{eq:moving-defect}.
\end{proof}

\begin{proof}[Proof of \zcref{thm:uniform-parisi}]
Fix $0<s\le1/2$. We first bound the overlap defect at a fixed
interpolation parameter. Write
\[
 c=\beta^2(1-t)>0.
\]
At every differentiability point of the optimized value $f$, define
\[
 D(c)=-\frac{4}{\beta^2}f'(t),
 \qquad t=1-\frac{c}{\beta^2}.
\]
By \zcref{prop:fn-envelope}, this equals $D_{t,s}(\mu_t)$ for every
minimizer $\mu_t$ at that interpolation parameter.

We first retain the dependence on the fractional-moment parameter
$s$. Let $Q$ be the quantile function of the remaining measure
$\nu$, and divide the quantile variable into the dyadic intervals
\[
 J_j
 =
 \left\{
 v:2^js\le s+v<\min\{2^{j+1}s,1\}
 \right\},
 \qquad j\ge0.
\]
Each nonempty $J_j$ has length at most $2^js$ and satisfies
$a(Q(v))\ge2^js$. Applying
\zcref{lem:cr-refined-shell} with the trivial diameter bound
$\ell\le1$, summing the resulting geometric series, and adding the
trace estimate \eqref{eq:fn-trace-error} through
\eqref{eq:fn-defect} gives
\[
 D(c)
 \le
 C\left[
   (Nc)^{-3/4}s^{-1/2}
   +(Nc)^{-2/3}s^{-1/3}
   +(Ncs)^{-1}
 \right].
\]

We compare these three terms according to the size of $Ncs^2$.
If
\[
 Ncs^2\ge1,
\]
then, relative to the middle term,
\[
 \frac{(Nc)^{-3/4}s^{-1/2}}
      {(Nc)^{-2/3}s^{-1/3}}
 =(Ncs^2)^{-1/12}\le1,
\]
and
\[
 \frac{(Ncs)^{-1}}
      {(Nc)^{-2/3}s^{-1/3}}
 =(Ncs^2)^{-1/3}\le1.
\]
Hence
\[
 D(c)\le C(Nc)^{-2/3}s^{-1/3}.
\]

If instead
\[
 Ncs^2<1,
\]
we use \zcref{lem:moving-cutoff}, which gives
\[
 D(c)\le C\min\{1,(Nc)^{-1/2}\}.
\]
Moreover,
\[
 \frac{(Nc)^{-1/2}}
      {(Nc)^{-2/3}s^{-1/3}}
 =(Ncs^2)^{1/6}<1.
\]
Combining the two cases with the trivial bound $D(c)\le C$ yields
\begin{equation}\label{eq:full-defect}
 D(c)
 \le
 C\min\left\{
   1,\,
   (Nc)^{-1/2},\,
   (Nc)^{-2/3}s^{-1/3}
 \right\}.
\end{equation}

We now integrate this estimate over the interpolation. By
\eqref{eq:fn-endpoints} and the change of variables
$c=\beta^2(1-t)$,
\[
 0\le
 P_\beta(s)-\phi^c_{N,\beta}(s)
 =
 \frac14\int_0^{\beta^2}D(c)\,\dd c.
\]

The three terms in \eqref{eq:full-defect} give three separate bounds:
\[
 \int_0^{\beta^2}1\,\dd c=\beta^2,
\]
\[
 \int_0^{\beta^2}(Nc)^{-1/2}\,\dd c
 =
 2\beta N^{-1/2},
\]
and
\[
 \int_0^{\beta^2}
 (Nc)^{-2/3}s^{-1/3}\,\dd c
 =
 3\beta^{2/3}N^{-2/3}s^{-1/3}.
\]
Therefore
\[
 0\le
 P_\beta(s)-\phi^c_{N,\beta}(s)
 \le
 C\min\left\{
   \beta^2,\,
   \beta N^{-1/2},\,
   \beta^{2/3}N^{-2/3}s^{-1/3}
 \right\}.
\]
This argument also covers the regime $\beta^2<1/N$; in that case the
first term provides the useful bound.

It remains to pass to $s=0$. At fixed $N$ and $\beta$, Gaussian
domination gives
\[
 \phi^c_{N,\beta}(s)\longrightarrow p_{N,\beta}
 \qquad\text{as }s\downarrow0.
\]
We also have
\[
 P_\beta(s)\longrightarrow P_\beta.
\]
Indeed, the constraint defining $P_\beta(s)$ only shrinks the
admissible class, so
\[
 P_\beta\le P_\beta(s).
\]
If $\mu$ is an unconstrained minimizer and
\[
 \mu_s=s\delta_0+(1-s)\mu,
\]
then $\mu_s$ is admissible for $P_\beta(s)$ and its cumulative
function converges in $L^1$ to that of $\mu$. By
\eqref{eq:fn-coefficient-continuity}, including the deterministic
penalty,
\[
 P_\beta(s)
 \le \mathcal P_\beta(\mu_s)
 \longrightarrow
 \mathcal P_\beta(\mu)=P_\beta.
\]
Thus $P_\beta(s)\to P_\beta$.

Letting $s\downarrow0$ in the positive-$s$ comparison therefore gives
\[
 0\le P_\beta-p_{N,\beta}
 \le
 C\min\{\beta^2,\beta N^{-1/2}\},
\]
which is the asserted expected-pressure bound.
\end{proof}

\subsection{Zero temperature}\zlabel{subsec:sec-zero-temperature}

We first define the zero-temperature scalar PDE value for every
$\gamma\in\mathcal M_\lambda$, including profiles that are integrable
but unbounded. We use the stochastic-control formulation of the Parisi
PDE \cite{AuffingerChen2015Uniqueness,Chen2017Variational}. The control
representation is convenient here because it requires only
integrability of $\gamma$ and remains meaningful for the nonsmooth
terminal function $|x|$.

Let $B$ be standard Brownian motion. For $0\le q\le1$ and
$x\in\mathbb R$, define
\begin{equation}\label{eq:zero-control}
 \Psi_\gamma(q,x)
 =
 \sup_{b}
 \E\left[
 \left|
 x+B_1-B_q+\int_q^1\gamma(r)b_r\,\dd r
 \right|
 -\frac12\int_q^1\gamma(r)b_r^2\,\dd r
 \right],
\end{equation}
where the supremum is over processes $b$ satisfying
$|b_r|\le1$ almost everywhere and predictable with respect to the
augmented natural filtration of $(B_r-B_q)_{q\le r\le1}$.
The value is finite because
\[
 \left|\int_q^1\gamma(r)b_r\,\dd r\right|
 \le\int_q^1\gamma(r)\,\dd r<\infty.
\]

For bounded step profiles, \eqref{eq:zero-control} agrees with the
usual Cole--Hopf solution of the zero-temperature Parisi PDE. To see
why, first replace the terminal function $|x|$ by a smooth
$1$-Lipschitz function $f$. Let $\Psi_{\gamma,f}$ denote the
corresponding PDE solution. The maximum principle gives
\[
 |\partial_x\Psi_{\gamma,f}|\le1.
\]
For any admissible control $b$, let
\[
 \dd Y_r=\dd B_r+\gamma(r)b_r\,\dd r,
 \qquad Y_q=x.
\]
It\^o's formula gives
\[
 \begin{aligned}
 &\E\left[
 f(Y_1)-\frac12\int_q^1\gamma(r)b_r^2\,\dd r
 \right]
 -\Psi_{\gamma,f}(q,x)\\
 &\qquad
 =-\frac12
 \E\int_q^1
 \gamma(r)
 \bigl(b_r-\partial_x\Psi_{\gamma,f}(r,Y_r)\bigr)^2\,\dd r.
 \end{aligned}
\]
Thus every control gives a value at most
$\Psi_{\gamma,f}(q,x)$, and equality is attained by the feedback
control
\[
 b_r=\partial_x\Psi_{\gamma,f}(r,Y_r).
\]
Approximating $|x|$ uniformly by smooth $1$-Lipschitz functions
therefore proves \eqref{eq:zero-control} for the nonsmooth terminal
condition.

We next record continuity in the profile. If
$\gamma,\widetilde\gamma\in\mathcal M_\lambda$, evaluate the two
control problems using the same admissible control $b$. Since
$|b|\le1$ and the terminal function is $1$-Lipschitz,
\[
 \left|
 \Psi_\gamma(q,x)-\Psi_{\widetilde\gamma}(q,x)
 \right|
 \le
 \frac32\int_q^1
 |\gamma(r)-\widetilde\gamma(r)|\,\dd r.
\]
Taking the supremum over $q,x$ gives
\begin{equation}\label{eq:zero-continuity}
 \sup_{q,x}
 |\Psi_\gamma(q,x)-\Psi_{\widetilde\gamma}(q,x)|
 \le
 \frac32\|\gamma-\widetilde\gamma\|_{L^1}.
\end{equation}

For bounded profiles, step approximation together with
\eqref{eq:zero-continuity} therefore gives the same value as
\eqref{eq:zero-control}. For a general integrable profile, truncate at
height $R$:
\[
 \gamma_R=\gamma\wedge R.
\]
Since $\gamma_R\to\gamma$ in $L^1$, the values
$\Psi_{\gamma_R}$ converge uniformly to $\Psi_\gamma$.
Thus the control representation, step approximation, and truncation
all define the same zero-temperature value.

The zero-temperature Parisi functional consists of
$\Psi_\gamma(0,0)$ and the penalty
\[
 -\frac12\int_0^1q\gamma(q)\,\dd q.
\]
The first term is $3/2$-Lipschitz in $L^1$ by
\eqref{eq:zero-continuity}, and the penalty is $1/2$-Lipschitz.
Hence the full functional is $2$-Lipschitz in $L^1$. In particular,
its infimum is unchanged if we restrict to bounded profiles in
$\mathcal M_\lambda$.

The infimum is finite. A bounded constant profile gives a finite upper
bound. For the lower bound, use the admissible control $b_r=1$ in
\eqref{eq:zero-control}. Jensen's inequality gives
\[
 \Psi_\gamma(0,0)
 \ge
 \frac12\int_0^1\gamma(q)\,\dd q,
\]
and therefore
\[
 \Psi_\gamma(0,0)
 -\frac12\int_0^1q\gamma(q)\,\dd q
 \ge
 \frac12\int_0^1(1-q)\gamma(q)\,\dd q
 \ge0.
\]

We now identify this variational problem as the zero-temperature limit
of the positive-temperature constrained Parisi formula.

\begin{lemma}\zlabel{lem:zero-limit}
For every $\lambda\ge0$,
\begin{equation}\label{eq:zero-limit}
 \lim_{\beta\to\infty}
 \frac{P_\beta(\lambda/\beta)}{\beta}
 =
 \mathcal E(\lambda).
\end{equation}
\end{lemma}

\begin{proof}
Fix $\lambda\ge0$ and take $\beta\ge\lambda$. If $a$ is an admissible
positive-temperature cumulative profile, set
\[
 \gamma=\beta a,
 \qquad
 \Psi_{\gamma,\beta}(q,x)
 =
 \frac1\beta u_a(q,\beta x).
\]
Then $\Psi_{\gamma,\beta}$ satisfies
\[
 \partial_q\Psi_{\gamma,\beta}
 =
 -\frac12\left(
 \Psi_{\gamma,\beta,xx}
 +\gamma(q)\Psi_{\gamma,\beta,x}^2
 \right),
\]
with terminal condition
\[
 \Psi_{\gamma,\beta}(1,x)
 =
 \frac1\beta\log(2\cosh(\beta x)).
\]
The elementary bounds
\[
 |x|
 \le
 \frac1\beta\log(2\cosh(\beta x))
 \le
 |x|+\frac{\log2}{\beta}
\]
hold uniformly in $x$. By the control representation, equivalently by
the comparison principle for bounded step profiles followed by
approximation,
\[
 0
 \le
 \Psi_{\gamma,\beta}(0,0)-\Psi_\gamma(0,0)
 \le
 \frac{\log2}{\beta}.
\]
The bound is uniform over all admissible $\gamma$.

After the rescaling $\gamma=\beta a$, the constraint
$a(0)\ge\lambda/\beta$ becomes
\[
 \gamma(q)\ge\lambda,
\]
and the positive-temperature bound $a\le1$ becomes
$\gamma\le\beta$. Thus
$P_\beta(\lambda/\beta)/\beta$ is, up to an error at most
$\log2/\beta$, the minimum of the zero-temperature functional over
nondecreasing right-continuous profiles satisfying
\[
 \lambda\le\gamma\le\beta
 \qquad\text{on }[0,1).
\]
Any remaining endpoint mass at $q=1$ enforces the positive-temperature
normalization $a(1)=1$ but does not change the functional.

As $\beta$ increases, these admissible classes increase to the set of
all bounded profiles in $\mathcal M_\lambda$. Their infima therefore
decrease to the infimum over bounded profiles. By the truncation
argument above, that infimum is exactly $\mathcal E(\lambda)$.
Together with the uniform terminal error
$\log2/\beta$, this proves the desired conclusion.
\end{proof}

\begin{proof}[Proof of \zcref{thm:ground-state}]
We first pass the expected-pressure comparison to zero temperature.
For every realization of the completed Hamiltonian,
\begin{equation}\label{eq:soft-maximum}
 G_N^c
 \le
 \frac{F^c_{N,\beta}}{\beta}
 \le
 G_N^c+\frac{N\log2}{\beta}.
\end{equation}
Indeed, the largest Boltzmann weight is bounded by the full partition
sum, while the latter is at most $2^N$ times its largest term.
Consequently
\[
 \frac1N\E G_N^c
 \le
 \frac1{N\beta}\E F^c_{N,\beta}
 \le
 \frac1N\E G_N^c+\frac{\log2}{\beta}.
\]

Apply the expected-pressure estimate
\eqref{eq:quantitative-pressure} and divide by $\beta$. By
\zcref{lem:zero-limit} with $\lambda=0$,
\[
 \frac{P_\beta}{\beta}\longrightarrow\mathcal E(0),
\]
while \eqref{eq:soft-maximum} gives
\[
 \frac1{N\beta}\E F^c_{N,\beta}
 \longrightarrow
 \frac1N\E G_N^c.
\]
Thus the finite-temperature comparison passes directly to the
completed ground-state mean. Since
\[
 G_N^c=G_N+D_N
\]
and $\E D_N=0$, the completed and off-diagonal ground-state means are
equal. This proves \eqref{eq:ground-mean}.

We next pass the positive fractional-moment comparison to the
ground-state log-Laplace transform. Fix $N\ge1$ and $\lambda>0$, and
for $\beta\ge2\lambda$ set
\[
 s=\lambda/\beta\le1/2.
\]
Then
\[
 \frac{\phi^c_{N,\beta}(\lambda/\beta)}{\beta}
 =
 \frac1{N\lambda}
 \log\E\exp\left\{
   \lambda\frac{F^c_{N,\beta}}{\beta}
 \right\}.
\]
Multiplying \eqref{eq:soft-maximum} by $\lambda$ and exponentiating
shows that
\[
 \E e^{\lambda G_N^c}
 \le
 \E\exp\left\{
   \lambda\frac{F^c_{N,\beta}}{\beta}
 \right\}
 \le
 e^{N\lambda\log2/\beta}\E e^{\lambda G_N^c}.
\]
Hence
\[
 0
 \le
 \frac{\phi^c_{N,\beta}(\lambda/\beta)}{\beta}
 -
 \frac1{N\lambda}\log\E e^{\lambda G_N^c}
 \le
 \frac{\log2}{\beta}.
\]
All of these exponential moments are finite because $G_N^c$ is the
maximum of finitely many Gaussian affine functions.

Apply \eqref{eq:uniform-parisi} with $s=\lambda/\beta$, divide by
$\beta$, and let $\beta\to\infty$. The scalar variational term
converges by \zcref{lem:zero-limit}, while the finite-system term
converges by the preceding display. This gives
\eqref{eq:ground-laplace}. Because the constant in
\eqref{eq:uniform-parisi} is numerical and independent of $\beta$,
the same constant remains valid after taking the limit, uniformly in
$N$ and $\lambda$.

Finally, remove the diagonal completion. Since
\[
 G_N^c=G_N+D_N,
\]
with $D_N$ independent of $G_N$ and
$D_N\sim N(0,1/2)$,
\[
 \log\E e^{\lambda G_N^c}
 =
 \log\E e^{\lambda G_N}
 +\log\E e^{\lambda D_N}
 =
 \log\E e^{\lambda G_N}
 +\frac{\lambda^2}{4}.
\]
Substituting this exact identity into the completed ground-state
comparison gives \eqref{eq:ground-offdiagonal}.
\end{proof}

\section{Positive moments and variance}\zlabel{sec:moments}

To convert the fractional-moment comparison into a lower bound on the
variance, we need an upper bound on the centered moment generating
function whose scale is determined by $\Var F$, rather than by the
dimension of the underlying Gaussian vector. The required estimate
comes from two facts. First, transitivity of the Gaussian model allows
us to pin one configuration without changing the distribution of the
log partition function. Second, after adjoining an independent
exponential random variable, the pinned law becomes the marginal of a
log-concave measure. A one-dimensional tail estimate for log-concave
densities then gives the desired variance-sensitive moment bound.

We formulate the argument for a general finite Gaussian log partition
function. Let
\begin{equation}\label{eq:transitive-F}
 F(g)=\log\sum_{\sigma\in S}e^{a_\sigma\cdot g},
 \qquad M=|S|,
\end{equation}
where $g\sim N(0,I_d)$ and $a_\sigma\in\mathbb R^d$. Repeated vectors
$a_\sigma$ are allowed. Assume that for every
$\sigma,\tau\in S$ there is an orthogonal map $O$ such that
\[
 F(Og)=F(g)
 \qquad\text{and}\qquad
 O^Ta_\sigma=a_\tau.
\]
Thus the orthogonal symmetries preserving $F$ act transitively on the
features $a_\sigma$. In the off-diagonal SK model, these symmetries are
the coordinatewise gauge transformations
\[
 g_{ij}\mapsto \tau_i\tau_j g_{ij},
 \qquad \tau\in\{-1,1\}^N.
\]

\begin{proposition}\zlabel{prop:variance-sensitive-moments}
Let $F$ be as in \eqref{eq:transitive-F}, under the transitivity
assumption above, and set
\[
 V=\Var(F).
\]
Define
\[
 b=\frac{\log3}{8},
 \qquad
 B=\frac{64}{(\log3)^2}.
\]
Then, for every
\[
 0\le s\le\frac{b}{\sqrt{V+1}},
\]
we have
\begin{equation}\label{eq:variance-sensitive-mgf}
 \log\E e^{s(F-\E F)}
 \le
 Bs^2(V+1).
\end{equation}
\end{proposition}

\begin{proof}
We proceed in three steps.

\medskip
\emph{Step 1: Pinning one configuration without changing the law of
$F$.}
Let $\varphi_d$ denote standard Gaussian measure on $\mathbb R^d$,
and fix a label $\ast\in S$. For any bounded measurable function $h$,
orthogonal invariance of $\varphi_d$ and the transitivity assumption
show that
\[
 \E\!\left[
   h(F(g))e^{a_\sigma\cdot g-F(g)}
 \right]
\]
has the same value for every $\sigma\in S$. Summing this identity over
$\sigma$ and using
\[
 \sum_{\sigma\in S}e^{a_\sigma\cdot g-F(g)}=1
\]
gives
\[
 M\E\!\left[
   h(F(g))e^{a_\ast\cdot g-F(g)}
 \right]
 =
 \E h(F(g)).
\]
Consequently
\begin{equation}\label{eq:pinned-law}
 q_0(\dd g)
 =
 M e^{a_\ast\cdot g-F(g)}\varphi_d(\dd g)
\end{equation}
is a probability measure, and $F$ has the same distribution under
$q_0$ as under the original Gaussian law. The identity extends from
bounded $h$ to every integrable $h$ by truncation.

\medskip
\emph{Step 2: Exponential augmentation produces a log-concave
one-dimensional law.}
Under $q_0$, let $E_1$ be an independent exponential random variable
of rate one and set
\[
 T=F(g)+E_1.
\]
With
\[
 U(g)=\frac{|g|^2}{2}-a_\ast\cdot g,
\]
the joint law of $(g,T)$ has density, up to normalization,
\begin{equation}\label{eq:epigraph-law}
 e^{-U(g)-t}\1_{\{t\ge F(g)\}}
\end{equation}
with respect to Lebesgue measure on $\mathbb R^d\times\mathbb R$.
Indeed, the factor $e^{F(g)}$ arising from the change of variables
$E_1=t-F(g)$ cancels the factor $e^{-F(g)}$ in
\eqref{eq:pinned-law}.

The function $F$ is convex because it is a log-sum-exp of affine
functions. Hence
\[
 \{(g,t):t\ge F(g)\}
\]
is convex. The function $U(g)+t$ is also convex, so the density in
\eqref{eq:epigraph-law} is log-concave. By Pr\'ekopa's marginal
theorem \cite[Theorems~6 and~8]{Prekopa}, the one-dimensional density
of $T$ is log-concave. Since $F$ has its original distribution under
$q_0$ and $E_1$ is independent,
\[
 \E T=\E F+1,
 \qquad
 \Var(T)=V+1.
\]

\medskip
\emph{Step 3: A log-concave variable has Gaussian-type moments for
small tilts.}
Let $X$ be a real random variable with a log-concave density, mean
$m$, and standard deviation $d>0$. Pr\'ekopa's theorem also implies
that the functions
\[
 x\mapsto\Prob(X\ge x),
 \qquad
 x\mapsto\Prob(X\le x)
\]
are log-concave. Chebyshev's inequality gives
\begin{equation}\label{eq:chebyshev-tail}
 \Prob(X\ge m-2d)\ge\frac34,
 \qquad
 \Prob(X\ge m+2d)\le\frac14.
\end{equation}
Apply concavity of the logarithm of the survival function to the two
points $m-2d$ and $m+2d$. For $x\ge2d$, the resulting secant-slope
bound gives an exponential upper tail. Applying the same argument to
$-X$ and using the trivial bound for $0\le x\le2d$ yields
\begin{equation}\label{eq:logconcave-tail}
 \Prob(|X-m|\ge x)
 \le
 2e^{-hx/d},
 \qquad
 h=\frac{\log3}{4},
 \qquad x\ge0.
\end{equation}
For completeness, when $x\le2d$ the right-hand side is at least one.
For $x\ge2d$, the two one-sided estimates obtained from
\eqref{eq:chebyshev-tail} sum to at most
\[
 \frac{\sqrt3}{2}e^{-hx/d}
 \le2e^{-hx/d}.
\]

Integrating \eqref{eq:logconcave-tail} gives, for every integer
$k\ge2$,
\[
 \E|X-m|^k
 \le
 2k!\left(\frac d h\right)^k.
\]
Let
\[
 r=\frac{|\lambda|d}{h}.
\]
For $r<1$, the exponential series is absolutely convergent. Since
$\E(X-m)=0$,
\[
 \begin{aligned}
 \E e^{\lambda(X-m)}
 &\le
 1+
 \sum_{k=2}^\infty
 \frac{|\lambda|^k}{k!}\E|X-m|^k\\
 &\le
 1+2\sum_{k=2}^\infty r^k
 =
 1+\frac{2r^2}{1-r}.
 \end{aligned}
\]
Therefore
\begin{equation}\label{eq:logconcave-mgf}
 \log\E e^{\lambda(X-m)}
 \le
 \frac{2r^2}{1-r}
 \le4r^2
 \qquad
 \text{whenever }|\lambda|d\le h/2.
\end{equation}

Apply this estimate to $T$, for which
$d=\sqrt{V+1}$. For a random variable $Y$ with exponential moments
near zero, write
\[
 K_Y(\lambda)
 =
 \log\E e^{\lambda(Y-\E Y)}.
\]
Because $T=F+E_1$ under $q_0$, with $F$ and $E_1$ independent there,
\[
 K_T(\lambda)
 =
 K_F(\lambda)-\lambda-\log(1-\lambda),
 \qquad 0\le\lambda<1.
\]
Equivalently,
\begin{equation}\label{eq:mgf-transfer}
 K_F(\lambda)
 =
 K_T(\lambda)+\lambda+\log(1-\lambda)
 \le K_T(\lambda).
\end{equation}
If
\[
 0\le\lambda
 \le
 \frac{h}{2\sqrt{V+1}},
\]
then $\lambda<1$ and \eqref{eq:logconcave-mgf} gives
\[
 K_F(\lambda)
 \le
 \frac{4}{h^2}\lambda^2(V+1).
\]
Since
\[
 b=\frac h2=\frac{\log3}{8},
 \qquad
 B=\frac4{h^2}
   =\frac{64}{(\log3)^2},
\]
this is exactly \eqref{eq:variance-sensitive-mgf}.
\end{proof}

For the SK free energy, $V=\Var(F)=V_N$. Thus the proposition controls
positive tilts up to order
\[
 (V_N+1)^{-1/2},
\]
which is the scale needed below to convert the fractional-moment lower
bound into a variance lower bound.

\subsection{Convexity of normalized Gaussian log moments}

We will use convexity in the tilt parameter to compare a positive
logarithmic moment with its derivative at zero. For a convex Lipschitz
function of a Gaussian vector, this convexity is a special case of
Chen's Gaussian convexity theorem
\cite[Theorem~1]{ChenConvexity}. We follow Chen's stochastic-control proof in the convex Lipschitz
setting.

\begin{lemma}
\zlabel{lem:short-laplace-convexity}
Let $g$ be a standard Gaussian vector in $\mathbb R^d$, and let
$f:\mathbb R^d\to\mathbb R$ be convex and $K$-Lipschitz. For
$\lambda\ne0$, define
\[
 L_f(\lambda)
 =
 \frac1\lambda\log\E e^{\lambda f(g)}.
\]
Then $L_f$ extends to a convex function on $\mathbb R$, with
\[
 L_f(0)=\E f(g),
 \qquad
 L_f'(0)=\frac12\Var(f(g)).
\]
\end{lemma}

\begin{proof}
We first derive the form of the Bou\'e--Dupuis control representation
\cite{BoueDupuis1998} needed below for a smooth terminal function, and
then remove the smoothing. Let $B$ be standard Brownian
motion in $\mathbb R^d$, equipped with its usual augmented natural
filtration, and let $\mathcal A_K$ be the set of predictable processes
$a=(a_q)_{0\le q\le1}$ satisfying
\[
 |a_q|\le K
\]
almost everywhere. The same admissible class will be used for every
value of $\lambda$.

For $\varepsilon>0$, let
\[
 f_\varepsilon(x)=\E f(x+\sqrt\varepsilon\,g).
\]
Gaussian convolution preserves convexity and the $K$-Lipschitz
constant. It also makes $f_\varepsilon$ smooth with bounded Hessian.
We first work with $f_\varepsilon$ and suppress the subscript
$\varepsilon$.

Fix $\lambda\ne0$ and define
\[
 u_\lambda(q,x)
 =
 \frac1\lambda
 \log\E
 \exp\left\{
   \lambda f(x+B_1-B_q)
 \right\}.
\]
Then
\[
 u_\lambda(0,0)=L_f(\lambda)
\]
and $u_\lambda$ solves
\[
 \partial_q u_\lambda
 +\frac12\Delta u_\lambda
 +\frac{\lambda}{2}|\nabla u_\lambda|^2
 =0,
 \qquad
 u_\lambda(1,x)=f(x).
\]
Differentiating the normalized Gaussian expectation gives
\[
 |\nabla u_\lambda|\le K.
\]
Moreover,
\[
 \nabla^2u_\lambda
 =
 \E_\lambda[\nabla^2f]
 +\lambda\Cov_\lambda(\nabla f),
\]
where $\E_\lambda$ denotes expectation under the corresponding tilted
Gaussian law. Hence
\[
 \|\nabla^2u_\lambda\|_{\mathrm{op}}
 \le
 \|\nabla^2f\|_\infty+|\lambda|K^2.
\]
In particular, $\nabla u_\lambda$ is globally Lipschitz.

For $a\in\mathcal A_K$, let
\[
 Y_q^\lambda
 =
 B_q+\lambda\int_0^q a_r\,\dd r
\]
and define
\[
 J_\lambda(a)
 =
 \E\left[
 f(Y_1^\lambda)
 -\frac{\lambda}{2}\int_0^1|a_q|^2\,\dd q
 \right].
\]
It\^o's formula, together with the PDE for $u_\lambda$, gives
\[
 J_\lambda(a)-L_f(\lambda)
 =
 -\frac{\lambda}{2}
 \E\int_0^1
 \left|
   a_q-\nabla u_\lambda(q,Y_q^\lambda)
 \right|^2\dd q.
\]
The stochastic integral has mean zero because
$|\nabla u_\lambda|\le K$. Since $\nabla u_\lambda$ is bounded and
globally Lipschitz, the feedback equation
\[
 \dd Y_q
 =
 \dd B_q
 +\lambda\nabla u_\lambda(q,Y_q)\,\dd q,
 \qquad
 Y_0=0,
\]
has a unique strong solution. Its feedback control
\[
 a_q=\nabla u_\lambda(q,Y_q)
\]
belongs to $\mathcal A_K$ and makes the integral in the preceding
identity vanish. Therefore
\begin{equation}\label{eq:convex-control-formula}
 L_f(\lambda)
 =
 \begin{cases}
 \displaystyle
 \sup_{a\in\mathcal A_K}J_\lambda(a),
 &\lambda>0,\\[6pt]
 \displaystyle
 \inf_{a\in\mathcal A_K}J_\lambda(a),
 &\lambda<0.
 \end{cases}
\end{equation}

We next prove convexity for positive $\lambda$. For each fixed
$a\in\mathcal A_K$, the endpoint
\[
 Y_1^\lambda
 =
 B_1+\lambda\int_0^1a_q\,\dd q
\]
is affine in $\lambda$. Since $f$ is convex,
$\lambda\mapsto\E f(Y_1^\lambda)$ is convex, while the second term in
$J_\lambda(a)$ is linear in $\lambda$. Thus
$\lambda\mapsto J_\lambda(a)$ is convex. By
\eqref{eq:convex-control-formula}, $L_f$ is the supremum of convex
functions on $(0,\infty)$ and is therefore convex there.

For negative $\lambda$, the infimum representation requires a
different argument. Fix $\lambda_0,\lambda_1<0$,
$a_0,a_1\in\mathcal A_K$, and $0\le\theta\le1$. Set
\[
 \lambda_\theta
 =
 \theta\lambda_0+(1-\theta)\lambda_1
\]
and
\[
 a_\theta
 =
 \frac{
   \theta\lambda_0a_0+(1-\theta)\lambda_1a_1
 }{\lambda_\theta}.
\]
Because all three $\lambda$'s are negative, the two coefficients in
this convex combination are nonnegative and sum to one. Hence
$a_\theta\in\mathcal A_K$. Moreover,
\[
 Y_1^{\lambda_\theta,a_\theta}
 =
 \theta Y_1^{\lambda_0,a_0}
 +(1-\theta)Y_1^{\lambda_1,a_1}.
\]
Convexity of $f$ therefore gives
\[
 \E f(Y_1^{\lambda_\theta,a_\theta})
 \le
 \theta\E f(Y_1^{\lambda_0,a_0})
 +(1-\theta)\E f(Y_1^{\lambda_1,a_1}).
\]
Also, convexity of $x\mapsto|x|^2$ gives, pointwise in $q$,
\[
 -\lambda_\theta|a_\theta|^2
 \le
 -\theta\lambda_0|a_0|^2
 -(1-\theta)\lambda_1|a_1|^2.
\]
Consequently
\[
 J_{\lambda_\theta}(a_\theta)
 \le
 \theta J_{\lambda_0}(a_0)
 +(1-\theta)J_{\lambda_1}(a_1).
\]
Using the infimum representation at $\lambda_\theta$ and then taking
the infimum independently over $a_0$ and $a_1$ gives
\[
 L_f(\lambda_\theta)
 \le
 \theta L_f(\lambda_0)
 +(1-\theta)L_f(\lambda_1).
\]
Thus $L_f$ is convex on $(-\infty,0)$.

We now remove the smoothing. Since $f$ is $K$-Lipschitz,
\[
 \sup_x|f_\varepsilon(x)-f(x)|
 \le
 K\sqrt\varepsilon\,\E|g|.
\]
If two terminal functions differ uniformly by at most $\delta$, then
their normalized logarithmic moments differ by at most $\delta$ for
every $\lambda\ne0$, and the corresponding control objectives differ
by the same amount. Hence the control representations and the
convexity conclusions above pass to $f$ as
$\varepsilon\downarrow0$.

It remains to identify the value and derivative at $\lambda=0$.
Because a Lipschitz function of a Gaussian vector has exponential
moments in a neighborhood of zero,
\[
 \log\E
 e^{\lambda(f(g)-\E f(g))}
 =
 \frac{\lambda^2}{2}\Var(f(g))
 +O(\lambda^3)
 \qquad (\lambda\to0).
\]
Therefore
\[
 L_f(\lambda)
 =
 \E f(g)
 +\frac{\lambda}{2}\Var(f(g))
 +O(\lambda^2).
\]
Thus $L_f$ extends continuously to zero with
\[
 L_f(0)=\E f(g),
 \qquad
 L_f'(0)=\frac12\Var(f(g)).
\]
The restrictions of $L_f$ to $(-\infty,0]$ and $[0,\infty)$ are
convex and have the same one-sided derivative at zero. They therefore
join to a convex function on all of $\mathbb R$.
\end{proof}

\subsection{Consequences of a mean-deficit estimate}

We now record how a bound on the finite-size mean deficit
\[
 A_N-\frac{\E X_N}{N}
\]
combines with a fifth-order bound on normalized positive moments. The
mean deficit determines a crossover tilt $t_N$: below this scale,
convexity of the normalized logarithmic moment gives a quadratic bound,
while above it the fifth-order variational error gives a sixth-order
bound for the centered log moment. These two regimes yield a variance
bound and two corresponding upper-tail scales.

We will also use Chen's negative-tilt estimate
\cite[Theorem~2(1)]{ChenConvexity}, which follows from
\zcref{lem:short-laplace-convexity}. If $X=f(g)$ with $f$ convex and
Lipschitz, then 
\begin{equation}\label{eq:chen-negative-moment}
 \log\E e^{s(X-\E X)}
 \le\frac{s^2}{2}\Var(X),
 \qquad s\le0.
\end{equation}
Indeed, convexity of
\[
 L_f(s)=\frac1s\log\E e^{sf(g)}
\]
and the tangent inequality at zero give
\[
 L_f(s)\ge L_f(0)+sL_f'(0)
 =\E X+\frac{s}{2}\Var(X).
\]
Multiplication by $s<0$ reverses the inequality and gives
\eqref{eq:chen-negative-moment}; at $s=0$ both sides vanish. This
estimate will control the lower tail around the finite-system mean.

\begin{proposition}\zlabel{prop:mean-moment-consequences}
Let $0<a<5/6$. For each $N\ge1$, let $X_N$ be a convex function of a
standard Gaussian vector, with Lipschitz constant at most
$L\sqrt N$, where $L>0$ is fixed, and let $A_N$ be deterministic.
Suppose that there are fixed constants $M_0,C_0,s_*>0$ such that, for
all sufficiently large $N$,
\begin{equation}\label{eq:mean-moment-hypotheses}
 \begin{gathered}
 0\le A_N-\frac{\E X_N}{N}\le M_0N^{-a},\\
 \frac1{Ns}\log\E e^{sX_N}
 \le A_N+C_0s^5,
 \qquad 0<s\le s_*.
 \end{gathered}
\end{equation}

Choose the crossover tilt by matching the two errors,
\[
 t_N^5=N^{-a},
 \qquad\text{that is,}\qquad
 t_N=N^{-a/5},
\]
and set
\[
 Q_N=\frac{N^{1-a}}{t_N}=N^{1-4a/5},
 \qquad
 K_N^X(s)=\log\E e^{s(X_N-\E X_N)}.
\]
Then there are constants $C_1,c_1>0$ and $N_1<\infty$, depending only
on $a,L,M_0,C_0,s_*$ and the threshold in
\eqref{eq:mean-moment-hypotheses}, such that, for every $N\ge N_1$,
\begin{equation}\label{eq:mean-moment-window}
 \begin{aligned}
 K_N^X(s)&\le C_1Q_Ns^2
 &&\text{for }s\le t_N,\\
 K_N^X(s)&\le C_1Ns^6
 &&\text{for }t_N\le s\le s_*,\\
 \Var(X_N)&\le2C_1Q_N.
 \end{aligned}
\end{equation}
The first estimate includes the entire negative half-line.

For every $u\ge0$,
\begin{equation}\label{eq:mean-moment-tails}
 \begin{aligned}
 \Prob\{X_N-\E X_N\le-u\}
 &\le
 \exp\left(-c_1\frac{u^2}{Q_N}\right),\\
 \Prob\{X_N-\E X_N\ge u\}
 &\le
 \exp\left[
  -c_1\min\left\{
    \frac{u^2}{Q_N},
    \frac{u^{6/5}}{N^{1/5}}
  \right\}
 \right].
 \end{aligned}
\end{equation}
Moreover, for every $s<0$,
\begin{equation}\label{eq:mean-negative-replicas}
 0\le
 A_N-\frac1{Ns}\log\E e^{sX_N}
 \le
 C_1\left[
   N^{-a}+|s|N^{-4a/5}
 \right].
\end{equation}
The two upper-tail exponents in \eqref{eq:mean-moment-tails} are equal
when
\[
 u=N^{1-a},
\]
and their common value is of order $N^{1-6a/5}$. The restriction
$a<5/6$ ensures that this crossover exponent grows with $N$.
\end{proposition}

\begin{proof}
Let
\[
 D_0=M_0+C_0
\]
and define
\[
 L_N(s)=\frac1s\log\E e^{sX_N},
 \qquad s\ne0,
 \qquad
 L_N(0)=\E X_N.
\]
By \zcref{lem:short-laplace-convexity}, $L_N$ is convex on
$\mathbb R$.

Take $N$ large enough that $t_N\le s_*$. Since $t_N^5=N^{-a}$,
\eqref{eq:mean-moment-hypotheses} gives
\[
 \begin{aligned}
 L_N(t_N)-L_N(0)
 &\le
 N\left(A_N-\frac{\E X_N}{N}\right)
 +C_0Nt_N^5\\
 &\le D_0N^{1-a}.
 \end{aligned}
\]
For $0\le s\le t_N$, convexity gives the  bound
\[
 L_N(s)-L_N(0)
 \le
 \frac{s}{t_N}\bigl(L_N(t_N)-L_N(0)\bigr).
\]
Multiplying by $s$ yields
\[
 K_N^X(s)
 \le
 D_0\frac{N^{1-a}}{t_N}s^2
 =
 D_0Q_Ns^2.
\]
Thus the mean-deficit estimate produces a quadratic centered moment
bound below the crossover tilt.

For $t_N\le s\le s_*$, we instead use the second hypothesis directly:
\[
 \begin{aligned}
 K_N^X(s)
 &=s\bigl(L_N(s)-L_N(0)\bigr)\\
 &\le
 Ns\left[
   M_0N^{-a}+C_0s^5
 \right].
 \end{aligned}
\]
Since $s\ge t_N$ implies $N^{-a}=t_N^5\le s^5$,
\[
 K_N^X(s)\le D_0Ns^6.
\]
This gives the second line of \eqref{eq:mean-moment-window}.

As $s\to0$,
\[
 K_N^X(s)=\frac{s^2}{2}\Var(X_N)+o(s^2).
\]
Comparing this expansion with the quadratic bound above gives
\[
 \Var(X_N)\le2D_0Q_N.
\]
Equation \eqref{eq:chen-negative-moment} then gives, for every $s<0$,
\[
 K_N^X(s)
 \le
 \frac{s^2}{2}\Var(X_N)
 \le
 D_0Q_Ns^2.
\]
After enlarging the constant, all three estimates in
\eqref{eq:mean-moment-window} follow.

We next derive the tail bounds. For the lower tail, Chernoff's
inequality and the negative quadratic moment bound give, for every
$s<0$,
\[
 \Prob\{X_N-\E X_N\le-u\}
 \le
 \exp\{K_N^X(s)+su\}.
\]
Choosing
\[
 s=-\frac{u}{2C_1Q_N}
\]
gives
\[
 \Prob\{X_N-\E X_N\le-u\}
 \le
 \exp\left(-c\frac{u^2}{Q_N}\right).
\]

For the upper tail, the two positive-tilt estimates imply, after
enlarging $C_1$ if necessary,
\[
 K_N^X(s)
 \le
 C_1\bigl(Q_Ns^2+Ns^6\bigr),
 \qquad 0\le s\le s_*.
\]
Choose $c_0>0$, depending only on $C_1$, sufficiently small that
\[
 C_1(c_0+c_0^5)\le\frac12,
\]
and set
\[
 s
 =
 c_0\min\left\{
   \frac{u}{Q_N},
   \left(\frac{u}{N}\right)^{1/5}
 \right\}.
\]
Then
\[
 Q_Ns^2\le c_0su,
 \qquad
 Ns^6\le c_0^5su,
\]
and hence
\[
 K_N^X(s)\le\frac12su.
\]
For $u\le c_2N$, where $c_2>0$ is chosen sufficiently small in terms
of $c_0$ and $s_*$, this tilt satisfies $s\le s_*$. Chernoff's
inequality therefore gives
\[
 \Prob\{X_N-\E X_N\ge u\}
 \le e^{-su/2},
\]
and consequently
\[
 \Prob\{X_N-\E X_N\ge u\}
 \le
 \exp\left[
  -c\min\left\{
    \frac{u^2}{Q_N},
    \frac{u^{6/5}}{N^{1/5}}
  \right\}
 \right]
\]
throughout this range.

For $u>c_2N$, the Gaussian Lipschitz concentration inequality gives
\[
 \Prob\{X_N-\E X_N\ge u\}
 \le
 \exp\left(-\frac{u^2}{2L^2N}\right).
\]
In this range,
\[
 \frac{u^2/N}{u^{6/5}/N^{1/5}}
 =
 \left(\frac{u}{N}\right)^{4/5}
 \ge c_2^{4/5}.
\]
Thus, after decreasing the constant in the exponent, the
same upper-tail estimate holds for all $u\ge0$. This proves
\eqref{eq:mean-moment-tails}.

Finally, let $s<0$. Since
\[
 \log\E e^{sX_N}
 =
 s\E X_N+K_N^X(s),
\]
we have the exact identity
\[
 A_N-\frac1{Ns}\log\E e^{sX_N}
 =
 A_N-\frac{\E X_N}{N}
 +\frac{K_N^X(s)}{N|s|}.
\]
Jensen's inequality gives $K_N^X(s)\ge0$, so the left-hand side is
nonnegative. Using \eqref{eq:chen-negative-moment} and the variance
bound,
\[
 \frac{K_N^X(s)}{N|s|}
 \le
 \frac{|s|}{2N}\Var(X_N)
 \le
 C|s|\frac{Q_N}{N}
 =
 C|s|N^{-4a/5}.
\]
Together with the mean-deficit hypothesis, this proves
\eqref{eq:mean-negative-replicas}.

To locate the crossover in the upper-tail estimate, solve
\[
 \frac{u^2}{Q_N}
 =
 \frac{u^{6/5}}{N^{1/5}}.
\]
Since $Q_N=N^{1-4a/5}$, this gives $
 u=N^{1-a}$, 
and substitution into either exponent gives $
 N^{1-6a/5}$.
\end{proof}

\section{Two-sided upper-tail bounds}\zlabel{sec:tails}

\subsection{Consequences of a sixth-power moment window}

We now convert matching sixth-power bounds on logarithmic moments of
$Y_N-NA_N$ into matching upper and lower bounds for the upper-tail
probability. A tilt $s$ probes a per-particle deviation of size
$x\asymp s^5$, while
\[
 \ell_N(s)\asymp Ns^6\asymp Nx^{6/5}.
\]
The lower moment bound is assumed only for
$s\gtrsim N^{-\tau}$, so the matching lower bound on the upper-tail probability begins at
$x\gtrsim N^{-5\tau}$. The restriction $\tau<1/6$ ensures that the
corresponding exponent $N^{1-6\tau}$ tends to infinity.

The same assumptions also control the law obtained by exponentially
tilting $Y_N$. Convexity of the normalized logarithmic moment
$\ell_N(s)/s$ forces the tilted mean to move by order $Ns^5$ and,
on the same tilt window, gives a lower bound of order $Ns^4$ on the
tilted variance. These estimates will be used both for the
positive-temperature free energy and, after the zero-temperature
rescaling, for the ground-state energy.

\begin{proposition}\zlabel{prop:sixth-window-consequences}
Let $0<\tau<1/6$. For each $N\ge1$, let $Y_N$ be a random variable
with finite exponential moments of all orders, let $A_N$ be
deterministic, and define
\[
 \ell_N(s)
 =
 \log\E e^{s(Y_N-NA_N)}.
\]
Assume that
\[
 s\longmapsto\frac{\ell_N(s)}{s}
\]
is convex on $(0,\infty)$. Suppose further that there are fixed
constants
\[
 a_0,b_0,\kappa,s_*>0
\]
and an integer $N_{\mathrm m}$ such that, for every
$N\ge N_{\mathrm m}$,
\begin{equation}\label{eq:sixth-window-hypotheses}
 \begin{aligned}
 \ell_N(s)
 &\le b_0Ns^6,
 &&0<s\le s_*,\\
 \ell_N(s)
 &\ge a_0Ns^6,
 &&\kappa N^{-\tau}\le s\le s_*.
 \end{aligned}
\end{equation}

Then there exist constants
\[
 A_x,x_*,c_x,C_x>0
\]
and an integer $N_0<\infty$, depending only on
$a_0,b_0,\kappa,s_*,\tau$, and $N_{\mathrm m}$, such that, for every
$N\ge N_0$ and
\[
 A_xN^{-5\tau}\le x\le x_*,
\]
the upper-tail probability satisfies
\begin{equation}\label{eq:sixth-window-tails}
 \exp\{-C_xNx^{6/5}\}
 \le
 \Prob\{Y_N\ge N(A_N+x)\}
 \le
 \exp\{-c_xNx^{6/5}\}.
\end{equation}
If only the upper bound in
\eqref{eq:sixth-window-hypotheses} is assumed, the upper probability
bound in \eqref{eq:sixth-window-tails} still holds for every
$0<x\le x_*$.

For $s>0$, let $\E_s$ denote expectation under the exponentially
tilted law
\[
 \E_s f(Y_N)
 =
 \frac{\E[f(Y_N)e^{sY_N}]}{\E e^{sY_N}},
\]
and write
\[
 \Var_s(Y_N)
 =
 \E_s\bigl[(Y_N-\E_sY_N)^2\bigr].
\]
After increasing $N_0$ if necessary, for every $N\ge N_0$ and
\[
 \kappa N^{-\tau}\le s\le s_*/2,
\]
we have
\begin{equation}\label{eq:sixth-window-tilted}
 \begin{gathered}
 a_0Ns^5
 \le
 \E_sY_N-NA_N
 \le
 64b_0Ns^5,\\
 \Var_s(Y_N)\ge a_0Ns^4.
 \end{gathered}
\end{equation}

In particular, if the constants and the onset
$\kappa N^{-\tau}$ in \eqref{eq:sixth-window-hypotheses} are uniform
over an additional parameter, then the constants in
\eqref{eq:sixth-window-tails} and
\eqref{eq:sixth-window-tilted} may be chosen uniformly over that
parameter as well.
\end{proposition}

\begin{proof}
We first choose the constants so that the lower moment window is
nonempty and the Paley--Zygmund argument below has a nonvanishing
exponential scale. Increase $N_0$ so that, for every $N\ge N_0$,
\[
 \kappa N^{-\tau}\le\frac{s_*}{2},
 \qquad
 a_0\kappa^6N^{1-6\tau}\ge2\log2.
\]
This is possible because $\tau<1/6$. Since the upper and lower bounds
in \eqref{eq:sixth-window-hypotheses} hold simultaneously on a
nonempty interval for all sufficiently large $N$, necessarily
$a_0\le b_0$.

Set
\[
 A_x=\frac{a_0}{2}\kappa^5,
 \qquad
 x_*=
 \min\left\{
   \frac{a_0}{2}\left(\frac{s_*}{2}\right)^5,
   6b_0s_*^5
 \right\}.
\]

\medskip
\emph{Upper-tail bound.}
Fix $0<x\le x_*$. Chernoff's inequality gives, for every
$0<s\le s_*$,
\[
 \Prob\{Y_N-NA_N\ge Nx\}
 \le
 \exp\{\ell_N(s)-Nsx\}
 \le
 \exp\{N(b_0s^6-sx)\}.
\]
Choose
\[
 s=\left(\frac{x}{6b_0}\right)^{1/5}.
\]
Because $a_0\le b_0$ and $x\le x_*$, this choice lies in
$(0,s_*]$. Substitution gives 
\[
 \Prob\{Y_N-NA_N\ge Nx\}
 \le
 \exp\left\{
 -5\cdot6^{-6/5}b_0^{-1/5}Nx^{6/5}
 \right\}.
\]
This proves the upper bound in \eqref{eq:sixth-window-tails}; only the
upper moment hypothesis was used.

\medskip
\emph{Lower bound on the upper-tail probability.}
Now assume
\[
 A_xN^{-5\tau}\le x\le x_*,
\]
and choose
\[
 s=\left(\frac{2x}{a_0}\right)^{1/5}.
\]
By the definitions of $A_x$ and $x_*$,
\[
 \kappa N^{-\tau}\le s\le\frac{s_*}{2}.
\]
Hence both $s$ and $2s$ lie in the ranges needed below, and
\[
 a_0Ns^6
 \ge
 a_0\kappa^6N^{1-6\tau}
 \ge2\log2.
\]

Apply Paley--Zygmund to
\[
 Z=e^{s(Y_N-NA_N)}.
\]
Since $\E Z=e^{\ell_N(s)}$,
\[
 \begin{aligned}
 \Prob\left\{
   Z\ge\frac12\E Z
 \right\}
 &\ge
 \frac14\frac{(\E Z)^2}{\E Z^2}\\
 &=
 \frac14
 \exp\{2\ell_N(s)-\ell_N(2s)\}\\
 &\ge
 \frac14e^{-64b_0Ns^6}.
 \end{aligned}
\]
On this event,
\[
 Y_N-NA_N
 \ge
 \frac{\ell_N(s)-\log2}{s}.
\]
The lower moment bound and the choice of $N_0$ give
\[
 \frac{\ell_N(s)-\log2}{s}
 \ge
 \frac{a_0}{2}Ns^5
 =
 Nx.
\]
Therefore
\[
 \Prob\{Y_N-NA_N\ge Nx\}
 \ge
 \frac14e^{-64b_0Ns^6}.
\]
Since
\[
 \log4
 \le
 a_0Ns^6,
\]
we may absorb the prefactor into the exponential:
\[
 \Prob\{Y_N-NA_N\ge Nx\}
 \ge
 \exp\{-(64b_0+a_0)Ns^6\}.
\]
Finally,
\[
 s^6=\left(\frac{2x}{a_0}\right)^{6/5},
\]
so the lower bound in \eqref{eq:sixth-window-tails} follows with, for
example,
\[
 C_x=(64b_0+a_0)\left(\frac{2}{a_0}\right)^{6/5}.
\]

\medskip
\emph{Mean and variance under the tilted law.}
It remains to prove \eqref{eq:sixth-window-tilted}. Define
\[
 H(s)=\frac{\ell_N(s)}{s},
 \qquad s>0,
\]
and set
\[
 d=\max\left\{
 2,\left(\frac{2b_0}{a_0}\right)^{1/5}
 \right\}.
\]
Fix
\[
 \kappa N^{-\tau}\le s\le\frac{s_*}{2}.
\]
The lower moment estimate applies at $s$, while the upper estimate
applies at $s/d$ whether or not $s/d$ remains inside the lower moment
window. Thus
\[
 H(s)\ge a_0Ns^5,
 \qquad
 H(s/d)\le b_0Nd^{-5}s^5.
\]

By convexity of $H$, its derivative at $s$ is bounded below by the
backward secant slope:
\[
 \begin{aligned}
 H'(s)
 &\ge
 \frac{H(s)-H(s/d)}{s-s/d}\\
 &\ge
 \frac{a_0-b_0d^{-5}}{1-d^{-1}}Ns^4.
 \end{aligned}
\]
The definition of $d$ gives
$b_0d^{-5}\le a_0/2$, and $1-d^{-1}\le1$. Hence
\begin{equation}\label{eq:sixth-window-Hprime}
 H'(s)\ge\frac{a_0}{2}Ns^4.
\end{equation}

Because $\ell_N(s)=sH(s)$,
\[
 \ell_N'(s)=H(s)+sH'(s).
\]
Using $H(s)\ge a_0Ns^5$ gives
\[
 \ell_N'(s)\ge a_0Ns^5.
\]
For the upper bound, ordinary convexity of $\ell_N$ gives
\[
 \ell_N'(s)
 \le
 \frac{\ell_N(2s)-\ell_N(s)}{s}.
\]
Here $2s\le s_*$, and the lower moment estimate implies
$\ell_N(s)\ge0$. Therefore
\[
 \ell_N'(s)
 \le
 \frac{\ell_N(2s)}{s}
 \le
 64b_0Ns^5.
\]

Finally, $H$ is convex and smooth on $(0,\infty)$, so
$H''(s)\ge0$. Differentiating $\ell_N=sH$ twice and using
\eqref{eq:sixth-window-Hprime} gives
\[
 \ell_N''(s)
 =
 2H'(s)+sH''(s)
 \ge
 a_0Ns^4.
\]
Finite exponential moments justify these differentiations, and
differentiating the logarithmic moment identifies
\[
 \ell_N'(s)=\E_sY_N-NA_N,
 \qquad
 \ell_N''(s)=\Var_s(Y_N).
\]
Thus
\[
 a_0Ns^5
 \le
 \E_sY_N-NA_N
 \le
 64b_0Ns^5,
 \qquad
 \Var_s(Y_N)\ge a_0Ns^4,
\]
which proves \eqref{eq:sixth-window-tilted}.
\end{proof}

An upper bound on the tilted variance follows from a posterior
identity for the general transitive Gaussian model.

\begin{proposition}\zlabel{prop:tilted-transitive-variance}
Let $S$ be finite and let $(a_\sigma)_{\sigma\in S}\subset\mathbb R^d$.
Assume that a family of orthogonal maps acts transitively on the labels
and preserves the feature family: for every such map $O$ there is a
permutation $\pi_O$ of $S$ such that
\[
 O^Ta_\sigma=a_{\pi_O(\sigma)},
\]
and the induced action on $S$ is transitive. Repeated features are
allowed. For $\kappa>0$, define
\[
 Y_\kappa(g)
 =
 \frac1\kappa\log\sum_{\sigma\in S}
 e^{\kappa a_\sigma\cdot g},
 \qquad
 R=\max_{\sigma\in S}|a_\sigma|,
\]
where $g$ is standard Gaussian in $\mathbb R^d$. Let $\Var_s$ denote
variance under the exponentially tilted law with density proportional
to $e^{sY_\kappa(g)}$. Then
\begin{equation}\label{eq:tilted-transitive-cap}
 \Var_s(Y_\kappa)\le R^2,
 \qquad 0\le s\le\kappa.
\end{equation}
For
\[
 Y_\infty(g)=\max_{\sigma\in S}a_\sigma\cdot g,
\]
the same bound holds under its exponential tilt for every finite
$s\ge0$.
\end{proposition}

\begin{proof}
Let $M=|S|$, let $\varphi_d$ denote standard Gaussian measure on
$\mathbb R^d$, and fix a label $\ast\in S$. Define the Gibbs weights
\[
 p_\sigma(g)
 =
 \frac{e^{\kappa a_\sigma\cdot g}}
      {\sum_{\tau\in S}e^{\kappa a_\tau\cdot g}}
 =
 e^{\kappa a_\sigma\cdot g-\kappa Y_\kappa(g)}.
\]
The symmetry assumption implies
\[
 Y_\kappa(Og)=Y_\kappa(g),
\]
and, by orthogonal invariance of $\varphi_d$, for every bounded
measurable $h$ the quantity
\[
 \int
 h(Y_\kappa(g))e^{sY_\kappa(g)}p_\sigma(g)\,
 \varphi_d(\dd g)
\]
is independent of $\sigma$. Summing over $\sigma$ and using
$\sum_\sigma p_\sigma=1$ shows that the distribution of $Y_\kappa$
under its exponential tilt is the same as its distribution under
\begin{equation}\label{eq:tilted-pinned-law}
 q_s(\dd g)
 =
 \frac{
 M e^{\kappa a_\ast\cdot g-(\kappa-s)Y_\kappa(g)}
 }{
 \E e^{sY_\kappa(g)}
 }
 \,\varphi_d(\dd g).
\end{equation}
Taking $h\equiv1$ also verifies that $q_s$ is a probability measure.
This argument is unchanged when different labels carry the same
feature vector.

We now use the curvature of the pinned density. Relative to Lebesgue
measure, its negative logarithm is, up to an additive constant,
\[
 V_s(g)
 =
 \frac{|g|^2}{2}
 -\kappa a_\ast\cdot g
 +(\kappa-s)Y_\kappa(g).
\]
Since $Y_\kappa$ is convex,
\[
 \nabla^2V_s(g)
 =
 I+(\kappa-s)\nabla^2Y_\kappa(g)
 \succeq I
\]
whenever $0\le s\le\kappa$. Brascamp--Lieb therefore gives
\[
 \Var_{q_s}(Y_\kappa)
 \le
 \E_{q_s}\!\left[
 \nabla Y_\kappa^T
 (\nabla^2V_s)^{-1}
 \nabla Y_\kappa
 \right]
 \le
 \E_{q_s}|\nabla Y_\kappa|^2.
\]
Moreover,
\[
 \nabla Y_\kappa(g)
 =
 \sum_{\sigma\in S}p_\sigma(g)a_\sigma.
\]
This is a convex combination of vectors of norm at most $R$, and hence
\[
 |\nabla Y_\kappa(g)|\le R.
\]
Because $Y_\kappa$ has the same tilted distribution under $q_s$ as
under the original Gaussian law, we conclude that
\[
 \Var_s(Y_\kappa)
 =
 \Var_{q_s}(Y_\kappa)
 \le R^2,
\]
which proves \eqref{eq:tilted-transitive-cap}.

Finally,
\[
 0\le Y_\kappa(g)-Y_\infty(g)
 \le\frac{\log M}{\kappa}
\]
for every $g$. Thus $Y_\kappa\to Y_\infty$ uniformly as
$\kappa\to\infty$. For each fixed finite $s\ge0$, we may take
$\kappa\ge s$ and apply \eqref{eq:tilted-transitive-cap}. In addition,
\[
 |Y_\kappa(g)|
 \le R|g|+\frac{\log M}{\kappa},
\]
so, for $j=0,1,2$, the quantities
\[
 |Y_\kappa(g)|^j e^{sY_\kappa(g)}
\]
are dominated, for all sufficiently large $\kappa$, by an integrable
function of the form
\[
 C(1+|g|^2)e^{sR|g|}.
\]
Dominated convergence therefore passes the tilted zeroth, first, and
second moments to the limit. Hence
\[
 \Var_s(Y_\infty)
 =
 \lim_{\kappa\to\infty}\Var_s(Y_\kappa)
 \le R^2,
\]
as claimed.
\end{proof}

\section{Fluctuations uniformly at low temperature}
\zlabel{sec:ground-fluctuations}

Fix $\beta_0>1$. Throughout this section $\beta\ge\beta_0$, and
constants and integer thresholds may depend only on $\beta_0$.

\subsection{Rescaling for the zero-temperature regime}

We now rewrite the scalar and finite-$N$ variational problems in the
variables that remain of order one as $\beta\to\infty$. Fix
$\beta_0>1$ and consider $\beta\ge\beta_0$. For a positive-temperature
cumulative profile $a$ and fractional-moment parameter $s$, set
\[
 \gamma=\beta a,
 \qquad
 \lambda=\beta s.
\]
We also rescale the spatial variable by
$x_{\mathrm{old}}=\beta x$. Thus
\[
 u_\gamma^{\mathrm{en}}(q,x)
 =
 \frac1\beta u_{\gamma/\beta}(q,\beta x),
 \qquad
 \cP_\beta^{\mathrm{en}}(\gamma)
 =
 \frac1\beta\cP_\beta(\gamma/\beta).
\]
In these variables the Parisi PDE has coefficient $\gamma$, rather
than $a$, and its terminal value is
\[
 f_\beta(x)
 =
 \frac1\beta\log(2\cosh(\beta x)).
\]

For later use, write
\[
 X_{N,\beta}=\frac{F_{N,\beta}}{\beta},
 \qquad
 E_\beta(\lambda)
 =
 \frac1\beta P_\beta(\lambda/\beta).
\]
The rescaled scalar functional has penalty
\[
 \frac12\int_0^1q\gamma(q)\,\dd q.
\]
Its admissible profiles are nondecreasing and right-continuous on
$[0,1)$, with
\[
 0\le\gamma(q)\le\beta,
 \qquad
 \gamma(0-)=0.
\]
The normalization of the original probability measure requires total
mass $\beta$. Any missing mass may be placed at $q=1$; this changes
only the endpoint value $\gamma(1)=\beta$ and affects neither the PDE
nor the integral penalty. The constraint that the original measure
have mass at least $s=\lambda/\beta$ at zero becomes
\[
 \gamma(q)\ge\lambda,
 \qquad 0\le q<1.
\]
Consequently $E_\beta(\lambda)$ is the minimum of the rescaled
functional over profiles satisfying this lower bound. Notice that
$\dd\gamma$ has total mass $\beta$, not one, and that
$\gamma\ge\lambda$ is a lower bound on the cumulative profile. It
allows additional mass at zero and is not a bound on the density of
$\dd\gamma$.

The variational definition is meaningful for
$0\le\lambda\le\beta$. Below we use it only on a fixed interval
independent of $\beta$:
\[
 0<\lambda\le\lambda_0,
 \qquad
 \lambda_0\le\min\{1,\beta_0/2\}.
\]
Define the completed finite-$N$ normalized log moment by
\[
 \psi^c_{N,\beta}(\lambda)
 =
 \frac1{N\lambda}
 \log\E e^{\lambda(X_{N,\beta}+D_N)},
 \qquad \lambda>0,
\]
and set
\[
 \psi^c_{N,\beta}(0)
 =
 \frac{\E X_{N,\beta}}{N},
\]
using $\E D_N=0$. Applying
\eqref{eq:uniform-parisi} with $s=\lambda/\beta$ and dividing by
$\beta$ gives
\begin{equation}\label{eq:gs-full-fv}
 0
 \le
 E_\beta(\lambda)-\psi^c_{N,\beta}(\lambda)
 \le
 C\min\left\{
   N^{-1/2},
   N^{-2/3}\lambda^{-1/3}
 \right\},
 \qquad 0<\lambda\le\lambda_0.
\end{equation}
The important point is that the right-hand side is now independent of
$\beta$.

We will also need the finite-$N$ interpolation in these rescaled
variables. Put
\[
 r=1-t,
 \qquad
 J_{ij}=\sqrt{\frac{t}{N}}\,g_{ij}\quad(i<j),
 \qquad
 J_{ii}=0.
\]
For a fixed profile $\gamma$, let $u$ solve
\begin{equation}\label{eq:gs-energy-interpolation}
\begin{gathered}
 \partial_q u
 +\frac r2\bigl(\Delta u+\gamma(q)|\nabla u|^2\bigr)=0,\\
 u(1,y;J)
 =
 \frac1\beta
 \log\sum_{\sigma\in\{-1,1\}^N}
 \exp\left\{
   \beta\left(\frac12\sigma^TJ\sigma+y\cdot\sigma\right)
 \right\},
 \\[2mm]
 \mathfrak F_t^c(\gamma)
 =
 \frac1{N\lambda}\log\E e^{\lambda u(0,0;J)}
 -\frac r2\int_0^1q\gamma(q)\,\dd q
 +\frac{\lambda t}{4N},
 \qquad
 V(t)=\min_{\gamma\ge\lambda}\mathfrak F_t^c(\gamma).
\end{gathered}
\end{equation}
At $\lambda=0$, the first term in $\mathfrak F_t^c$ is interpreted as
\[
 \frac1N\E u(0,0;J).
\]
The functional in \eqref{eq:gs-energy-interpolation} is exactly
$\beta^{-1}\mathcal F_{t,\lambda/\beta}$ after the spatial rescaling.
Here the matrix $J$ contains only the off-diagonal couplings. The
diagonal Gaussian completion is independent of $J$ and can be
integrated explicitly; its normalized logarithmic moment is
$\lambda t/(4N)$, which accounts for the final deterministic term.

For reference, denote the terminal function in
\eqref{eq:gs-energy-interpolation} by
\[
 f_{\beta,J}(y)=u(1,y;J).
\]
For a fixed profile $\gamma$, let $Y_q$, conditional on $J$, solve
\[
 \dd Y_q
 =
 r\gamma(q)\nabla u(q,Y_q;J)\,\dd q+\sqrt r\,\dd B_q,
 \qquad
 Y_0=0,
\]
where $B$ is standard $N$-dimensional Brownian motion. Take
expectations under the normalized joint law induced by this
interpolation and define
\[
 m_q=\nabla u(q,Y_q;J),
 \qquad
 H_q=\nabla^2u(q,Y_q;J),
\]
together with the deterministic functions
\[
 S(q)=\frac1N\E|m_q|^2,
 \qquad
 b(q)=\frac1N\E\Tr H_q,
 \qquad
 \zeta(q)=\frac1N\E\Tr H_q^2.
\]
Thus $S(q)$ is the expectation of the random observable
$|m_q|^2/N$ used in \zcref{sec:finite-n}, expressed after the present
rescaling.

The martingale identities from
\eqref{eq:fn-martingale-sdes} become
\[
 S'(q)=r\zeta(q),
 \qquad
 b'(q)=-r\gamma(q)\zeta(q).
\]
At zero external field, global spin-flip symmetry gives
\[
 S(0)=0.
\]
Moreover $|m_{q,i}|\le1$, so
\[
 0\le S(q)\le1.
\]
Finally, Cauchy--Schwarz gives
\[
 \zeta(q)
 =
 \frac1N\E\Tr H_q^2
 \ge
 \left(\frac1N\E\Tr H_q\right)^2
 =
 b(q)^2.
\]
These identities will be the finite-$N$ inputs for the variational
estimates below.

\subsection{Uniform scalar estimates near zero and the cost of imposing a floor}

We now study the unconstrained scalar minimizer near $q=0$. The goal
is to show, uniformly for $\beta\ge\beta_0>1$, that its rescaled
profile grows linearly from zero. This local information will then be
used to estimate the increase in the scalar variational minimum caused
by the constraint
\[
 \gamma(q)\ge\lambda.
\]

We first record the continuity and variational facts needed below.
They follow from the scalar version of the finite-temperature
calculation leading to \eqref{eq:abs-first-variation}, with $N=1$,
$r=1$, and no outer exponential tilt. That calculation uses only the
finite-temperature Parisi PDE and does not depend on any of the
near-zero estimates proved in this subsection.

For any two admissible scalar profiles $\gamma$ and
$\widetilde\gamma$,
\[
 \left|
 \cP_\beta^{\mathrm{en}}(\gamma)
 -\cP_\beta^{\mathrm{en}}(\widetilde\gamma)
 \right|
 \le
 \|\gamma-\widetilde\gamma\|_{L^1([0,1])}.
\]
Indeed, the first-variation formula bounds the change in the PDE value
by
\[
 \frac12\|\gamma-\widetilde\gamma\|_{L^1},
\]
while the change in the penalty
\[
 -\frac12\int_0^1 q\gamma(q)\,\dd q
\]
is bounded by the same quantity. If the measures associated with a
sequence of admissible profiles converge weakly, then their cumulative
functions converge at every continuity point of the limiting
cumulative function. Since these functions are uniformly bounded by
$\beta$, dominated convergence gives convergence in $L^1([0,1])$.
The admissible measures form a weakly compact set, so the preceding
continuity implies existence of a scalar minimizer.

Fix one such unconstrained minimizer and denote its rescaled
cumulative profile by $\gamma_\beta$. We use here the full interval
support theorem \cite[Theorem~1.1]{Lopatto}: for some
$q_\beta<1$,
\[
 \supp\bigl(\beta^{-1}\dd\gamma_\beta\bigr)
 =
 [0,q_\beta],
\]
and
\[
 \gamma_\beta(q)=\beta,
 \qquad q_\beta<q\le1.
\]
Thus every point of $(0,q_\beta)$ is a contact point for the scalar
first variation. 

Let $u_\beta$ be the scalar Parisi PDE solution associated with
$\gamma_\beta$. To obtain a uniform lower bound on
$\gamma_\beta(q)/q$ for small $q$, we will first prove that
\[
 -u_{\beta,xxxx}(0,0)
\]
is bounded below by a positive constant independent of
$\beta\ge\beta_0$. The required estimate is obtained from a
parabolic boundary comparison applied to
\[
 W(s,x)=-u_{\beta,xxx}(T-s,x),
\]
where the time $T>0$ is chosen uniformly in $\beta$ in the proof of
\zcref{lem:gs-scalar-root}. The next lemma isolates that comparison.
Its assumptions record explicitly the three quantities on which the
estimate depends: the bound on the drift, the amount of positive mass
in the initial data, and the spatial regularity of that data.

\begin{lemma}
\zlabel{lem:scalar-boundary-strictness}
Let $T>0$ and $M\ge1$. Suppose
\[
 W:[0,T]\times[0,5]\to[0,\infty)
\]
satisfies the boundary condition
\[
 W(s,0)=0,\qquad 0\le s\le T,
\]
and, for almost every $s\in[0,T]$,
\[
 W_s
 =
 \frac12W_{xx}+B(s,x)W_x+C(s,x)W
\]
on $(0,5)$, where $B$ and $C$ are bounded and measurable and
\[
 0\le B(s,x)\le2,\qquad C(s,x)\ge0.
\]
Assume that $W,W_x,W_{xx}$ are continuous up to the spatial boundary,
that $W$ is absolutely continuous in $s$, and that
\[
 |W_x(s,x)|\le M
 \qquad\text{on }[0,T]\times[0,5].
\]
Finally, suppose that the initial profile contains a fixed amount of
mass away from the boundary:
\[
 \int_0^4W(0,x)\,\dd x\ge\frac34.
\]

Define
\[
 a=\frac3{32},
 \qquad
 R=\min\left\{\frac14,\frac3{32M}\right\},
\]
and then
\[
 A=2+\frac6T,
 \qquad
 K=\frac3{R^2}+\frac32A^2,
 \qquad
 \eta=ae^{-KT},
 \qquad
 \alpha=1+\frac6T.
\]
Then the inward spatial derivative at the boundary satisfies the
explicit lower bound
\[
 W_x(T,0)
 \ge
 a_0(T,M)
 :=
 4\eta\alpha e^{-4\alpha}
 >0.
\]

Thus, for fixed $T$ and $M$, the lower bound is uniform over all
solutions satisfying the hypotheses above. No uniform upper bound on
$C$ is needed across such a family: it is enough that, for each
solution, $C$ is bounded and nonnegative.
\end{lemma}

\begin{proof}
The argument has two stages. First we propagate the lower bound in the
initial profile to a uniform lower bound for $W(s,1)$ on
$T/2\le s\le T$. We then use this interior lower bound as boundary data
for a second barrier on $[0,1]$, which yields a positive derivative at
$x=0$ at time $T$.

\medskip
\emph{Comparison principle.}
We first record the comparison argument used for both barriers. Let
$z$ satisfy
\[
 z_s-\frac12z_{xx}-Bz_x-Cz\ge0
\]
on a bounded space-time rectangle, with nonnegative initial data and
nonnegative values on the spatial boundary. Put
\[
 z_-=\max\{-z,0\}.
\]
Multiplying the differential inequality by $-2z_-$ and integrating in
space gives
\[
 \frac{\dd}{\dd s}\|z_-\|_2^2
 \le
 -\frac12\|(z_-)_x\|_2^2
 +(2\|B\|_\infty^2+2\|C\|_\infty)\|z_-\|_2^2.
\]
Indeed, integration by parts produces
$-\|(z_-)_x\|_2^2$, while
\[
 2\int |B|z_-|(z_-)_x|
 \le
 \frac12\|(z_-)_x\|_2^2
 +2\|B\|_\infty^2\|z_-\|_2^2.
\]
The zeroth-order term contributes at most
$2\|C\|_\infty\|z_-\|_2^2$. Approximating the negative part by smooth
convex functions justifies the calculation under the stated spatial
regularity, and absolute continuity in time is enough to integrate the
resulting differential inequality. Since $z_-$ vanishes initially,
Gronwall's inequality gives $z_-\equiv0$, hence $z\ge0$.

We will also use the following localized version. Suppose $W\ge0$ and
$v$ satisfies
\[
 v_s-\frac12v_{xx}-Bv_x-Cv\le0
\]
only on the set where $v>0$, with $v\le W$ on the parabolic boundary.
Then the same comparison applies to $z=W-v$: its negative part is
supported on
\[
 \{W-v<0\}\subset\{v>W\ge0\}\subset\{v>0\},
\]
precisely where the required differential inequality for $v$ holds.
Thus $W\ge v$ everywhere. Notice that neither argument differentiates
$B$ or $C$ in time.

\medskip
\emph{Step 1: Produce a lower bound at $x=1$.}
The assumption
\[
 \int_0^4W(0,x)\,\dd x\ge\frac34
\]
implies that there is $x_0\in[0,4]$ such that
\[
 W(0,x_0)\ge\frac{3}{16}.
\]
Since $W(0,0)=0$ and $|W_x|\le M$,
\[
 x_0\ge\frac{3}{16M}\ge2R.
\]
Moreover, because
\[
 R\le\frac{3}{32M},
\]
we have
\[
 W(0,x)
 \ge
 \frac{3}{16}-MR
 \ge
 \frac{3}{32}
 =a
 \qquad\text{whenever }|x-x_0|\le R.
\]
The choice $R\le1/4$ also ensures that this interval lies inside
$(0,5)$.

We construct a compactly supported bump that can be transported from
$x_0$ to $1$. Define
\[
 \psi(z)
 =
 \left(1-\frac{z^2}{R^2}\right)_+^3.
\]
Then $\psi\in C^2(\mathbb R)$, $0\le\psi\le1$, and $\psi(0)=1$.
For $|z|<R$, write
\[
 y=1-\frac{z^2}{R^2}.
\]
A direct calculation gives
\[
 \psi'
 =
 -\frac{6z}{R^2}y^2,
 \qquad
 \psi''
 =
 \frac{24z^2}{R^4}y-\frac6{R^2}y^2.
\]
For any $A>0$,
\[
 \frac{6A|z|}{R^2}y^2
 \le
 \frac{6z^2}{R^4}y+\frac32A^2y^3,
\]
by $2uv\le u^2+v^2$. Hence
\[
 \begin{aligned}
 A|\psi'|-\frac12\psi''
 &\le
 \frac3{R^2}(3y^2-2y)
 +\frac32A^2y^3\\
 &\le
 \left(\frac3{R^2}+\frac32A^2\right)y^3
 =K\psi.
 \end{aligned}
\]
For the second inequality we used
\[
 y^3-3y^2+2y
 =
 y(y-1)(y-2)\ge0,
 \qquad 0\le y\le1.
\]
At $|z|\ge R$, the function and its first two derivatives vanish, so
the same inequality holds everywhere.

Let $c(s)$ move linearly from $x_0$ to $1$ during
$0\le s\le T/2$ and remain equal to $1$ for $T/2\le s\le T$.
Since $x_0\in[0,4]$,
\[
 |c'(s)|
 \le
 \frac{2|x_0-1|}{T}
 \le\frac6T
\]
for almost every $s$. The $R$-neighborhood of the path of $c$ remains
inside $(0,5)$. Define
\[
 v(s,x)
 =
 ae^{-Ks}\psi(x-c(s)).
\]
On the support of $v$, its parabolic operator is
\[
 \begin{aligned}
 &v_s-\frac12v_{xx}-Bv_x-Cv\\
 &\qquad=
 ae^{-Ks}
 \left[
 -K\psi-\frac12\psi''
 -(B+c')\psi'-C\psi
 \right].
 \end{aligned}
\]
Because
\[
 |B+c'|\le2+\frac6T=A
\]
and $C\ge0$, the preceding bound on $\psi$ gives
\[
 v_s-\frac12v_{xx}-Bv_x-Cv\le0.
\]
At time zero, $v(0,\cdot)\le W(0,\cdot)$ by the construction of
$x_0$ and $R$, while $v=0\le W$ at $x=0$ and $x=5$.
Comparison therefore yields $W\ge v$. For $T/2\le s\le T$ we have
$c(s)=1$, and hence
\[
 W(s,1)
 \ge
 ae^{-Ks}
 \ge
 ae^{-KT}
 =\eta.
\]

\medskip
\emph{Step 2: Convert the bound at $x=1$ into a boundary derivative.}
We now work on the rectangle
\[
 [T/2,T]\times[0,1].
\]
Set
\[
 \kappa=\frac6T,
 \qquad
 \alpha=\kappa+1=1+\frac6T,
\]
and define
\[
 F(s,x)
 =
 \exp\left\{
 -\alpha\bigl((2-x)^2+\kappa(T-s)\bigr)
 \right\},
\]
\[
 V(s,x)
 =
 \eta\bigl(F(s,x)-e^{-4\alpha}\bigr).
\]
We first verify the parabolic boundary inequalities. At $s=T/2$,
\[
 (2-x)^2+\kappa T/2
 \ge
 1+3=4
 \qquad (0\le x\le1),
\]
so $V\le0\le W$. At $x=0$,
\[
 (2-x)^2=4,
\]
and therefore $V(s,0)\le0=W(s,0)$. At $x=1$,
\[
 F(s,1)\le1,
\]
so
\[
 V(s,1)\le\eta\le W(s,1)
\]
by Step~1.

It remains to check that $V$ is a subsolution wherever it is positive.
A direct differentiation gives
\[
 \begin{aligned}
 &(\partial_s-\tfrac12\partial_{xx}
      -B\partial_x-C)V\\
 &\qquad=
 \eta\alpha
 \left[
 \kappa+1
 -2\alpha(2-x)^2
 -2B(2-x)
 \right]F
 -CV.
 \end{aligned}
\]
On $0\le x\le1$,
\[
 (2-x)^2\ge1,
 \qquad
 B\ge0,
 \qquad
 C\ge0.
\]
Since $\alpha=\kappa+1$, on the set $\{V>0\}$ we therefore have
\[
 \kappa+1-2\alpha(2-x)^2-2B(2-x)
 \le
 \alpha-2\alpha
 \le0,
\]
and also $-CV\le0$. Hence $V$ is a subsolution wherever $V>0$.
The localized comparison principle gives
\[
 W\ge V
 \qquad\text{on }[T/2,T]\times[0,1].
\]

At the point $(T,0)$,
\[
 W(T,0)=V(T,0)=0.
\]
For $x>0$,
\[
 \frac{W(T,x)}{x}
 \ge
 \frac{V(T,x)}{x}.
\]
Letting $x\downarrow0$ and using the assumed continuity of the first
spatial derivatives gives
\[
 W_x(T,0)
 \ge
 V_x(T,0).
\]
Finally,
\[
 V_x(T,0)
 =
 4\eta\alpha e^{-4\alpha}
 =
 a_0(T,M).
\]
This proves the claimed uniform positive lower bound.
\end{proof}

Parabolic smoothing gives the spatial derivative bounds needed to
apply \zcref{lem:scalar-boundary-strictness} uniformly in
$\beta\ge\beta_0$. The resulting positive lower bound on
$-u_{\beta,xxxx}(0,0)$ can then be inserted into the scalar contact
identities. This yields two-sided linear control of the minimizing
profile $\gamma_\beta$ near $q=0$, together with a uniform bound on
its slope. Consequently, the trial profile
$\max\{\lambda,\gamma_\beta\}$ differs from $\gamma_\beta$ only on
an interval of length of order $\lambda$. The first variation
vanishes there at the unconstrained minimizer, and the excess
variational cost of this trial profile is at most of order
$\lambda^5$.

\begin{lemma}
\zlabel{lem:gs-scalar-root}
There exist constants
\[
 T_*>0,\qquad k>0,\qquad C_0<\infty,\qquad
 C_1<\infty,\qquad \lambda_*>0,
\]
depending only on $\beta_0$, such that the following holds for every
$\beta\ge\beta_0$.

Let $\gamma_\beta$ be an unconstrained minimizer of the rescaled scalar
Parisi functional. Then
\[
 \gamma_\beta(0)=0,
\]
and $\gamma_\beta$ is absolutely continuous on $[0,T_*]$. On this
interval,
\[
 kq\le\gamma_\beta(q)\le C_0q,
 \qquad 0\le q\le T_*,
\]
and
\[
 0\le\gamma_\beta'(q)\le C_0
\]
for almost every $q\in[0,T_*]$.

For $0<\lambda\le\lambda_*$, define
\[
 \gamma_{\beta,\lambda}^{\mathrm{tr}}(q)
 =
 \max\{\lambda,\gamma_\beta(q)\}.
\]
This profile satisfies the constraint
$\gamma_{\beta,\lambda}^{\mathrm{tr}}\ge\lambda$ and is therefore
admissible for the variational problem defining $E_\beta(\lambda)$.
Set
\[
 A_\beta(\lambda)
 =
 \cP_\beta^{\mathrm{en}}
 \bigl(\gamma_{\beta,\lambda}^{\mathrm{tr}}\bigr).
\]
Then
\begin{equation}
 E_\beta(0)
 \le
 E_\beta(\lambda)
 \le
 A_\beta(\lambda)
 \le
 E_\beta(0)+C_1\lambda^5.
 \label{eq:gs-scalar-trial-cost}
\end{equation}
Thus imposing the floor $\gamma\ge\lambda$ increases the scalar
variational minimum by at most order $\lambda^5$, uniformly in
$\beta\ge\beta_0$.
\end{lemma}

\begin{proof}
Let $u_\beta$ denote the scalar Parisi PDE solution associated with
the unconstrained minimizer $\gamma_\beta$. We prove the result in
four steps. First we obtain uniform bounds on $\gamma_\beta$ and on
the spatial derivatives of $u_\beta$ on a fixed interval near
$q=0$. 
We then use \zcref{lem:scalar-boundary-strictness} to obtain a uniform
positive lower bound on $-u_{\beta,xxxx}(0,0)$. The scalar contact
identities convert this strictness into two-sided linear bounds for
$\gamma_\beta$. Finally, we estimate the cost of replacing
$\gamma_\beta$ by $\max\{\lambda,\gamma_\beta\}$.

\medskip
\emph{Step 1: A uniform early-time strip and spatial regularity.}
Let $X_q$ be the scalar diffusion associated with $\gamma_\beta$,
started from $X_0=0$, and define
\[
 S(q)=\E u_{\beta,x}(q,X_q)^2.
\]
By the interval-support theorem, every
$q\in[0,q_\beta]$ is a contact point. We use the scalar contact and
differential identities from
\cite[Proposition~3 and Theorem~5]{AuffingerChen2015Properties}, written
here along the associated diffusion. The first-variation identity
therefore gives
\[
 S(q)=q,\qquad 0\le q\le q_\beta.
\]
The diffusion identity
\[
 S'(q)=\E u_{\beta,xx}(q,X_q)^2
\]
then gives
\[
 u_{\beta,xx}(0,0)=1,
\]
because $X_0=0$ and $u_{\beta,xx}\ge0$.

We next obtain a bound on the coefficient $\gamma_\beta$ that is
uniform in $\beta$. Set
\[
 b(q)=\E u_{\beta,xx}(q,X_q).
\]
The scalar identities give
\[
 b'(q)=-\gamma_\beta(q)S'(q),
 \qquad
 b(0)=1.
\]
Above the endpoint $q_\beta$ of the support,
$\gamma_\beta=\beta$, and the PDE can be solved explicitly:
\[
 u_\beta(q,x)
 =
 f_\beta(x)+\frac{\beta}{2}(1-q),
 \qquad q_\beta\le q\le1.
\]
Consequently
\[
 b(1)=\beta(1-S(1)).
\]
Integrating $b'=-\gamma_\beta S'$ first on the contact interval,
where $S'=1$, and then on $[q_\beta,1]$, where
$\gamma_\beta=\beta$, gives
\[
 \beta(1-S(1))
 =
 1-\int_0^{q_\beta}\gamma_\beta(q)\,\dd q
 -\beta\bigl(S(1)-q_\beta\bigr).
\]
The terms containing $S(1)$ cancel, and hence
\begin{equation}\label{eq:gs-scalar-mass-identity}
 \int_0^1\gamma_\beta(q)\,\dd q
 =
 \int_0^{q_\beta}\gamma_\beta(q)\,\dd q
 +\beta(1-q_\beta)
 =1.
\end{equation}
Since $\gamma_\beta\ge0$,
\[
 \beta(1-q_\beta)\le1,
\]
so
\[
 q_\beta
 \ge1-\beta^{-1}
 \ge1-\beta_0^{-1}>0.
\]
Moreover, monotonicity of $\gamma_\beta$ and
\eqref{eq:gs-scalar-mass-identity} imply
\[
 (1-q)\gamma_\beta(q)
 \le
 \int_q^1\gamma_\beta(s)\,\dd s
 \le1.
\]
Therefore
\begin{equation}\label{eq:gs-scalar-crude-profile}
 \gamma_\beta(q)\le\frac1{1-q},
 \qquad 0\le q<1.
\end{equation}
In particular,
\[
 0\le\gamma_\beta(q)\le2,
 \qquad 0\le q\le\frac12.
\]

We now obtain spatial derivative bounds that are uniform in
$\beta\ge\beta_0$. Put
\[
 p(s,x)=u_{\beta,x}(1/2-s,x),
 \qquad 0\le s\le\frac12.
\]
Then
\[
 p_s
 =
 \frac12p_{xx}
 +\gamma_\beta(1/2-s)\,p\,p_x,
 \qquad |p|\le1.
\]
Let
\[
 M_1(s)=\|p_x(s,\cdot)\|_\infty.
\]
The heat-kernel representation, together with
$|p|\le1$ and \eqref{eq:gs-scalar-crude-profile}, gives
\[
 M_1(s)
 \le
 Cs^{-1/2}
 +C\int_0^s(s-v)^{-1/2}M_1(v)\,\dd v.
\]
Iterating this Volterra inequality bounds $M_1$ uniformly on every
interval bounded away from $s=0$. Indeed, the $j$-fold convolution
of $s^{-1/2}$ contributes a factor proportional to
\[
 \frac{\Gamma(1/2)^j}{\Gamma(j/2)}.
\]

The same argument applies successively to higher derivatives. Write
\[
 M_j(s)=\|\partial_x^jp(s,\cdot)\|_\infty.
\]
Once derivatives of orders below $j$ are bounded on a strip beginning
at $s_0>0$, differentiate the  equation $j-1$ times and place one
spatial derivative on the heat kernel. In
$\partial_x^{j-1}(pp_x)$, the only term containing the derivative of
order $j$ is $p\,\partial_x^jp$; all remaining factors involve lower
derivatives. Thus
\[
 M_j(s)
 \le
 C_j(s-s_0)^{-1/2}
 +C_j\int_{s_0}^s
 (s-v)^{-1/2}\bigl(1+M_j(v)\bigr)\,\dd v.
\]
Applying the same convolution argument on successively smaller strips
gives
\begin{equation}\label{eq:gs-scalar-smoothing}
 \sup_{\beta\ge\beta_0}
 \sup_{0\le q\le1/4}
 \|\partial_x^ju_\beta(q,\cdot)\|_\infty
 \le C_j,
 \qquad 1\le j\le6.
\end{equation}
The proof uses only $|u_x|\le1$ and
$0\le\gamma\le2$ on $[0,1/2]$. Hence the same constants apply to
any profile satisfying these two bounds. No time derivative of
$\gamma$ is used.

\medskip
\emph{Step 2: A uniform lower bound on $-u_{\beta,xxxx}(0,0)$.}
Set
\[
 w(q,x)=-u_{\beta,xxx}(q,x),
 \qquad x\ge0.
\]
By parity, $w(q,0)=0$. Its linear parabolic equation has
nonnegative terminal data on the half-line, so comparison gives
\[
 w(q,x)\ge0.
\]
Thus $x\mapsto u_{\beta,xx}(q,x)$ is nonincreasing on
$[0,\infty)$.

At each fixed $\beta$, the control representation gives a uniform
bound on $u_\beta(q,x)-|x|$ as $|x|\to\infty$. Convexity and parity
therefore imply
\[
 \int_0^\infty u_{\beta,xx}(q,x)\,\dd x=1.
\]
Choose the common contact time
\[
 T_0
 =
 \min\left\{
 \frac18,\,
 \frac{1-\beta_0^{-1}}2
 \right\}>0.
\]
Since $q_\beta\ge1-\beta_0^{-1}$, we have
$T_0<q_\beta$ for every $\beta\ge\beta_0$.

At this contact point,
\[
 \E u_{\beta,xx}(T_0,X_{T_0})^2=1.
\]
Because $u_{\beta,xx}(T_0,x)$ is nonnegative, even, and decreasing
for $x\ge0$,
\[
 u_{\beta,xx}(T_0,0)\ge1.
\]
The identity
$\int_0^\infty u_{\beta,xx}(T_0,x)\,\dd x=1$ then implies
\[
 u_{\beta,xx}(T_0,4)\le\frac14.
\]
Consequently
\[
 \int_0^4 w(T_0,x)\,\dd x
 =
 u_{\beta,xx}(T_0,0)-u_{\beta,xx}(T_0,4)
 \ge\frac34.
\]

Define
\[
 W(s,x)=w(T_0-s,x),
 \qquad 0\le s\le T_0.
\]
Its equation on $x\ge0$ has the form
\[
 W_s
 =
 \frac12W_{xx}+B(s,x)W_x+C(s,x)W,
\]
with
\[
 B(s,x)
 =
 \gamma_\beta(T_0-s)
 u_{\beta,x}(T_0-s,x),
\]
and
\[
 C(s,x)
 =
 3\gamma_\beta(T_0-s)
 u_{\beta,xx}(T_0-s,x).
\]
For $x\ge0$, convexity and parity give
$0\le u_{\beta,x}\le1$, while
\eqref{eq:gs-scalar-crude-profile} gives
$\gamma_\beta\le2$ on this strip. Hence
\[
 0\le B\le2,
 \qquad
 C\ge0.
\]
The smoothing estimate \eqref{eq:gs-scalar-smoothing} supplies a
constant $M\ge1$, depending only on $\beta_0$, such that
\[
 |W_x|\le M.
\]
It also bounds the higher spatial derivatives appearing in the PDE.
Integrating the PDE in time therefore gives the absolute continuity
required by \zcref{lem:scalar-boundary-strictness}; no derivative of
$\gamma_\beta$ in $q$ is needed.

Applying that lemma on the interval of length $T_0$ yields
\begin{equation}\label{eq:gs-fourth-strict}
 -u_{\beta,xxxx}(0,0)
 =
 W_x(T_0,0)
 \ge
 a_0(T_0,M)
 =:c_0>0,
\end{equation}
where $c_0$ depends only on $\beta_0$.

\medskip
\emph{Step 3: Convert the fourth-derivative bound into linear growth
of $\gamma_\beta$.}
Along the scalar diffusion, write
\[
 h_q=u_{\beta,xx}(q,X_q),\qquad
 j_q=u_{\beta,xxx}(q,X_q),\qquad
 k_q=u_{\beta,xxxx}(q,X_q),
\]
and define the deterministic functions
\[
 \mathcal A(q)=\E j_q^2,
 \qquad
 \mathcal B(q)=\E h_q^3.
\]
On the contact interval, the scalar It\^o identities give
\begin{align}
 \mathcal A'(q)
 &=
 \E k_q^2
 -6\gamma_\beta(q)\E(h_qj_q^2),
 \label{eq:gs-A-evolution}\\
 \mathcal B'(q)
 &=
 -3\gamma_\beta(q)\E h_q^4
 +3\E(h_qj_q^2).
 \label{eq:gs-B-evolution}
\end{align}
Parity gives
\[
 \mathcal A(0)=0.
\]
Moreover, on the contact interval
\[
 \E h_q^2=S'(q)=1,
\]
so monotonicity of $L^p$ norms gives
\begin{equation}\label{eq:gs-B-lower}
 \mathcal B(q)=\E h_q^3
 \ge(\E h_q^2)^{3/2}=1.
\end{equation}

Let $\mathcal B_q$ denote the Brownian motion driving $X_q$. The
fourth derivative satisfies
\[
 \dd k_q
 =
 -\gamma_\beta(q)
 \bigl(4h_qk_q+3j_q^2\bigr)\,\dd q
 +u_{\beta,xxxxx}(q,X_q)\,\dd\mathcal B_q.
\]
On the strip supplied by \eqref{eq:gs-scalar-smoothing}, both the
drift and diffusion coefficients are bounded uniformly in $\beta$.
It\^o isometry and Cauchy--Schwarz therefore give
\[
 \E|k_q-k_0|^2
 \le C(q+q^2)
 \le Cq.
\]
By \eqref{eq:gs-fourth-strict},
$|k_0|\ge c_0$. After decreasing the common strip if necessary,
\[
 \E k_q^2\ge c_1>0.
\]
Using the uniform Hessian bound in
\eqref{eq:gs-A-evolution} gives
\[
 c_1-C\mathcal A(q)
 \le
 \mathcal A'(q)
 \le C
\]
for almost every sufficiently small $q$. Since
$\mathcal A(0)=0$, integration yields
\begin{equation}\label{eq:gs-A-linear}
 c_2q\le\mathcal A(q)\le C_2q
\end{equation}
on a common interval $[0,T_*]$, after decreasing $T_*$ if necessary.

We now use the contact identity once more. Since
$\E h_q^2=1$ on $[0,T_*]$, differentiating this constant expectation
by It\^o's formula gives, for almost every $q$,
\[
 0
 =
 \frac{\dd}{\dd q}\E h_q^2
 =
 \E j_q^2
 -2\gamma_\beta(q)\E h_q^3.
\]
Thus
\begin{equation}\label{eq:gs-gamma-quotient}
 \gamma_\beta(q)
 =
 \frac{\mathcal A(q)}{2\mathcal B(q)}
\end{equation}
for almost every $q\in[0,T_*]$. The functions
$\mathcal A$ and $\mathcal B$ are absolutely continuous, and
$\mathcal B\ge1$. Hence the quotient on the right is absolutely
continuous. Since $\gamma_\beta$ is right-continuous,
\eqref{eq:gs-gamma-quotient} holds for every $q\in[0,T_*)$.
Decreasing $T_*$ once more makes the identity valid on the closed
interval $[0,T_*]$.

In particular,
\[
 \gamma_\beta(0)=0.
\]
The smoothing estimates give a uniform upper bound on
$\mathcal B$, while \eqref{eq:gs-A-linear} and
\eqref{eq:gs-B-lower} give
\[
 kq\le\gamma_\beta(q)\le C_s q,
 \qquad 0\le q\le T_*,
\]
for constants depending only on $\beta_0$. Equations
\eqref{eq:gs-A-evolution}--\eqref{eq:gs-B-evolution} also give
uniform bounds on $\mathcal A'$ and $\mathcal B'$ on this strip.
Differentiating the quotient
\eqref{eq:gs-gamma-quotient} therefore yields
\[
 |\gamma_\beta'(q)|\le C_s
\]
for almost every $q\in[0,T_*]$. Since $\gamma_\beta$ is
nondecreasing,
\[
 0\le\gamma_\beta'(q)\le C_s
\]
almost everywhere. This proves the asserted near-zero regularity.

\medskip
\emph{Step 4: Cost of imposing the floor $\gamma\ge\lambda$.}
For $0<\lambda\le\lambda_*$, set
\[
 \gamma_{\beta,\lambda}^{\mathrm{tr}}
 =
 \max\{\lambda,\gamma_\beta\},
 \qquad
 v
 =
 \gamma_{\beta,\lambda}^{\mathrm{tr}}-\gamma_\beta.
\]
By the lower bound $\gamma_\beta(q)\ge kq$,
\[
 v(q)=0
 \qquad\text{for }q\ge T_\lambda,
 \qquad
 T_\lambda:=\frac{\lambda}{k}.
\]
Choose $\lambda_*>0$, depending only on $\beta_0$, small enough that
\[
 T_\lambda\le T_*
\]
and that all profiles considered below satisfy the uniform early-strip
bounds from Step~1. We also have
\[
 0\le v\le\lambda.
\]

For $0\le\theta\le1$, interpolate between the unconstrained profile
and the trial profile by
\[
 \gamma_\theta
 =
 \gamma_\beta+\theta v.
\]
Let $u_\theta$ be the corresponding scalar solution and
$X_q^\theta$ its scalar diffusion started from zero. Expectations in
the remainder of the proof are taken under this fixed
$\theta$-profile law and are denoted by $\E_\theta$. Set
\[
 m_\theta(q)
 =
 u_{\theta,x}(q,X_q^\theta),
 \qquad
 h_\theta(q)
 =
 u_{\theta,xx}(q,X_q^\theta).
\]
The smoothing argument of Step~1 applies uniformly to
$\gamma_\theta$. Since zero external field gives
$m_\theta(0)=0$, the martingale identity for $m_\theta$ yields
\begin{equation}\label{eq:gs-trial-m-bound}
 \|m_\theta(q)\|_{L^2(\E_\theta)}
 \le C\sqrt q,
 \qquad 0\le q\le T_\lambda.
\end{equation}

Let
\[
 z_\theta=\partial_\theta u_\theta.
\]
For $q\le T_\lambda$, let $X_s^{q,x;\theta}$ denote the scalar
diffusion for $\gamma_\theta$ started from $x$ at time $q$, and let
$\E_{q,x}^\theta$ denote expectation under that law. Since
$v=0$ above $T_\lambda$, differentiation of the PDE gives
\[
 z_\theta(q,x)
 =
 \frac12\E_{q,x}^\theta
 \int_q^{T_\lambda}
 v(s)
 u_{\theta,x}(s,X_s^{q,x;\theta})^2\,\dd s.
\]
The spatial derivative of the diffusion flow is
\[
 J_{q,s}^\theta
 =
 \partial_xX_s^{q,x;\theta}
 =
 \exp\left(
 \int_q^s
 \gamma_\theta(r)
 u_{\theta,xx}(r,X_r^{q,x;\theta})\,\dd r
 \right).
\]
The uniform Hessian bound on the early strip gives
\[
 |J_{q,s}^\theta|\le C.
\]
Differentiating the representation for $z_\theta$ therefore yields
\[
 z_{\theta,x}(q,x)
 =
 \E_{q,x}^\theta
 \int_q^{T_\lambda}
 v(s)
 u_{\theta,x}(s,X_s^{q,x;\theta})
 u_{\theta,xx}(s,X_s^{q,x;\theta})
 J_{q,s}^\theta\,\dd s.
\]
Using the uniform Hessian and flow bounds, followed by conditional
Jensen and Minkowski, gives
\[
 \begin{aligned}
 \|z_{\theta,x}(q,X_q^\theta)\|_{L^2(\E_\theta)}
 &\le
 C\lambda
 \int_q^{T_\lambda}
 \|m_\theta(s)\|_{L^2(\E_\theta)}\,\dd s\\
 &\le
 C\lambda T_\lambda^{3/2}.
 \end{aligned}
\]
Here we used \eqref{eq:gs-trial-m-bound} in the second line.

To control the second variation, set
\[
 y_\theta=\partial_\theta^2u_\theta
\]
and define
\[
 \mathcal L_\theta
 =
 \partial_q+\frac12\partial_{xx}
 +\gamma_\theta u_{\theta,x}\partial_x.
\]
Differentiating the PDE once and twice in $\theta$ gives
\[
 \mathcal L_\theta z_\theta
 =
 -\frac12v u_{\theta,x}^2,
\]
and
\[
 \mathcal L_\theta y_\theta
 =
 -2v u_{\theta,x}z_{\theta,x}
 -\gamma_\theta z_{\theta,x}^2,
\]
with
\[
 z_\theta(1,\cdot)=y_\theta(1,\cdot)=0.
\]
Both functions vanish identically above $T_\lambda$. Applying
It\^o's formula under the fixed $\theta$-profile law and using
$\gamma_\theta\le2$ on the relevant strip gives
\[
 \begin{aligned}
 |y_\theta(0,0)|
 &\le
 2\lambda\int_0^{T_\lambda}
 \|m_\theta(q)\|_2
 \|z_{\theta,x}(q,X_q^\theta)\|_2\,\dd q\\
 &\qquad
 +2\int_0^{T_\lambda}
 \|z_{\theta,x}(q,X_q^\theta)\|_2^2\,\dd q\\
 &\le
 C\lambda^2T_\lambda^3
 +C\lambda^2T_\lambda^4\\
 &\le
 C\lambda^2T_\lambda^3.
 \end{aligned}
\]
All $L^2$ norms in this display are taken with respect to
$\E_\theta$.

The penalty in $\cP_\beta^{\mathrm{en}}$ is affine in $\theta$.
Hence
\[
 \frac{\dd^2}{\dd\theta^2}
 \cP_\beta^{\mathrm{en}}(\gamma_\theta)
 =
 y_\theta(0,0).
\]
The perturbation $v$ is supported in
$[0,T_\lambda]\subset[0,q_\beta]$, where the scalar contact identity
holds. Therefore the first variation at $\theta=0$ vanishes:
\[
 \left.
 \frac{\dd}{\dd\theta}
 \cP_\beta^{\mathrm{en}}(\gamma_\theta)
 \right|_{\theta=0}
 =0.
\]
Taylor's formula with integral remainder now gives
\[
 \begin{aligned}
 A_\beta(\lambda)-E_\beta(0)
 &=
 \cP_\beta^{\mathrm{en}}
   (\gamma_{\beta,\lambda}^{\mathrm{tr}})
 -
 \cP_\beta^{\mathrm{en}}(\gamma_\beta)\\
 &\le
 C\lambda^2T_\lambda^3
 \le
 C\lambda^5,
 \end{aligned}
\]
because $T_\lambda=\lambda/k$. Finally,
$\gamma_\beta$ minimizes the unconstrained problem and
$\gamma_{\beta,\lambda}^{\mathrm{tr}}$ is admissible for the
constrained problem, so
\[
 E_\beta(0)
 \le
 E_\beta(\lambda)
 \le
 A_\beta(\lambda).
\]
Combining these inequalities proves
\eqref{eq:gs-scalar-trial-cost}.
\end{proof}

The upper bound in \zcref{lem:gs-scalar-root} was obtained from one
particular trial profile. For the matching lower bound, we must show
that no admissible profile satisfying the floor constraint can have a
smaller-order cost.

\begin{proposition}
\zlabel{prop:scalar-lower-via-fisher}
There exist constants
\[
 c_0>0,\qquad \lambda_0>0,
\]
depending only on $\beta_0$, such that the following holds. For every
$\beta\ge\beta_0$, every $0<\lambda\le\lambda_0$, and every admissible
scalar profile $\eta$ satisfying
\[
 \eta(q)\ge\lambda,
 \qquad 0\le q<1,
\]
one has
\[
 \cP_\beta^{\mathrm{en}}(\eta)-E_\beta(0)
 \ge c_0\lambda^5.
\]
Equivalently,
\[
 E_\beta(\lambda)-E_\beta(0)\ge c_0\lambda^5
 \qquad
 (\beta\ge\beta_0,\ 0<\lambda\le\lambda_0).
\]
\end{proposition}
\begin{proof}
Fix $\beta\ge\beta_0$, $0<\lambda\le\lambda_0$, and an admissible
profile $\eta$ satisfying $\eta\ge\lambda$. Set
\[
 D
 =
 \cP_\beta^{\mathrm{en}}(\eta)-E_\beta(0).
\]
We first prove the desired bound under the additional assumption
$D\le1$. At the end we remove this restriction.

\medskip
\emph{Step 1: Uniform bounds along the segment from the minimizer to
the competitor.}
The profile that vanishes on $[0,1)$ is admissible for the
unconstrained scalar problem. If $Z$ is standard Gaussian, it gives
\[
 E_\beta(0)
 \le
 \E f_\beta(Z)
 \le
 \sqrt{\frac2\pi}+\frac{\log2}{\beta}
 \le
 \sqrt{\frac2\pi}+\log2
 <2.
\]
Hence $D\le1$ implies
\[
 \cP_\beta^{\mathrm{en}}(\eta)<3.
\]

We next convert this sublevel bound into a pointwise bound on $\eta$
near the origin. The constant control $b\equiv1$ in the scalar control
representation gives
\[
 \cP_\beta^{\mathrm{en}}(\eta)
 \ge
 \frac12\int_0^1(1-q)\eta(q)\,\dd q.
\]
Since $\eta$ is nondecreasing,
\[
 \frac12\int_0^1(1-q)\eta(q)\,\dd q
 \ge
 \frac{\eta(1/2)}{16}.
\]
Therefore
\[
 \eta(q)\le\eta(1/2)\le48,
 \qquad 0\le q\le\frac12.
\]
By \zcref{lem:gs-scalar-root}, 
\[
 \gamma_\beta(q)\le2,
 \qquad 0\le q\le\frac12.
\]
Thus every convex interpolation
\[
 \gamma_\theta
 =
 (1-\theta)\gamma_\beta+\theta\eta,
 \qquad 0\le\theta\le1,
\]
satisfies
\begin{equation}\label{eq:gs-lower-segment-bound}
 0\le\gamma_\theta(q)\le48,
 \qquad 0\le q\le\frac12.
\end{equation}

We now invoke the response estimates
\zcref{prop:gs-response-tools,lem:response-fisher} with
$N=1$, $t=0$, $r=1$, and outer tilt equal to zero. The floor constraint $\eta\ge\lambda$ is unrelated to the outer tilt;
the latter is zero for both profiles $\gamma_\beta$ and $\eta$. 

For a fixed $\theta$, let $X^\theta$ denote the scalar diffusion
associated with $\gamma_\theta$, and let $\E_\theta$ denote
expectation under that diffusion law. Write
\[
 h_\theta(q)
 =
 u_{\gamma_\theta,xx}(q,X_q^\theta),
\]
and define
\[
 b_\theta(q)=\E_\theta h_\theta(q),
 \qquad
 \zeta_\theta(q)=\E_\theta h_\theta(q)^2.
\]
The response estimates also provide nonnegative functions
$e_\theta$ and $\mathcal J_\theta$. We use only the following two
facts about them:
\[
 \frac{\dd^2}{\dd\theta^2}
 \cP_\beta^{\mathrm{en}}(\gamma_\theta)
 =
 \int_0^1 e_\theta(q)\,\dd q,
\]
and the pairing and cutoff inequalities
\eqref{eq:shared-pairing} and \eqref{eq:shared-fisher}.

The pointwise curvature estimate in
\zcref{lem:response-curvature}, applied with terminal time $1/2$ and
profile bound $48$, gives
\[
 u_{\gamma_\theta,xx}(q,x)
 \le
 48+(1/2-q)^{-1/2}
 \le50,
 \qquad 0\le q\le\frac14.
\]
The averaged curvature estimate
\eqref{eq:shared-averaged-curvature}, with the same profile bound,
gives a constant $c_{48}>0$ such that
\[
 \zeta_\theta(q)\ge c_{48}^2.
\]
Consequently, uniformly in $\theta$,
\begin{equation}\label{eq:gs-lower-curvature-window}
 b_\theta(q)\le50,
 \qquad
 c_{48}^2\le\zeta_\theta(q)\le2500,
 \qquad
 0\le q\le\frac14.
\end{equation}

\medskip
\emph{Step 2: The floor forces a perturbation of size $\lambda$ on
an interval of length of order $\lambda$.}
Let $C_r$ be the upper-slope constant from
\zcref{lem:gs-scalar-root}, enlarged so that $C_r\ge1$, and define
\[
 T_\lambda=\frac{\lambda}{2C_r}.
\]
Choose $\lambda_0>0$, depending only on $\beta_0$, so small that
\[
 T_\lambda\le T_*,
 \qquad
 T_\lambda\le\frac14
\]
whenever $0<\lambda\le\lambda_0$.

Set
\[
 v=\eta-\gamma_\beta.
\]
Since $\eta\ge\lambda$ and
$\gamma_\beta(q)\le C_rq$ on $[0,T_*]$, for
$0\le q\le T_\lambda$ we have
\begin{equation}\label{eq:gs-lower-v}
 v(q)
 \ge
 \lambda-C_rT_\lambda
 =
 \frac{\lambda}{2}.
\end{equation}

Choose a fixed nonnegative Lipschitz cutoff $\varphi$ supported in
\[
 [T_\lambda/4,3T_\lambda/4],
\]
with height at most one and satisfying
\[
 \int_0^1\varphi(q)\,\dd q\ge cT_\lambda,
 \qquad
 \int_0^1|\varphi'(q)|^2\,\dd q\le\frac{C}{T_\lambda}.
\]
The constants $c,C$ are numerical.

Apply the general-profile cutoff estimate
\eqref{eq:shared-fisher}. Using
\eqref{eq:gs-lower-curvature-window},
\[
 \begin{aligned}
 \int_0^1\varphi(q)^2\mathcal J_\theta(q)\,\dd q
 &\le
 \int_0^1|\varphi'(q)|^2b_\theta(q)\,\dd q
 +
 \int_0^1\varphi(q)^2\zeta_\theta(q)\,
       \dd\gamma_\theta(q)\\
 &\le
 \frac{C}{T_\lambda}+C.
 \end{aligned}
\]
Indeed, the second Stieltjes integral is supported below
$T_\lambda\le1/4$ and is at most
\[
 2500\,\gamma_\theta(T_\lambda)\le C,
\]
by \eqref{eq:gs-lower-segment-bound}. Since $T_\lambda\le1$, we may
absorb this constant into the first term and conclude that
\begin{equation}\label{eq:gs-lower-fisher-cutoff}
 \int_0^1\varphi^2\mathcal J_\theta\,\dd q
 \le
 \frac{C}{T_\lambda}.
\end{equation}
The cutoff vanishes in a neighborhood of $q=0$, so an atom of
$\dd\gamma_\theta$ at zero makes no contribution. Atoms inside the
support of $\varphi$ are automatically included in the Stieltjes
integral.

\medskip
\emph{Step 3: A uniform lower bound on the second variation.}
By \eqref{eq:gs-lower-v},
\eqref{eq:gs-lower-curvature-window}, and the lower bound on
$\int\varphi$,
\[
 \int_0^1
 v(q)\zeta_\theta(q)\varphi(q)\,\dd q
 \ge
 c\lambda T_\lambda.
\]
The response pairing estimate \eqref{eq:shared-pairing}, together with
\eqref{eq:gs-lower-curvature-window}, the cutoff bounds above, and
\eqref{eq:gs-lower-fisher-cutoff}, gives
\[
 \begin{aligned}
 \int_0^1
 v\zeta_\theta\varphi\,\dd q
 &\le
 \left(\int_0^1 e_\theta(q)\,\dd q\right)^{1/2}
 \left[
 \left(\int_0^1b_\theta(q)|\varphi'(q)|^2\,\dd q\right)^{1/2}
 +
 2\left(\int_0^1\varphi(q)^2
          \mathcal J_\theta(q)\,\dd q\right)^{1/2}
 \right]\\
 &\le
 C T_\lambda^{-1/2}
 \left(\int_0^1e_\theta(q)\,\dd q\right)^{1/2}.
 \end{aligned}
\]
Combining the last two displays yields
\[
 \int_0^1e_\theta(q)\,\dd q
 \ge
 c\lambda^2T_\lambda^3.
\]
Therefore, for every $0<\theta<1$,
\begin{equation}\label{eq:gs-lower-second-variation}
 \frac{\dd^2}{\dd\theta^2}
 \cP_\beta^{\mathrm{en}}(\gamma_\theta)
 \ge
 c\lambda^2T_\lambda^3.
\end{equation}

\medskip
\emph{Step 4: Convert strict convexity along the segment into the
floor cost.}
For any function $f$ whose first derivative is absolutely continuous,
\begin{equation}\label{eq:midpoint-convexity-gap}
 f(0)+f(1)-2f(1/2)
 =
 \int_0^1
 \min\{\theta,1-\theta\}f''(\theta)\,\dd\theta.
\end{equation}
The weight on the right has integral $1/4$.

Apply this identity to
\[
 f(\theta)
 =
 \cP_\beta^{\mathrm{en}}(\gamma_\theta).
\]
Using \eqref{eq:gs-lower-second-variation},
\[
 \begin{aligned}
 &\cP_\beta^{\mathrm{en}}(\eta)
 +\cP_\beta^{\mathrm{en}}(\gamma_\beta)
 -2\cP_\beta^{\mathrm{en}}
   \left(\frac{\eta+\gamma_\beta}{2}\right)\\
 &\qquad\ge
 c\lambda^2T_\lambda^3.
 \end{aligned}
\]
Both $\gamma_\beta$ and
$(\eta+\gamma_\beta)/2$ are admissible for the unconstrained scalar
problem, and $\gamma_\beta$ is a minimizer. Hence
\[
 \cP_\beta^{\mathrm{en}}(\gamma_\beta)=E_\beta(0),
 \qquad
 \cP_\beta^{\mathrm{en}}
 \left(\frac{\eta+\gamma_\beta}{2}\right)
 \ge E_\beta(0).
\]
It follows that
\[
 \cP_\beta^{\mathrm{en}}(\eta)-E_\beta(0)
 \ge
 c\lambda^2T_\lambda^3.
\]
Since
\[
 T_\lambda=\frac{\lambda}{2C_r},
\]
we obtain
\[
 \cP_\beta^{\mathrm{en}}(\eta)-E_\beta(0)
 \ge
 c_0\lambda^5
\]
for a constant $c_0>0$ depending only on $\beta_0$.

This proves the result for competitors with $D\le1$. To remove that
restriction, decrease $c_0$ if necessary so that
\[
 c_0\lambda_0^5\le1.
\]
If $D>1$ and $0<\lambda\le\lambda_0$, then automatically
\[
 D>1\ge c_0\lambda^5,
\]
so the same lower bound holds for every admissible competitor.

At each fixed finite $\beta$, the scalar terminal function is smooth
and strictly convex. The inverse-Hessian identities used in the
response estimates may therefore be justified by first stopping the
scalar diffusion on bounded spatial intervals and then removing the
stopping. The exponential bounds in
\zcref{app:response-estimates} provide the required domination.
\end{proof}

After this proposition, decrease the value of $\lambda_*$ from
\zcref{lem:gs-scalar-root}, if necessary, so that
\[
 \lambda_*
 \le
 \min\{1,\beta_0/2,\lambda_0\}.
\]
Then the upper and lower  estimates hold on the same
$\lambda$-interval.

\subsection{Positive moments and the mean deficit}

We next relate the finite-size mean deficit to the variance. The
mechanism is the following. Convexity of the normalized logarithmic
moment gives a linear gain of order
\[
 \lambda\,\frac{\Var X}{N}
\]
at small positive tilt, while \zcref{lem:gs-scalar-root} bounds the
increase of the scalar variational value by order $\lambda^5$.
Balancing these two terms gives
$\lambda\asymp(\Var X/N)^{1/4}$ and hence a deficit of order
$(\Var X/N)^{5/4}$. Once an upper bound on the mean deficit is
available, this relation will give an upper bound on the variance.
Combined with the variance lower bound proved below, it will also give
a lower bound on the mean deficit.

\begin{lemma}
\zlabel{lem:short-variance-mean}
There exists a constant $c>0$, depending only on $\beta_0$, such that
for every $N\ge1$ and $\beta\ge\beta_0$,
\[
 E_\beta(0)-\frac{\E F_{N,\beta}}{N\beta}
 \ge
 c\left(
 \frac{\Var(F_{N,\beta}/\beta)+1/2}{N}
 \right)^{5/4}.
\]
The same bound holds at zero temperature:
\[
 e_*-\frac{\E G_N}{N}
 \ge
 c\left(
 \frac{\Var(G_N)+1/2}{N}
 \right)^{5/4}.
\]
\end{lemma}

\begin{proof}
We work first at finite temperature. Add the independent diagonal
completion and set
\[
 X=\frac{F_{N,\beta}}{\beta}+D_N,
 \qquad
 \Var(D_N)=\frac12.
\]
Since $D_N$ is independent of $F_{N,\beta}$,
\[
 \Var(X)
 =
 \Var(F_{N,\beta}/\beta)+\frac12.
\]
Define
\[
 v=\frac{\Var(X)}{N}
\]
and
\[
 \psi(\lambda)
 =
 \frac1{N\lambda}\log\E e^{\lambda X},
 \qquad
 \psi(0)=\frac{\E X}{N}
 =\frac{\E F_{N,\beta}}{N\beta}.
\]

As a function of the completed Gaussian coordinates, $X$ is convex.
Its gradient is a convex combination of the spin feature vectors,
each of squared norm $N/2$. Hence $X$ is
$\sqrt{N/2}$-Lipschitz. Gaussian Poincar\'e therefore gives
\[
 \Var(X)\le\frac N2,
 \qquad\text{so}\qquad
 0<v\le\frac12.
\]
The strict positivity also follows directly from the independent
completion, whose variance is $1/2$.

By \zcref{lem:short-laplace-convexity}, the function
\[
 \lambda\longmapsto
 \frac1\lambda\log\E e^{\lambda X}
\]
is convex and has derivative $\Var(X)/2$ at zero. Its tangent
inequality at zero therefore gives, for $\lambda\ge0$,
\[
 \psi(\lambda)-\psi(0)
 \ge
 \frac{v\lambda}{2}.
\]
On the other hand, the finite-size comparison
\eqref{eq:gs-full-fv} gives
\[
 \psi(\lambda)\le E_\beta(\lambda),
\]
and the scalar trial estimate
\eqref{eq:gs-scalar-trial-cost} gives, for
$0<\lambda\le\lambda_*$,
\[
 E_\beta(\lambda)
 \le
 E_\beta(0)+C_s\lambda^5.
\]
Combining these inequalities,
\begin{equation}\label{eq:variance-mean-balance}
 \frac{v\lambda}{2}
 \le
 E_\beta(0)-\psi(0)+C_s\lambda^5.
\end{equation}

We now choose the tilt so that the fifth-order term is at most one
half of the linear gain. Set
\[
 a_0
 =
 \min\left\{
 1,\lambda_*,
 (4C_s)^{-1/4}
 \right\},
 \qquad
 \lambda=a_0v^{1/4}.
\]
Because $v\le1/2$, this choice satisfies $\lambda\le\lambda_*$, and
\[
 C_s\lambda^5
 =
 C_sa_0^5v^{5/4}
 \le
 \frac{a_0}{4}v^{5/4}
 =
 \frac{v\lambda}{4}.
\]
Substituting into \eqref{eq:variance-mean-balance} gives
\[
 \begin{aligned}
 E_\beta(0)-\psi(0)
 &\ge
 a_0\left(\frac12-C_sa_0^4\right)v^{5/4}\\
 &\ge
 \frac{a_0}{4}v^{5/4}.
 \end{aligned}
\]
Since
\[
 \psi(0)=\frac{\E F_{N,\beta}}{N\beta},
 \qquad
 v=
 \frac{\Var(F_{N,\beta}/\beta)+1/2}{N},
\]
this proves the finite-temperature inequality.

It remains to pass to zero temperature. For fixed $N$, the
soft-maximum estimate gives
\[
 \frac{F_{N,\beta}}{\beta}\longrightarrow G_N
\]
uniformly in the Gaussian disorder, up to the deterministic error
$N\log2/\beta$. In particular, the convergence holds in $L^2$, so
\[
 \E\frac{F_{N,\beta}}{\beta}\longrightarrow\E G_N,
 \qquad
 \Var(F_{N,\beta}/\beta)\longrightarrow\Var(G_N).
\]
By \zcref{lem:zero-limit},
\[
 E_\beta(0)\longrightarrow e_*.
\]
The constants above are uniform for $\beta\ge\beta_0$, so letting
$\beta\to\infty$ proves
\[
 e_*-\frac{\E G_N}{N}
 \ge
 c\left(
 \frac{\Var(G_N)+1/2}{N}
 \right)^{5/4}.
\]
\end{proof}

We next need an upper bound on small positive centered moments whose
admissible tilt is determined by the variance. Apply
\eqref{eq:variance-sensitive-mgf} to the finite-temperature SK log
partition function and rescale the tilt by
$s=\lambda/\beta$. With
\[
 X_{N,\beta}=\frac{F_{N,\beta}}{\beta},
 \qquad
 K_{N,\beta}(\lambda)
 =
 \log\E
 e^{\lambda(X_{N,\beta}-\E X_{N,\beta})},
\]
we obtain
\begin{equation}\label{eq:gs-energy-mgf}
 K_{N,\beta}(\lambda)
 \le
 B\lambda^2
 \left(
 \Var X_{N,\beta}+\beta^{-2}
 \right)
\end{equation}
whenever
\[
 0\le
 \lambda
 \sqrt{\Var X_{N,\beta}+\beta^{-2}}
 \le b.
\]
The constants $B,b>0$ are numerical and may be chosen so that
\[
 Bb^2=1.
\]
Thus the available positive-tilt window has size comparable to
\[
 \left(
 \Var X_{N,\beta}+\beta^{-2}
 \right)^{-1/2}.
\]
On the temperature range $\beta\ge\beta_0$, the auxiliary term is
uniformly bounded by
\[
 \beta^{-2}\le\beta_0^{-2}.
\]

\subsection{Reflected regularization and the uniform comparison}
\zlabel{subsec:abs-section}

The refined comparison will use a local regularization of the free
mass $\nu$ below a specified quantile. The amount of mass moved by
this regularization is small by construction, while the resulting
change in the variational functional is quadratic in the smoothing
scale. To make this estimate uniform, we first record the convexity
and first-variation identities that will be used along the segment
joining the original profile to its regularization.

Fix finite $N\ge1$, $\beta>0$, $0<r\le1$, and
$0\le\lambda\le\beta$. Set
\[
 t=1-r,\qquad
 J_{ij}=\sqrt{\frac{t}{N}}\,g_{ij}\quad(i<j),
 \qquad
 J_{ii}=0,
\]
and write
\[
 F_r(\gamma)=\mathfrak F_t^c(\gamma)
\]
for the rescaled finite-$N$ functional in
\eqref{eq:gs-energy-interpolation}, with outer tilt $\lambda$.
When $\lambda=0$, its logarithmic outer moment is interpreted as the
corresponding expectation.

An admissible profile has total mass
\[
 \dd\gamma([0,1])=\beta
\]
and may be written
\[
 \dd\gamma=\lambda\delta_0+\nu,
 \qquad \nu\ge0,
 \qquad \nu([0,1])=\beta-\lambda.
\]
Thus $\lambda\delta_0$ is the prescribed root mass and $\nu$ is the
remaining free mass. The independent diagonal completion contributes
only the deterministic term $\lambda t/(4N)$ to $F_r$ and therefore
plays no role in the profile variations below.

For a fixed profile $\gamma$, take expectations under its normalized
joint law and define
\[
 S_\gamma(q)
 =
 \frac1N\E_\gamma|m_q|^2.
\]
The response function associated with the first variation is
\[
 G_\gamma(q)
 =
 \int_q^1\bigl(S_\gamma(s)-s\bigr)\,\dd s.
\]
If $\nu'\ge0$ is any measure of total mass $\beta-\lambda$, then the
segment obtained by replacing $\nu$ with
$(1-\theta)\nu+\theta\nu'$ has cumulative profile
\[
 \gamma_\theta(q)
 =
 \gamma(q)
 +\theta(\nu'-\nu)([0,q]).
\]
The functional $F_r$ is convex along every such segment, and its
one-sided derivative at $\theta=0$ is
\begin{equation}\label{eq:abs-first-variation}
 DF_r(\gamma)[\nu'-\nu]
 =
 \frac r2
 \int_0^1G_\gamma(q)\,(\nu'-\nu)(\dd q).
\end{equation}
Moreover,
\begin{equation}\label{eq:abs-semiconcavity}
 G_\gamma'(0)=0,
 \qquad
 G_\gamma''(q)\le1
\end{equation}
at every point where the second derivative exists, and hence in the
distributional sense on $[0,1]$.

We briefly justify these facts. Fix the coupling matrix $J$. Let $B$
be standard $N$-dimensional Brownian motion, and let $\E_B$ denote
expectation over its law. The admissible controls below are predictable
with respect to the augmented natural filtration of $B$ and satisfy
\[
 |b_i(q)|\le1
\]
for every coordinate, almost everywhere in time. The
finite-temperature control representation is
\[
 u_\gamma(0,0;J)
 =
 \sup_b
 \E_B\left[
 f_{\beta,J}\left(
   \sqrt r\,B_1
   +r\int_0^1\gamma(q)b_q\,\dd q
 \right)
 -\frac r2\int_0^1\gamma(q)|b_q|^2\,\dd q
 \right].
\]
For each fixed control, the endpoint inside $f_{\beta,J}$ depends
affinely on $\gamma$, while $f_{\beta,J}$ is convex and the running
cost is affine. Hence the control objective is convex in $\gamma$,
and taking the supremum preserves convexity. For $\lambda>0$, the map
\[
 u\longmapsto
 \frac1{N\lambda}\log\E_J e^{\lambda u}
\]
also preserves convexity by H\"older's inequality. At $\lambda=0$,
the same conclusion follows by taking expectation. Since the penalty
in $F_r$ is affine in $\gamma$, the full functional is convex.

The control representation itself follows from It\^o's formula and
completion of the square; the maximizing feedback is
\[
 b_q=\nabla u_\gamma(q,Y_q).
\]
Now perturb the cumulative profile in a deterministic direction $v$.
Linearizing the PDE and evaluating along the feedback diffusion for
the unperturbed profile gives, for fixed $J$,
\[
 \left.
 \frac{\dd}{\dd\theta}
 u_{\gamma+\theta v}(0,0;J)
 \right|_{\theta=0}
 =
 \frac r2
 \E_B\int_0^1v(q)|m_q|^2\,\dd q.
\]
Differentiating the normalized outer moment and applying Fubini then
gives \eqref{eq:abs-first-variation} after writing
\[
 v(q)=(\nu'-\nu)([0,q])
\]
and integrating once in the measure variable.

The bound \eqref{eq:abs-semiconcavity} follows from the same diffusion.
Zero-field spin-flip symmetry gives
\[
 m_0=0
\]
for every fixed $J$, and hence
\[
 S_\gamma(0)=0.
\]
Differentiating the PDE and applying It\^o's formula gives
\[
 \dd m_q=\sqrt r\,H_q\,\dd B_q.
\]
Therefore
\[
 S_\gamma'(q)
 =
 \frac rN\E_\gamma\Tr H_q^2
 \ge0.
\]
Since
\[
 G_\gamma'(q)=q-S_\gamma(q),
\]
we obtain
\[
 G_\gamma'(0)=0,
 \qquad
 G_\gamma''(q)=1-S_\gamma'(q)\le1.
\]

The preceding calculations are first justified for step profiles.
At fixed finite $N$ and $\beta$, all spatial derivatives of the
terminal log-sum-exp of any fixed order are bounded. In particular,
\[
 |\partial_i u|\le1,
\]
and the conditional-covariance representation of the Hessian gives
\[
 0\preceq H\preceq\beta C\preceq N\beta I.
\]
Successive differentiation of the PDE then bounds every further
fixed spatial derivative needed here. The control representation also
gives the profile-continuity estimate
\[
 \sup_{q,x,J}
 |u_\gamma(q,x;J)-u_{\widetilde\gamma}(q,x;J)|
 \le
 \frac{3rN}{2}
 \|\gamma-\widetilde\gamma\|_{L^1([0,1])}.
\]
Step approximation, the corresponding differentiated difference
equations, a common Brownian coupling, and Gaussian domination of the
outer weights therefore extend the convexity, first-variation, and
response identities to arbitrary admissible cumulative profiles.
Equivalently, one may first pass the integrated first-variation
identity along each admissible segment and then take its one-sided
derivative. All constants in this approximation step may depend on
the fixed $N$ and $\beta$.

We now construct the reflected regularization. It leaves the
prescribed atom $\lambda\delta_0$ unchanged and moves only a chosen
portion of the free mass $\nu$. The next lemma quantifies both the
amount of mass moved and the resulting increase in $F_r$.

\begin{lemma}\zlabel{lem:p2-smoothing}
Let $\gamma$ be an admissible profile with
\[
 \dd\gamma=\lambda\delta_0+\nu,
 \qquad
 \nu([0,1])=\beta-\lambda.
\]
Fix
\[
 0<T<1,\qquad 0<h\le1-T,
\]
and let
\[
 m=\nu([0,T))
\]
be the amount of free mass below $T$. Let $U$ be uniform on
$[-h,h]$. Define a deterministic measure $\bar\nu$ on $[0,1]$ by
\[
 \bar\nu(B)
 =
 \int_{[0,T)}
 \Prob_U\{|q+U|\in B\}\,\nu(\dd q)
 +
 \nu(B\cap[T,1])
\]
for every Borel set $B\subset[0,1]$, and let $\bar\gamma$ be the
cumulative profile associated with
\[
 \dd\bar\gamma=\lambda\delta_0+\bar\nu.
\]

Thus the prescribed atom $\lambda\delta_0$ and all free mass on
$[T,1]$ are left unchanged, while each unit of free mass at
$q<T$ is replaced by the law of the reflected displacement
$|q+U|$. Then $\bar\gamma$ is admissible,
\[
 \dd\bar\gamma([0,1])=\beta,
 \qquad
 \bar\gamma(0)=\lambda,
\]
and the restriction of $\dd\bar\gamma$ to $(0,T)$ is absolutely
continuous with respect to Lebesgue measure, with density bounded by
\[
 \frac{m}{h}.
\]

Moreover, there exists a finite measure $\pi$ on $[0,1]^2$ whose
marginals are $\dd\gamma$ and $\dd\bar\gamma$ and whose support is
contained in
\[
 \{(q,\bar q):|q-\bar q|\le h\}.
\]
Consequently, after extending both cumulative functions by $0$ on
$(-\infty,0)$ and by $\beta$ on $(1,\infty)$,
\[
 \gamma(x-h)\le\bar\gamma(x)\le\gamma(x+h),
\]
and
\[
 \bar\gamma(x-h)\le\gamma(x)\le\bar\gamma(x+h)
\]
for every $x\in\mathbb R$.

Finally, the increase in the finite-$N$ variational functional is
bounded by
\begin{equation}\label{eq:p2-smoothing-cost}
 F_r(\bar\gamma)-F_r(\gamma)
 \le
 \frac{rmh^2}{12}.
\end{equation}
If $\gamma$ minimizes $F_r$ over the admissible profiles, then also
\[
 0
 \le
 F_r(\bar\gamma)-F_r(\gamma)
 \le
 \frac{rmh^2}{12}.
\]
\end{lemma}

\begin{proof}
We verify first the measure-theoretic properties of the
regularization, then the displacement bound, and finally the
variational cost.

\medskip
\emph{Admissibility and density.}
Fix $q\in[0,T)$. The law of $|q+U|$ has, for $x>0$, density
\[
 \frac1{2h}
 \left(
   \mathbf 1_{\{|x-q|\le h\}}
   +
   \mathbf 1_{\{|x+q|\le h\}}
 \right)
 \le \frac1h.
\]
Because $U$ has a continuous distribution, this reflected law has no
atom at zero. Moreover,
\[
 |q+U|\le q+h<T+h\le1,
\]
so its support is contained in $[0,1]$.

Integrating these kernels against the free mass on $[0,T)$ shows that
the regularized part of $\bar\nu$ has density at most $m/h$ on
$(0,T)$. The free mass on $[T,1]$, including any atoms at $T$ or at
$1$, is left unchanged. The prescribed atom
$\lambda\delta_0$ is also unchanged. Hence
\[
 \bar\gamma(0)=\lambda,
 \qquad
 \dd\bar\gamma([0,1])=\beta,
\]
and $\bar\gamma$ is admissible.

\medskip
\emph{Displacement coupling.}
Couple each point $q<T$ with the reflected point
\[
 \bar q=|q+U|,
\]
and couple each point $q\ge T$, as well as the prescribed root atom,
with itself. Since $q\ge0$,
\[
 \bigl||q+U|-q\bigr|
 \le |U|
 \le h.
\]
This gives a coupling of $\dd\gamma$ and $\dd\bar\gamma$ supported on
pairs $(q,\bar q)$ satisfying $|q-\bar q|\le h$.

The corresponding cumulative functions therefore satisfy, after
extension by $0$ below zero and by $\beta$ above one,
\[
 \gamma(x-h)\le\bar\gamma(x)\le\gamma(x+h)
\]
and, by symmetry of the same displacement bound,
\[
 \bar\gamma(x-h)\le\gamma(x)\le\bar\gamma(x+h).
\]

\medskip
\emph{Variational cost.}
We estimate the cost using the first variation at the regularized
profile. Let
\[
 G(q)=G_{\bar\gamma}(q)
 =
 \int_q^1\bigl(S_{\bar\gamma}(s)-s\bigr)\,\dd s.
\]
By \eqref{eq:abs-semiconcavity},
\[
 G'(0)=0,
 \qquad
 G''\le1.
\]
Extend $G$ evenly to $[-1,1]$ by
\[
 \widetilde G(x)=G(|x|).
\]
We claim that $\widetilde G$ is $1$-semiconcave, equivalently that
\[
 x\longmapsto \widetilde G(x)-\frac{x^2}{2}
\]
is concave.

Indeed, for $x>0$,
\[
 \frac{\dd}{\dd x}
 \left(\widetilde G(x)-\frac{x^2}{2}\right)
 =
 G'(x)-x
 =
 -S_{\bar\gamma}(x),
\]
whereas for $x<0$,
\[
 \frac{\dd}{\dd x}
 \left(\widetilde G(x)-\frac{x^2}{2}\right)
 =
 -G'(-x)-x
 =
 S_{\bar\gamma}(-x).
\]
Since $S_{\bar\gamma}$ is nondecreasing and
$S_{\bar\gamma}(0)=0$, these two derivatives form a nonincreasing
function of $x$ and agree at zero. This proves the claimed
semiconcavity.

Consequently, for every $q\in[0,1]$ and every $z$ for which
$q+z\in[-1,1]$,
\[
 \widetilde G(q+z)
 \le
 \widetilde G(q)+G'(q)z+\frac{z^2}{2}.
\]
Taking $z=U$ and using
\[
 \widetilde G(q+U)=G(|q+U|),
 \qquad
 \E_U U=0,
 \qquad
 \E_U U^2=\frac{h^2}{3},
\]
gives
\begin{equation}\label{eq:p2-reflected-semiconcavity}
 \E_U G(|q+U|)-G(q)
 \le
 \frac{h^2}{6}.
\end{equation}

We now compare the two functional values. Convexity of $F_r$ at the
endpoint $\bar\gamma$, in the admissible direction from
$\bar\gamma$ back to $\gamma$, gives
\[
 F_r(\gamma)
 \ge
 F_r(\bar\gamma)
 +
 D_+F_r(\bar\gamma)[\nu-\bar\nu].
\]
Therefore, by \eqref{eq:abs-first-variation},
\[
 \begin{aligned}
 F_r(\bar\gamma)-F_r(\gamma)
 &\le
 -D_+F_r(\bar\gamma)[\nu-\bar\nu]\\
 &=
 \frac r2
 \int_0^1G(q)\,(\bar\nu-\nu)(\dd q).
 \end{aligned}
\]
The mass on $[T,1]$ is unchanged, so the last integral is
\[
 \int_{[0,T)}
 \left[
   \E_U G(|q+U|)-G(q)
 \right]\nu(\dd q).
\]
Using \eqref{eq:p2-reflected-semiconcavity} and
$\nu([0,T))=m$ yields
\[
 F_r(\bar\gamma)-F_r(\gamma)
 \le
 \frac r2\,m\,\frac{h^2}{6}
 =
 \frac{rmh^2}{12},
\]
which is \eqref{eq:p2-smoothing-cost}.

If $\gamma$ minimizes $F_r$, then $\bar\gamma$ is admissible and hence
\[
 F_r(\bar\gamma)-F_r(\gamma)\ge0.
\]
\end{proof}

Choose
\[
 0<T_0\le\min\{T_*,1/4\},
\]
where $T_*$ is the uniform near-zero interval from
\zcref{lem:gs-scalar-root}. Decrease $\lambda_*$, if necessary, so
that both scalar floor estimates hold on $(0,\lambda_*]$ and
\[
 \lambda_*\le\min\{1,\beta_0/2\}.
\]
Then fix
\[
 0<\lambda_0\le\lambda_*.
\]
For the remainder of this subsection, let
\[
 \beta\ge\beta_0,\qquad 0<\lambda\le\lambda_0,
\]
and use as reference profile the scalar trial
\[
 \gamma^{\mathrm{tr}}_{\beta,\lambda}(q)
 =
 \max\{\lambda,\gamma_\beta(q)\}.
\]
By \zcref{lem:gs-scalar-root}, after choosing $T_0$ and
$\lambda_0$ sufficiently small in terms of $\beta_0$, this profile
satisfies
\[
 kq\le\gamma^{\mathrm{tr}}_{\beta,\lambda}(q)\le2,
 \qquad 0\le q\le T_0,
\]
and is absolutely continuous there, with
\[
 0\le
 (\gamma^{\mathrm{tr}}_{\beta,\lambda})'(q)
 \le C_s
\]
for almost every $q\in[0,T_0]$. All constants below depend only on
$T_0$, $k$, $C_s$, and the uniform curvature constant
\[
 c_2:=c_U\big|_{U=2}>0
\]
from \eqref{eq:shared-averaged-curvature}. In particular, they are
independent of $N$, $\beta$, $\lambda$, and the interpolation
parameter $r$.

We next record the response estimates in the precise form used below.
Let
\[
 \gamma_\theta=\gamma_0+\theta v,
 \qquad 0\le\theta\le1,
\]
be any admissible profile segment under consideration. Every
expectation in the following identities is taken under the normalized
joint law associated with the current profile $\gamma_\theta$. Define
\[
 S_\theta(q)
 =
 \frac1N\E_\theta|m_q|^2,
 \qquad
 b_\theta(q)
 =
 \frac1N\E_\theta\Tr H_q,
 \qquad
 \zeta_\theta(q)
 =
 \frac1N\E_\theta\Tr H_q^2.
\]
Then
\[
 S_\theta'(q)=r\zeta_\theta(q),
 \qquad
 b_\theta'(q)=-\gamma_\theta(q)S_\theta'(q),
\]
and
\[
 S_\theta(0)=0,\qquad
 0\le S_\theta\le1,\qquad
 b_\theta\ge0,\qquad
 \zeta_\theta\ge b_\theta^2.
\]

The response estimates
\zcref{prop:gs-response-tools,lem:response-fisher,lem:response-curvature}
also provide nonnegative functions
$e_\theta$ and $\mathcal J_\theta$ such that
\begin{equation}\label{eq:abs-response-summary-second}
 \frac{\dd^2}{\dd\theta^2}F_r(\gamma_\theta)
 \ge
 r\int_0^1 e_\theta(q)\,\dd q,
\end{equation}
and, for every compactly supported Lipschitz function $\varphi$,
\begin{equation}\label{eq:abs-response-summary-pairing}
 r\left|
 \int_0^1v(q)\zeta_\theta(q)\varphi(q)\,\dd q
 \right|
 \le
 \left(\int_0^1e_\theta(q)\,\dd q\right)^{1/2}
 \left[
 \left(\int_0^1b_\theta(q)|\varphi'(q)|^2\,\dd q\right)^{1/2}
 +
 2\left(\int_0^1\mathcal J_\theta(q)\varphi(q)^2\,\dd q\right)^{1/2}
 \right],
\end{equation}
while the Fisher cutoff estimate gives
\begin{equation}\label{eq:abs-response-summary-fisher}
 \int_0^1\mathcal J_\theta(q)\varphi(q)^2\,\dd q
 \le
 \int_0^1b_\theta(q)|\varphi'(q)|^2\,\dd q
 +
 r\int_{[0,1]}
 \varphi(q)^2\zeta_\theta(q)\,\dd\gamma_\theta(q).
\end{equation}
The final integral is a Stieltjes integral, so atoms of
$\dd\gamma_\theta$ inside the support of $\varphi$ are included
automatically. This is the form of the Fisher identity used below. 

Finally, the averaged-curvature estimate gives the uniform implication
\begin{equation}\label{eq:abs-response-summary-curvature}
 \gamma_\theta(q)\le2
 \quad\Longrightarrow\quad
 \zeta_\theta(q)\ge c_2^2.
\end{equation}
The estimates
\eqref{eq:abs-response-summary-second}--\eqref{eq:abs-response-summary-curvature}
hold for arbitrary admissible $\gamma_\theta$, so we may apply them
along the entire segment joining a minimizing finite-$N$ profile to a
regularized comparison profile.

We now combine the scalar lower slope with the reflected
regularization of \zcref{lem:p2-smoothing}. The result is a uniform
localization bound for the low quantiles of the free mass of every
finite-$N$ minimizing profile.

\begin{lemma}\zlabel{lem:p2-quantile}
There exists a constant $C<\infty$, depending only on the fixed
constants $T_0,k,C_s,c_2$ above, with the following property.

Let $\gamma$ be any minimizer of $F_r$ over the full admissible class
of nondecreasing profiles satisfying
\[
 \dd\gamma([0,1])=\beta,
 \qquad
 \dd\gamma\ge\lambda\delta_0.
\]
No absolute-continuity or density assumption is imposed on
$\dd\gamma$. Write
\[
 \nu=\dd\gamma-\lambda\delta_0
\]
for the remaining free mass, and let
\[
 a(q)=\max\{\lambda,\gamma_\beta(q)\}
\]
be the scalar trial profile fixed above. Suppose that its excess over
the finite-$N$ minimum satisfies
\[
 0\le F_r(a)-F_r(\gamma)\le\eta.
\]

For $0<v<\beta-\lambda$, define the generalized quantile of the free
measure $\nu$ by
\[
 Q_\nu(v)
 =
 \inf\{q\in[0,1]:\nu([0,q])\ge v\}.
\]
Then, for every $0<u\le1$ and every
$0<v<\beta-\lambda$ satisfying
\[
 \lambda+v<2u,
\]
one has
\begin{equation}\label{eq:p2-quantile-bound}
 Q_\nu(v)
 \le
 C\left[
   u
   +r^{-7/11}\eta^{2/11}
   +r^{-5/8}u^{3/8}\eta^{1/8}
 \right].
\end{equation}

Thus every free-mass quantile whose cumulative level
$\lambda+v$ is below $2u$ must lie near the origin, with the
localization deteriorating only through the displayed powers of the
trial excess $\eta$ and the interpolation scale $r$.
The estimate holds for every minimizer $\gamma$ in the full
admissible class; no selection of a particular minimizer is required.
\end{lemma}

\begin{proof}
We first assume $\eta>0$. The argument localizes the $v$-th free-mass
quantile by smoothing only the mass strictly below that quantile and
then comparing the smoothed profile with the scalar trial profile.

Set
\[
 q=\min\{Q_\nu(v),T_0\}.
\]
Since $T_0\le1$, we have
\begin{equation}\label{eq:p2-q-to-quantile}
 Q_\nu(v)\le\frac{q}{T_0}.
\end{equation}
Indeed, if $Q_\nu(v)\le T_0$, then $q=Q_\nu(v)$; if
$Q_\nu(v)>T_0$, then $q=T_0$ and $Q_\nu(v)\le1=q/T_0$.
Thus it is enough to bound $q$.

If $q=0$, there is nothing to prove. If
\[
 q\le\frac{16u}{k},
\]
then \eqref{eq:p2-q-to-quantile} gives
\[
 Q_\nu(v)\le\frac{16u}{kT_0},
\]
as required. We therefore assume 
\begin{equation}\label{eq:p2-large-q}
 q>\frac{16u}{k}.
\end{equation}

\medskip
\emph{Step 1: Regularize the free mass below $q$.}
Let
\[
 m=\nu([0,q)).
\]
By the definition of the generalized quantile,
\[
 m\le v<2u-\lambda.
\]
We choose the smoothing scale to satisfy two requirements: the
displacement should be small compared with the interval $[0,q]$, and
the smoothing cost should be at most a fixed fraction of the available
budget $\eta$. Set
\[
 h
 =
 \min\left\{
   \frac q8,\,
   \sqrt{\frac{\eta}{rm}}
 \right\},
\]
where the second quantity is interpreted as $+\infty$ when $m=0$.
Since $q\le T_0\le1/4$, this choice satisfies the hypotheses of
\zcref{lem:p2-smoothing}. Let $\bar\gamma$ be the reflected
regularization obtained there with $T=q$.

Only the free mass on $[0,q)$ is moved. Hence, on $[0,q)$,
\[
 \bar\gamma(s)\le\lambda+m<2u.
\]
The restriction of $\dd\bar\gamma$ to $(0,q)$ has density at most
$m/h$. Any atom of $\nu$ at $q$ belongs to the unchanged tail and is
therefore not regularized. Moreover,
\[
 rmh^2\le\eta,
\]
so \eqref{eq:p2-smoothing-cost} gives
\begin{equation}\label{eq:p2-regularized-budget}
 F_r(\bar\gamma)
 \le
 F_r(\gamma)+\frac{\eta}{12}.
\end{equation}

\medskip
\emph{Step 2: Compare the regularized profile with the scalar trial.}
Recall that
\[
 a(q)=\max\{\lambda,\gamma_\beta(q)\}.
\]
Join $\bar\gamma$ to $a$ by the admissible segment
\[
 \gamma_\theta=(1-\theta)\bar\gamma+\theta a,
 \qquad 0\le\theta\le1,
\]
whose cumulative direction is
\[
 a-\bar\gamma.
\]
On $[0,q)$, both endpoint profiles are bounded by $2$:
\[
 \bar\gamma\le2u\le2,
 \qquad
 a\le2.
\]
Hence every $\gamma_\theta$ is bounded by $2$ on $[0,q)$. This is
sufficient below, since the profile bound is used only under time
integration or on the support of a cutoff contained strictly below $q$. 

For the current profile $\gamma_\theta$, suppress $\theta$ from the
notation temporarily and write
\[
 S=S_\theta,\qquad b=b_\theta,\qquad \zeta=\zeta_\theta.
\]
Since
\[
 b'(s)=-\gamma_\theta(s)S'(s),
 \qquad
 0\le\gamma_\theta(s)\le2
 \quad\text{for a.e. }s\in[0,q],
 \qquad
 S(0)=0,
\]
we have, for $0\le s\le q$,
\[
 b(s)\ge b(0)-2S(s)\ge b(0)-2.
\]
Using $S'=r\zeta$, $\zeta\ge b^2$, and $S(q)\le1$ therefore gives
\[
 1
 \ge
 S(q)
 =
 r\int_0^q\zeta(s)\,\dd s
 \ge
 rq\,(b(0)-2)_+^2.
\]
Thus
\[
 b(0)
 \le
 2+(rq)^{-1/2}.
\]
Since $b$ is nonincreasing,
\begin{equation}\label{eq:p2-b-bound}
 0\le b(s)\le b(0)
 \le
 3(rq)^{-1/2},
 \qquad 0\le s\le q,
\end{equation}
where we used $rq\le1$.

Choose the triangular cutoff
\[
 \varphi(s)
 =
 \left(
 1-\frac{4|s-q/2|}{q}
 \right)_+.
\]
It is supported on $[q/4,3q/4]$, has height one, and satisfies
\begin{equation}\label{eq:p2-triangle-data}
 \int_0^1\varphi(s)\,\dd s=\frac q4,
 \qquad
 \int_0^1|\varphi'(s)|^2\,\dd s=\frac8q.
\end{equation}
On its support, the scalar lower slope gives
\[
 a(s)\ge ks\ge\frac{kq}{4},
\]
whereas $\bar\gamma(s)<2u$. By
\eqref{eq:p2-large-q},
\[
 a(s)-\bar\gamma(s)
 \ge
 \frac{kq}{4}-2u
 \ge
 \frac{kq}{8}.
\]
The same profile bound $\gamma_\theta\le2$ gives, by
\eqref{eq:abs-response-summary-curvature},
\[
 \zeta_\theta(s)\ge c_2^2
\]
on the support of $\varphi$. Consequently
\begin{equation}\label{eq:p2-pairing-lower}
 \int_0^1
 (a-\bar\gamma)\zeta_\theta\varphi\,\dd s
 \ge cq^2.
\end{equation}

We next bound the two localization terms in the response pairing.
From \eqref{eq:p2-b-bound} and
\eqref{eq:p2-triangle-data},
\begin{equation}\label{eq:p2-cutoff-gradient}
 \int_0^1b_\theta|\varphi'|^2\,\dd s
 \le
 C r^{-1/2}q^{-3/2}.
\end{equation}
On the support of $\varphi$, the measure $\dd a$ is absolutely
continuous with density at most $C_s$, while
$\dd\bar\gamma$ has density at most $m/h$. The cutoff is separated
from both $0$ and $q$, so neither the prescribed root atom nor the
unchanged tail mass at $q$ contributes. Hence
\[
 \begin{aligned}
 r\int_{[0,1]}\varphi^2\zeta_\theta\,\dd\gamma_\theta
 &\le
 (C_s+m/h)\,
 r\int_0^q\zeta_\theta(s)\,\dd s\\
 &=
 (C_s+m/h)S_\theta(q)\\
 &\le
 C_s+\frac mh.
 \end{aligned}
\]
The localized Fisher estimate
\eqref{eq:abs-response-summary-fisher} therefore gives
\begin{equation}\label{eq:p2-fisher-bound}
 \int_0^1\mathcal J_\theta\varphi^2\,\dd s
 \le
 C r^{-1/2}q^{-3/2}
 +C_s+\frac mh.
\end{equation}

Combining
\eqref{eq:p2-pairing-lower},
\eqref{eq:p2-cutoff-gradient},
\eqref{eq:p2-fisher-bound}, and the response pairing
\eqref{eq:abs-response-summary-pairing}, we obtain
\[
 crq^2
 \le
 C\left(\int_0^1e_\theta\,\dd s\right)^{1/2}
 \left(
 r^{-1/2}q^{-3/2}+C_s+\frac mh
 \right)^{1/2}.
\]
After squaring and using
\eqref{eq:abs-response-summary-second},
\[
 \frac{\dd^2}{\dd\theta^2}F_r(\gamma_\theta)
 \ge
 \frac{c r^3q^4}
 {r^{-1/2}q^{-3/2}+C_s+m/h}.
\]
Since $r,q\le1$, the fixed $C_s$ term may be absorbed into the
constant after multiplying numerator and denominator by
$r^{1/2}q^{3/2}$. Thus
\begin{equation}\label{eq:p2-second-variation-lower}
 \frac{\dd^2}{\dd\theta^2}F_r(\gamma_\theta)
 \ge
 \frac{c r^{7/2}q^{11/2}}
 {1+\sqrt r\,q^{3/2}m/h},
 \qquad 0<\theta<1.
\end{equation}

\medskip
\emph{Step 3: Use the variational budget to bound $q$.}
Apply the midpoint identity
\eqref{eq:midpoint-convexity-gap} to
\[
 f(\theta)=F_r(\gamma_\theta).
\]
Since $\gamma$ minimizes $F_r$ over the full admissible class,
\[
 F_r\left(\frac{a+\bar\gamma}{2}\right)\ge F_r(\gamma).
\]
Together with the hypothesis
$F_r(a)-F_r(\gamma)\le\eta$ and
\eqref{eq:p2-regularized-budget}, this gives
\[
 \begin{aligned}
 &F_r(a)+F_r(\bar\gamma)
 -2F_r\left(\frac{a+\bar\gamma}{2}\right)\\
 &\qquad\le
 F_r(a)+F_r(\bar\gamma)-2F_r(\gamma)
 \le
 \frac{13}{12}\eta.
 \end{aligned}
\]
On the other hand, integrating
\eqref{eq:p2-second-variation-lower} against the midpoint weight,
whose integral is $1/4$, yields
\[
 \eta
 \ge
 \frac{c r^{7/2}q^{11/2}}
 {1+\sqrt r\,q^{3/2}m/h}.
\]
Equivalently,
\begin{equation}\label{eq:p2-q-prebalance}
 q^{11/2}
 \le
 C r^{-7/2}\eta
 \left(
 1+\sqrt r\,q^{3/2}\frac mh
 \right).
\end{equation}

By the definition of $h$,
\[
 \frac mh
 \le
 \frac{8m}{q}
 +\sqrt r\,m^{3/2}\eta^{-1/2}.
\]
Since $m<2u$ and $q>16u/k$, the first term is bounded by a constant
depending only on $k$. Hence
\begin{equation}\label{eq:p2-density-cost}
 \frac mh
 \le
 C+C\sqrt r\,u^{3/2}\eta^{-1/2}.
\end{equation}
Substituting this into \eqref{eq:p2-q-prebalance}, and using again
$r,q\le1$ to absorb the fixed contribution, gives
\begin{equation}\label{eq:p2-q-balance}
 q^{11/2}
 \le
 C r^{-7/2}\eta
 +
 C r^{-5/2}q^{3/2}u^{3/2}\eta^{1/2}.
\end{equation}

At least one term on the right of \eqref{eq:p2-q-balance} is at least
one half of the left-hand side. If the first term is dominant, then
\[
 q
 \le
 C r^{-7/11}\eta^{2/11}.
\]
If the second is dominant, then
\[
 q^4
 \le
 C r^{-5/2}u^{3/2}\eta^{1/2},
\]
and hence
\[
 q
 \le
 C r^{-5/8}u^{3/8}\eta^{1/8}.
\]
Together with the previously treated case
$q\le16u/k$, we have
\[
 q
 \le
 C\left[
 u
 +r^{-7/11}\eta^{2/11}
 +r^{-5/8}u^{3/8}\eta^{1/8}
 \right].
\]
Finally, \eqref{eq:p2-q-to-quantile} converts this estimate into
\[
 Q_\nu(v)
 \le
 C\left[
 u
 +r^{-7/11}\eta^{2/11}
 +r^{-5/8}u^{3/8}\eta^{1/8}
 \right],
\]
after enlarging $C$ by the fixed factor $T_0^{-1}$.

If $\eta=0$, the hypothesis also holds with every positive budget
$\varepsilon>0$ in place of $\eta$. Applying the estimate just proved
with $\varepsilon$ and then sending $\varepsilon\downarrow0$ gives
the same conclusion.
\end{proof}

Integrating the quantile localization from \zcref{lem:p2-quantile}
over the interpolation parameter closes the finite-$N$ comparison in
the rescaled variables. The resulting estimate is uniform in
$\beta\ge\beta_0$ and retains the dependence on the positive tilt
$\lambda$. We write
\[
 e_{N,\beta}(\lambda)
 =
 E_\beta(\lambda)-\psi^c_{N,\beta}(\lambda)
\]
for the difference between the constrained scalar variational value
and the completed finite-$N$ normalized logarithmic moment. The
fixed-temperature comparison used later is obtained from this estimate
by returning to the original variables.

\begin{proposition}\zlabel{prop:abs-selfconsistent-comparison}
There exist constants
\[
 C<\infty,
 \qquad
 0<\lambda_0\le\min\{1,\beta_0/2\},
\]
depending only on $\beta_0$, such that for every
\[
 N\ge1,\qquad
 \beta\ge\beta_0,\qquad
 0<\lambda\le\lambda_0,
\]
one has
\begin{equation}\label{eq:abs-comparison}
 0\le e_{N,\beta}(\lambda)
 \le
 C\left[
   N^{-2/3}
   +N^{-3/4}\lambda^{-1/2}
   +\frac{\log(2+N)}{N\lambda}
 \right].
\end{equation}
\end{proposition}

\begin{proof}
Write
\[
 e=e_{N,\beta}(\lambda)
 =
 E_\beta(\lambda)-\psi^c_{N,\beta}(\lambda).
\]
We first obtain a uniform bound on the excess of the scalar trial
profile at every interpolation parameter. We then apply
\zcref{lem:p2-quantile} with this excess as its budget, insert the
resulting quantile localization into the quantile-interval estimate, and finally
close the resulting inequality for $e$.

Let
\[
 a(q)=\max\{\lambda,\gamma_\beta(q)\}
\]
be the scalar trial profile fixed above. After rescaling,
the fixed-profile and optimized monotonicity from
\zcref{prop:fn-envelope} give, for every $0<r\le1$,
\[
 F_r(a)\le A_\beta(\lambda),
 \qquad
 \min_\gamma F_r(\gamma)
 \ge\psi^c_{N,\beta}(\lambda).
\]
Choose a constant $L_s\ge1$, depending only on $\beta_0$, such that
\eqref{eq:gs-scalar-trial-cost} gives
\[
 A_\beta(\lambda)-E_\beta(\lambda)
 \le
 A_\beta(\lambda)-E_\beta(0)
 \le
 L_s\lambda^5.
\]
It follows that
\begin{equation}\label{eq:abs-uniform-budget}
 F_r(a)-\min_\gamma F_r(\gamma)
 \le
 e+L_s\lambda^5
 =:\eta
\end{equation}
for every $0<r\le1$. Thus the same deterministic quantity $\eta$
bounds the scalar trial's excess at every interpolation point and is
positive because $\lambda>0$.

\medskip
\emph{Step 1: Express the interpolation error through the overlap
defect.}
Fix $r>0$ and let $\gamma$ be any minimizer of $F_r$ over the full
admissible class
\[
 \dd\gamma=\lambda\delta_0+\nu,
 \qquad
 \nu\ge0.
\]
Define
\begin{equation}\label{eq:energy-defect}
 \mathcal D(r)
 =
 \int_{[0,1]}
 \E\left[
   \bigl(R(\sigma^1,\sigma^2)-q\bigr)^2
 \right]\nu(\dd q),
\end{equation}
where, conditional on the interpolation history through $q$,
$\sigma^1$ and $\sigma^2$ are independent terminal-spin samples.

Returning momentarily to the probability-scale measure
$\beta^{-1}\dd\gamma$, we have
\[
 \mathcal D(r)
 =
 \beta D_{1-r,\lambda/\beta}.
\]
By \zcref{prop:fn-envelope}, at almost every $r$ the value of
$\mathcal D(r)$ is determined by the derivative of the optimized
envelope. The rescaled envelope identity is 
\begin{equation}\label{eq:abs-energy-envelope}
 4e=\int_0^1\mathcal D(r)\,\dd r.
\end{equation}

\medskip
\emph{Step 2: Localize the free-mass quantiles.}
For each dyadic level
\[
 u=u_j=2^j\lambda<\beta,
 \qquad j\ge0,
\]
define the quantile interval
\[
 I_u
 =
 \left\{
 v\in(0,\beta-\lambda):
 u\le\lambda+v<\min\{2u,\beta\}
 \right\}.
\]
Every nonempty $I_u$ has length at most $u$, and for
$v\in I_u$,
\[
 \gamma(Q_\nu(v))
 \ge
 \lambda+v
 \ge u.
\]
Let
\[
 \ell_u
 =
 \operatorname*{ess\,sup}_{v,w\in I_u}
 |Q_\nu(v)-Q_\nu(w)|
\]
be the diameter of its image in the overlap variable.

For $u\le1$, the condition defining $I_u$ implies
$\lambda+v<2u$ almost everywhere. Applying
\zcref{lem:p2-quantile} with the uniform budget
\eqref{eq:abs-uniform-budget} gives
\begin{equation}\label{eq:abs-diameter-bound}
 \ell_u
 \le
 \min\left\{
 1,\,
 C\bigl(u+A_r+B_{r,u}\bigr)
 \right\},
\end{equation}
where
\[
 A_r=r^{-7/11}\eta^{2/11},
 \qquad
 B_{r,u}=r^{-5/8}u^{3/8}\eta^{1/8}.
\]
For $u>1$ we use only the trivial bound
\[
 \ell_u\le1.
\]

\medskip
\emph{Step 3: Rescale the quantile-interval estimate.} 
We now apply \zcref{lem:cr-refined-shell}. In the original
probability coordinates,
\[
 c=\beta^2r,
 \qquad
 M=N\beta^2r,
\]
and a cumulative level $u$ for $\gamma$ corresponds to $u/\beta$ for
$\gamma/\beta$. The free measure for $\gamma/\beta$ is $\nu/\beta$,
so the quantile-interval estimate must be multiplied by $\beta$.

The three terms in \eqref{eq:cr-shell-bound} then become
\[
 \begin{aligned}
 \beta(N\beta^2r)^{-3/4}(u/\beta)^{-1/2}
 &=(Nr)^{-3/4}u^{-1/2},\\
 \beta(N\beta^2r)^{-2/3}(u/\beta)^{-1/3}\ell_u^{2/3}
 &=(Nr)^{-2/3}u^{-1/3}\ell_u^{2/3},\\
 \frac{\beta}{(N\beta^2r)(u/\beta)}
 &=(Nru)^{-1}.
 \end{aligned}
\]
In particular, every power of $\beta$ cancels. This cancellation is
what makes the final estimate uniform for $\beta\ge\beta_0$.

The trace contribution behaves in the same way. The probability-scale
bound \eqref{eq:fn-trace-error}, with
$s=\lambda/\beta$, becomes after multiplication by $\beta$
\begin{equation}\label{eq:abs-energy-trace}
 \frac1{Nr\lambda}.
\end{equation}

For $u_j\le1$, use \eqref{eq:abs-diameter-bound} and subadditivity of
$x\mapsto x^{2/3}$ to obtain
\begin{align}
 u_j^{-1/3}\ell_{u_j}^{2/3}
 \le C\biggl[
 &u_j^{1/3}
 +u_j^{-1/3}
   \min\{1,r^{-14/33}\eta^{4/33}\}\nonumber\\
 &+
 u_j^{-1/3}
   \min\{1,r^{-5/12}u_j^{1/4}\eta^{1/12}\}
 \biggr].
 \label{eq:abs-diameter-expanded}
\end{align}
The sum of $u_j^{1/3}$ over the levels $u_j\le1$ is bounded by a
numerical constant. For $u_j>1$, the trivial diameter bound gives
$u_j^{-1/3}\ell_{u_j}^{2/3}\le u_j^{-1/3}$, whose dyadic sum is also
bounded independently of the terminal level $\beta$.

The first and third terms in the quantile-interval estimate similarly
give the geometric sums 
\[
 C(Nr)^{-3/4}\lambda^{-1/2}
 \qquad\text{and}\qquad
 C(Nr\lambda)^{-1},
\]
respectively. It remains to integrate in $r$ the two localization
terms in \eqref{eq:abs-diameter-expanded}.

\medskip
\emph{Step 4: Integrate the localization terms.}
For every $\eta\ge0$,
\begin{equation}\label{eq:abs-first-r-integral}
 \int_0^1
 r^{-2/3}
 \min\{1,r^{-14/33}\eta^{4/33}\}\,\dd r
 \le
 14\eta^{2/21}.
\end{equation}
Indeed, when $0<\eta\le1$, split at
$r=\eta^{2/7}$; the integral is exactly
\[
 14\eta^{2/21}-11\eta^{4/33}.
\]
For $\eta\ge1$, the left side is at most $3$, which is bounded by the
right side after enlarging the numerical constant. The case
$\eta=0$ is immediate.

Likewise, for every $u,\eta\ge0$,
\begin{equation}\label{eq:abs-second-r-integral}
 \int_0^1
 r^{-2/3}
 \min\{1,r^{-5/12}u^{1/4}\eta^{1/12}\}\,\dd r
 \le
 15u^{1/5}\eta^{1/15}.
\end{equation}
To see this, put
\[
 w=u^{1/4}\eta^{1/12}.
\]
For $0<w\le1$, split at $r=w^{12/5}$; the exact value is
\[
 15w^{4/5}-12w.
\]
For $w\ge1$, use again the trivial integral bound $3$.

After multiplying by the remaining factor $u_j^{-1/3}$ and summing
over the dyadic levels $u_j\le1$,
\eqref{eq:abs-first-r-integral} contributes
\[
 C\lambda^{-1/3}\eta^{2/21},
\]
while \eqref{eq:abs-second-r-integral} contributes
\[
 C\lambda^{-2/15}\eta^{1/15}.
\]

\medskip
\emph{Step 5: Control the singular endpoint $r=0$.}
The terms $(Nr\lambda)^{-1}$ are not integrable at zero, so we use
the moving-cutoff estimate there. Rescaling
\zcref{lem:moving-cutoff} gives
\[
 \mathcal D(r)\le C(Nr)^{-1/2}.
\]
Hence
\[
 \int_0^{N^{-1}}\mathcal D(r)\,\dd r
 \le
 CN^{-1/2}\int_0^{N^{-1}}r^{-1/2}\,\dd r
 \le
 \frac{C}{N}.
\]
On $[N^{-1},1]$, the trace contribution
\eqref{eq:abs-energy-trace} and the third shell term each give at most
\[
 C\int_{N^{-1}}^1\frac{\dd r}{N\lambda r}
 \le
 C\frac{\log(2+N)}{N\lambda}.
\]
All remaining powers of $r$ are integrable at zero and were estimated
above.

Using $N^{-1}\le N^{-2/3}$ and the envelope identity
\eqref{eq:abs-energy-envelope}, we arrive at
\begin{equation}\label{eq:abs-budget-comparison}
 e
 \le
 C\left[
 N^{-2/3}
 +N^{-3/4}\lambda^{-1/2}
 +\frac{\log(2+N)}{N\lambda}
 +N^{-2/3}\lambda^{-1/3}\eta^{2/21}
 +N^{-2/3}\lambda^{-2/15}\eta^{1/15}
 \right].
\end{equation}

\medskip
\emph{Step 6: Close the self-consistent inequality.}
Recall that
\[
 \eta=e+L_s\lambda^5.
\]
Since the exponents $2/21$ and $1/15$ are below one,
subadditivity gives
\[
 \eta^{2/21}
 \le
 e^{2/21}+C\lambda^{10/21},
 \qquad
 \eta^{1/15}
 \le
 e^{1/15}+C\lambda^{1/3}.
\]
The contributions of the scalar term $L_s\lambda^5$ to
\eqref{eq:abs-budget-comparison} are therefore bounded by
\[
 CN^{-2/3}
 \left(
   \lambda^{1/7}+\lambda^{1/5}
 \right)
 \le
 CN^{-2/3},
\]
because $0<\lambda\le\lambda_0\le1$.

For the two terms containing $e$, Young's inequality gives
\[
 \begin{aligned}
 C N^{-2/3}\lambda^{-1/3}e^{2/21}
 &\le
 \frac e4
 +C\bigl(N^{-2/3}\lambda^{-1/3}\bigr)^{21/19}\\
 &=
 \frac e4
 +C N^{-14/19}\lambda^{-7/19},
 \end{aligned}
\]
and
\[
 \begin{aligned}
 C N^{-2/3}\lambda^{-2/15}e^{1/15}
 &\le
 \frac e4
 +C\bigl(N^{-2/3}\lambda^{-2/15}\bigr)^{15/14}\\
 &=
 \frac e4
 +C N^{-5/7}\lambda^{-1/7}.
 \end{aligned}
\]
Move the two $e/4$ terms to the left.

Finally set
\[
 A=N^{-2/3},
 \qquad
 B=N^{-3/4}\lambda^{-1/2}.
\]
The two remainders from Young's inequality are controlled by the
principal error terms. Indeed,
\[
 N^{-14/19}\lambda^{-7/19}
 =
 N^{-1/114}A^{5/19}B^{14/19}
 \le
 A+B,
\]
and
\[
 N^{-5/7}\lambda^{-1/7}
 =
 N^{-1/42}A^{5/7}B^{2/7}
 \le
 A+B,
\]
where we used $N\ge1$ and the weighted arithmetic--geometric mean
inequality. Substituting these estimates into
\eqref{eq:abs-budget-comparison} gives
\[
 e
 \le
 C\left[
 N^{-2/3}
 +N^{-3/4}\lambda^{-1/2}
 +\frac{\log(2+N)}{N\lambda}
 \right].
\]
This is \eqref{eq:abs-comparison}.
\end{proof}

\subsection{Uniform fluctuation exponents}

We now combine the scalar floor estimates with the refined finite-$N$
comparison and the Gaussian moment bounds. For
\[
 \Delta_{N,\beta}
 =
 E_\beta(0)-\frac{\E X_{N,\beta}}{N},
\]
the quantity $\Delta_{N,\beta}$ is the per-spin finite-size deficit of
the mean rescaled free energy from the scalar Parisi value.

The four estimates below are obtained in the following order. First,
the refined comparison at two positive tilts gives
\[
 \Delta_{N,\beta}\lesssim N^{-2/3}.
\]
The variance--deficit relation in
\zcref{lem:short-variance-mean} then gives the variance upper bound.
For the opposite direction, combine the scalar lower bound
\[
 E_\beta(\lambda)-E_\beta(0)\gtrsim\lambda^5
\]
with the positive-moment estimate \eqref{eq:gs-energy-mgf} to obtain
the variance lower bound. Substituting this bound back into
\zcref{lem:short-variance-mean} gives the lower bound on
$\Delta_{N,\beta}$. 

\begin{theorem}\zlabel{thm:uniform-energy-powers}
Fix $\beta_0>1$. There exist constants
\[
 c>0,\qquad C<\infty,\qquad N_0<\infty,
\]
depending only on $\beta_0$, such that for every
$\beta\ge\beta_0$ and every integer $N\ge N_0$,
\begin{equation}\label{eq:uniform-energy-four-bounds}
 \begin{gathered}
 cN^{4/15}
 \le
 \Var(X_{N,\beta})
 \le
 CN^{7/15},\\[2mm]
 cN^{-11/12}
 \le
 E_\beta(0)-\frac{\E X_{N,\beta}}{N}
 \le
 CN^{-2/3}.
 \end{gathered}
\end{equation}
After increasing $C$ if necessary, the two upper bounds hold for
every $N\ge1$ and every $\beta\ge\beta_0$.
\end{theorem}

\begin{proof}
Write
\[
 X=X_{N,\beta}=\frac{F_{N,\beta}}{\beta},
 \qquad
 M_{N,\beta}
 =
 E_\beta(0)-\frac{\E X}{N}.
\]
The comparison at zero gives
\[
 M_{N,\beta}\ge0.
\]
We prove the four bounds in the order described above.

\medskip
\emph{Step 1: Upper bound on the mean deficit.}
Choose the tilt so that the fifth-order scalar cost has the same order
as the leading finite-$N$ comparison error:
\[
 \lambda=N^{-2/15}.
\]
Take $N$ large enough that $2\lambda\le\lambda_0$. Write
\[
 \psi(\lambda)=\psi^c_{N,\beta}(\lambda),
 \qquad
 \psi(0)=\frac{\E X}{N}.
\]
By \zcref{lem:short-laplace-convexity}, $\psi$ is convex on the
positive half-line. Applying the midpoint inequality to the points
$0,\lambda,2\lambda$ gives
\[
 \psi(0)\ge2\psi(\lambda)-\psi(2\lambda).
\]
Since
\[
 \psi(\lambda)
 =
 E_\beta(\lambda)-e_{N,\beta}(\lambda)
\]
and $\psi(2\lambda)\le E_\beta(2\lambda)$, we obtain
\[
 \begin{aligned}
 M_{N,\beta}
 &\le
 E_\beta(0)-2\psi(\lambda)+\psi(2\lambda)\\
 &\le
 E_\beta(0)-2E_\beta(\lambda)+E_\beta(2\lambda)
 +2e_{N,\beta}(\lambda).
 \end{aligned}
\]
The constrained variational values satisfy
$E_\beta(\lambda)\ge E_\beta(0)$, while
\eqref{eq:gs-scalar-trial-cost} at $2\lambda$ gives
\[
 E_\beta(2\lambda)-E_\beta(0)\le C\lambda^5.
\]
Therefore
\begin{equation}\label{eq:uniform-power-mean-pre}
 M_{N,\beta}
 \le
 C\lambda^5+2e_{N,\beta}(\lambda).
\end{equation}

At $\lambda=N^{-2/15}$, the three terms in
\eqref{eq:abs-comparison} have sizes
\[
 N^{-2/3},
 \qquad
 N^{-3/4}\lambda^{-1/2}=N^{-41/60},
 \qquad
 \frac{\log(2+N)}{N\lambda}
 =
 N^{-13/15}\log(2+N).
\]
The last two are $o(N^{-2/3})$. Since
$\lambda^5=N^{-2/3}$, \eqref{eq:uniform-power-mean-pre} yields
\begin{equation}\label{eq:abs-endpoint-mean}
 0\le
 E_\beta(0)-\frac{\E X_{N,\beta}}{N}
 \le
 CN^{-2/3}
\end{equation}
for all sufficiently large $N$, uniformly in $\beta\ge\beta_0$.

For the remaining finite set of system sizes, the
energy-rescaled form of \eqref{eq:quantitative-pressure} gives a
uniform finite bound. Increasing $C$ therefore extends
\eqref{eq:abs-endpoint-mean} to every $N\ge1$.

\medskip
\emph{Step 2: Upper bound on the variance.}
By \zcref{lem:short-variance-mean},
\[
 M_{N,\beta}
 \ge
 c\left(
 \frac{\Var(X_{N,\beta})+1/2}{N}
 \right)^{5/4}.
\]
Solving for the variance and using
\eqref{eq:abs-endpoint-mean} gives
\begin{equation}\label{eq:abs-endpoint-variance}
 \begin{aligned}
 \Var(X_{N,\beta})+\frac12
 &\le
 C N M_{N,\beta}^{4/5}\\
 &\le
 C N^{1-(2/3)(4/5)}
 =
 C N^{7/15}.
 \end{aligned}
\end{equation}
This proves the variance upper bound, again uniformly for
$\beta\ge\beta_0$.

\medskip
\emph{Step 3: Lower bound on the variance.}
We now choose a larger fixed multiple of the same tilt scale:
\[
 \lambda=A N^{-2/15},
\]
where $A\ge1$ will be chosen independently of $N$ and $\beta$.
Let $c_{\mathrm{fl}}>0$ denote the lower floor-cost constant from
\zcref{prop:scalar-lower-via-fisher}. Dividing
\eqref{eq:abs-comparison} by $\lambda^5$ gives
\[
 \frac{e_{N,\beta}(\lambda)}{\lambda^5}
 \le
 C\left[
 A^{-5}
 +A^{-11/2}N^{-1/60}
 +A^{-6}N^{-1/5}\log(2+N)
 \right].
\]
Choose $A$ sufficiently large and then $N$ sufficiently large,
depending only on $\beta_0$, so that
\[
 \lambda\le\lambda_0
 \qquad\text{and}\qquad
 e_{N,\beta}(\lambda)
 \le
 \frac{c_{\mathrm{fl}}}{2}\lambda^5.
\]

Let
\[
 K_{N,\beta}(\lambda)
 =
 \log\E
 e^{\lambda(X_{N,\beta}-\E X_{N,\beta})}.
\]
Because
\[
 \psi^c_{N,\beta}(\lambda)
 =
 E_\beta(\lambda)-e_{N,\beta}(\lambda)
\]
and the independent diagonal completion contributes
$\lambda^2/4$ to its logarithmic moment, we have the exact identity
\[
 K_{N,\beta}(\lambda)
 =
 N\lambda
 \left[
 E_\beta(\lambda)-E_\beta(0)
 +M_{N,\beta}
 -e_{N,\beta}(\lambda)
 \right]
 -\frac{\lambda^2}{4}.
\]
The scalar lower bound gives
\[
 E_\beta(\lambda)-E_\beta(0)
 \ge
 c_{\mathrm{fl}}\lambda^5,
\]
and $M_{N,\beta}\ge0$. Hence
\begin{equation}\label{eq:uniform-power-K-lower}
 K_{N,\beta}(\lambda)
 \ge
 \frac{c_{\mathrm{fl}}}{2}N\lambda^6
 -\frac{\lambda^2}{4}.
\end{equation}
Since
\[
 N\lambda^6=A^6N^{1/5},
\]
the right-hand side exceeds $1$ once $N$ is large enough.

We compare this lower bound with the variance-sensitive estimate
\eqref{eq:gs-energy-mgf}. Put
\[
 V_{N,\beta}
 =
 \Var(X_{N,\beta})+\beta^{-2}.
\]
If
\[
 \lambda\sqrt{V_{N,\beta}}>b,
\]
then directly
\[
 V_{N,\beta}
 >
 b^2\lambda^{-2}
 =
 b^2A^{-2}N^{4/15}.
\]
If instead
\[
 \lambda\sqrt{V_{N,\beta}}\le b,
\]
then \eqref{eq:gs-energy-mgf} applies and gives
\[
 K_{N,\beta}(\lambda)
 \le
 B\lambda^2V_{N,\beta}.
\]
Because \eqref{eq:uniform-power-K-lower} is greater than $1$ and
$Bb^2=1$, this forces
\[
 V_{N,\beta}
 >
 b^2\lambda^{-2}
 =
 b^2A^{-2}N^{4/15}.
\]
Thus in either case
\[
 \Var(X_{N,\beta})+\beta^{-2}
 \ge
 cN^{4/15}.
\]
Since $\beta\ge\beta_0$, the term $\beta^{-2}$ is bounded uniformly.
After increasing the common threshold $N_0$, it can be absorbed into
the right-hand side, yielding
\[
 \Var(X_{N,\beta})\ge cN^{4/15}.
\]

\medskip
\emph{Step 4: Lower bound on the mean deficit.}
Substitute the variance lower bound into
\zcref{lem:short-variance-mean}:
\[
 \begin{aligned}
 M_{N,\beta}
 &\ge
 c\left(
 \frac{\Var(X_{N,\beta})+1/2}{N}
 \right)^{5/4}\\
 &\ge
 cN^{(4/15-1)(5/4)}
 =
 cN^{-11/12}.
 \end{aligned}
\]
Together with Steps~1--3, this proves
\eqref{eq:uniform-energy-four-bounds}. All constants and thresholds
depend only on $\beta_0$.
\end{proof}

\begin{proof}[Proof of \zcref{thm:ground-fluctuations}]
Apply \zcref{thm:uniform-energy-powers} with $\beta_0=2$. Its
constants are then independent of $\beta\ge2$.

Fix $N$. The deterministic soft-maximum bound for the off-diagonal
model gives
\[
 0
 \le
 X_{N,\beta}-G_N
 \le
 \frac{N\log2}{\beta}.
\]
Hence
\[
 X_{N,\beta}\longrightarrow G_N
\]
in $L^2$ as $\beta\to\infty$. Consequently
\[
 \E X_{N,\beta}\longrightarrow\E G_N,
 \qquad
 \Var(X_{N,\beta})\longrightarrow\Var(G_N).
\]
On the scalar side, \zcref{lem:zero-limit} with $\lambda=0$ gives
\[
 E_\beta(0)\longrightarrow e_*.
\]

For every $N$ above the threshold in
\zcref{thm:uniform-energy-powers}, we may therefore let
$\beta\to\infty$ in its four uniform estimates to obtain
\begin{equation}\label{eq:abs-new-powers}
 \begin{gathered}
 cN^{4/15}
 \le
 \Var(G_N)
 \le
 CN^{7/15},\\[2mm]
 cN^{-11/12}
 \le
 e_*-\frac{\E G_N}{N}
 \le
 CN^{-2/3}.
 \end{gathered}
\end{equation}
The lower bounds thus hold for all $N$ above a numerical threshold.
The two finite-temperature upper bounds hold for every $N\ge1$ after
enlarging $C$, and the same limiting argument therefore gives the
ground-state upper bounds for every $N\ge1$.
\end{proof}

\subsection{Uniform tails, tilted variance, and negative moments}

The refined comparison holds uniformly on an interval of positive
tilts, rather than only at a single value of the tilt. Combining it
with the matching scalar upper and lower bounds  gives a
sixth-order bound for the logarithmic moment centered at the limiting
variational value. This is the input needed for the two-sided
upper-tail estimate and for the variance bounds under exponential
tilting. Throughout this subsection, all constants and integer
thresholds depend only on $\beta_0$.

\begin{lemma}\zlabel{lem:uniform-energy-sixth-moment}
There exist constants
\[
 a_0,b_0,A_0,\lambda_*>0
 \qquad\text{and}\qquad
 N_0<\infty,
\]
depending only on $\beta_0$, with $\lambda_*\le1$, such that the
following holds uniformly for every $\beta\ge\beta_0$.

Define
\[
 \ell_{N,\beta}^{\mathrm{en}}(\lambda)
 =
 \log\E
 \exp\left\{
   \lambda\bigl(X_{N,\beta}-NE_\beta(0)\bigr)
 \right\}.
\]
For every $N\ge N_0$ and
\[
 A_0N^{-2/15}\le\lambda\le\lambda_*,
\]
one has
\begin{equation}\label{eq:uniform-energy-sixth-moment}
 a_0N\lambda^6
 \le
 \ell_{N,\beta}^{\mathrm{en}}(\lambda)
 \le
 b_0N\lambda^6.
\end{equation}
The upper bound
\[
 \ell_{N,\beta}^{\mathrm{en}}(\lambda)
 \le b_0N\lambda^6
\]
holds, after increasing $b_0$ if necessary, for every $N\ge1$ and
every $0<\lambda\le\lambda_*$.

At zero temperature, define
\[
 \ell_N^{\mathrm{gs}}(\lambda)
 =
 \log\E
 \exp\left\{
   \lambda(G_N-Ne_*)
 \right\}.
\]
With the same constants, for every $N\ge N_0$ and
\[
 A_0N^{-2/15}\le\lambda\le\lambda_*,
\]
\[
 a_0N\lambda^6
 \le
 \ell_N^{\mathrm{gs}}(\lambda)
 \le
 b_0N\lambda^6,
\]
and the upper bound again holds for every $N\ge1$ and
$0<\lambda\le\lambda_*$.
\end{lemma}

\begin{proof}
We first prove the finite-temperature bounds uniformly in
$\beta\ge\beta_0$. By the definition of
$e_{N,\beta}(\lambda)$,
\[
 \psi^c_{N,\beta}(\lambda)
 =
 E_\beta(\lambda)-e_{N,\beta}(\lambda).
\]
The completed moment contains the independent Gaussian
$D_N\sim N(0,1/2)$, whereas
$\ell_{N,\beta}^{\mathrm{en}}$ is defined for the off-diagonal
quantity $X_{N,\beta}$. Since
\[
 \log\E e^{\lambda D_N}=\frac{\lambda^2}{4},
\]
we obtain the exact identity
\begin{equation}\label{eq:energy-tail-moment-exact}
 \ell_{N,\beta}^{\mathrm{en}}(\lambda)
 =
 N\lambda
 \left[
   E_\beta(\lambda)-E_\beta(0)
   -e_{N,\beta}(\lambda)
 \right]
 -\frac{\lambda^2}{4}.
\end{equation}

\medskip
\emph{Upper bound.}
By \zcref{lem:gs-scalar-root},
\[
 E_\beta(\lambda)-E_\beta(0)\le C_{\mathrm{fl}}\lambda^5
\]
on a fixed interval $0<\lambda\le\lambda_*$. Since
$e_{N,\beta}(\lambda)\ge0$, \eqref{eq:energy-tail-moment-exact}
immediately gives
\[
 \ell_{N,\beta}^{\mathrm{en}}(\lambda)
 \le
 C_{\mathrm{fl}}N\lambda^6.
\]
After decreasing $\lambda_*$ to lie in the common scalar and
finite-$N$ comparison range, this proves the asserted upper bound
uniformly in $\beta\ge\beta_0$. The estimate is valid for every
$N\ge1$.

\medskip
\emph{Lower bound.}
Let $c_{\mathrm{fl}}>0$ be the constant in the lower bound from
\zcref{prop:scalar-lower-via-fisher}. Thus 
\[
 E_\beta(\lambda)-E_\beta(0)
 \ge
 c_{\mathrm{fl}}\lambda^5
\]
throughout the same fixed small-$\lambda$ interval.

We now require $\lambda$ to lie above a multiple of the scale
$N^{-2/15}$. Let
\[
 \lambda\ge A_0N^{-2/15}.
\]
Dividing \eqref{eq:abs-comparison} by $\lambda^5$ gives
\[
 \frac{e_{N,\beta}(\lambda)}{\lambda^5}
 \le
 C\left[
   A_0^{-5}
   +A_0^{-11/2}N^{-1/60}
   +A_0^{-6}N^{-1/5}\log(2+N)
 \right].
\]
Choose $A_0$ sufficiently large, and then choose $N_0$ sufficiently
large, both depending only on $\beta_0$, so that
\begin{equation}\label{eq:sixth-window-comparison-small}
 e_{N,\beta}(\lambda)
 \le
 \frac{c_{\mathrm{fl}}}{4}\lambda^5
\end{equation}
whenever
\[
 N\ge N_0,
 \qquad
 A_0N^{-2/15}\le\lambda\le\lambda_*.
\]

The remaining negative term in
\eqref{eq:energy-tail-moment-exact} is the diagonal-completion
correction. Relative to $N\lambda^6$,
\[
 \frac{\lambda^2/4}{N\lambda^6}
 =
 \frac1{4N\lambda^4}
 \le
 \frac14A_0^{-4}N^{-7/15}.
\]
Increasing $N_0$ if necessary, we may arrange that this is at most
$c_{\mathrm{fl}}/4$. Combining this estimate with
\eqref{eq:sixth-window-comparison-small} and the scalar lower floor
cost gives
\[
 \begin{aligned}
 \ell_{N,\beta}^{\mathrm{en}}(\lambda)
 &\ge
 N\lambda
 \left(
   c_{\mathrm{fl}}\lambda^5
   -\frac{c_{\mathrm{fl}}}{4}\lambda^5
 \right)
 -\frac{\lambda^2}{4}\\
 &\ge
 \frac{c_{\mathrm{fl}}}{2}N\lambda^6.
 \end{aligned}
\]
This proves the matching lower bound with an onset
$A_0N^{-2/15}$ that is uniform for all $\beta\ge\beta_0$.

\medskip
\emph{Zero-temperature limit.}
Fix $N$ and $\lambda>0$. The deterministic soft-maximum estimate gives
\[
 0\le X_{N,\beta}-G_N\le\frac{N\log2}{\beta},
\]
while \zcref{lem:zero-limit} gives
\[
 E_\beta(0)\longrightarrow e_*.
\]
Hence
\[
 \lambda\bigl(X_{N,\beta}-NE_\beta(0)\bigr)
 \longrightarrow
 \lambda(G_N-Ne_*)
\]
almost surely. Moreover,
\[
 |X_{N,\beta}|
 \le
 |G_N|+\frac{N\log2}{\beta}.
\]
Thus, for each fixed finite $\lambda$ and $\beta\ge1$,
\[
 e^{|\lambda||X_{N,\beta}|}
 \le
 2^{|\lambda|N}e^{|\lambda||G_N|}.
\]
The right-hand side is integrable because $G_N$ is the maximum of
finitely many Gaussian linear forms. Dominated convergence yields 
\[
 \ell_{N,\beta}^{\mathrm{en}}(\lambda)
 \longrightarrow
 \ell_N^{\mathrm{gs}}(\lambda).
\]

The finite-temperature constants and thresholds are independent of
$\beta$. Letting $\beta\to\infty$ therefore gives the same upper and
lower sixth-power bounds for $G_N$, with the same constants and the
same onset $A_0N^{-2/15}$.
\end{proof}

\begin{proof}[Proof of \zcref{thm:uniform-upper-tails}]
Set $\beta_0=2$. We apply
\zcref{prop:sixth-window-consequences} separately to the rescaled
finite-temperature free energy and to the ground-state energy.

For the finite-temperature statement, take
\[
 Y_N=X_{N,\beta},
 \qquad
 A_N=E_\beta(0),
 \qquad
 \tau=\frac{2}{15}.
\]
Then
\[
 \ell_N(s)
 =
 \log\E
 \exp\left\{
   s\bigl(X_{N,\beta}-NE_\beta(0)\bigr)
 \right\}
 =
 \ell_{N,\beta}^{\mathrm{en}}(s).
\]
By \zcref{lem:uniform-energy-sixth-moment}, there are constants,
independent of $\beta\ge2$, such that
\[
 \ell_N(s)\asymp Ns^6
\]
throughout
\[
 A_0N^{-2/15}\le s\le\lambda_*,
\]
with the upper bound valid on the full interval
$0<s\le\lambda_*$.

It remains to verify the convexity assumption in
\zcref{prop:sixth-window-consequences}. As a function of the
underlying Gaussian disorder, $X_{N,\beta}$ is convex and Lipschitz.
Hence \zcref{lem:short-laplace-convexity} shows that
\[
 s\longmapsto
 \frac1s\log\E e^{sX_{N,\beta}}
\]
is convex on $(0,\infty)$. Subtracting the constant
$NE_\beta(0)$ preserves convexity, so
\[
 s\longmapsto\frac{\ell_N(s)}{s}
\]
is convex as required. Gaussian Lipschitz concentration also implies
that all exponential moments are finite. Therefore
\zcref{prop:sixth-window-consequences} applies.

For the ground-state statement, take instead
\[
 Y_N=G_N,
 \qquad
 A_N=e_*,
 \qquad
 \tau=\frac{2}{15}.
\]
The zero-temperature part of
\zcref{lem:uniform-energy-sixth-moment} supplies the same sixth-power
moment window, with the same constants and the same onset
$A_0N^{-2/15}$. Since $G_N$ is the maximum of finitely many Gaussian
linear forms, it is again convex and Lipschitz. Thus
\zcref{lem:short-laplace-convexity} gives the required convexity of
\[
 s\longmapsto
 \frac1s\log\E e^{sG_N}-Ne_*
 =
 \frac{\ell_N^{\mathrm{gs}}(s)}{s}.
\]
Hence the same proposition applies directly at zero temperature.

In both applications,
\[
 5\tau=\frac23,
\]
so the lower end of the deviation window in
\zcref{prop:sixth-window-consequences} is of order
\[
 N^{-5\tau}=N^{-2/3}.
\]
The moment bounds have the same constants and parameter ranges in both
applications, so the resulting tail bounds do as well. 
\end{proof}

\begin{corollary}\zlabel{cor:uniform-energy-tilted}
There exist constants
\[
 c_0,C_0,A_0,\lambda_*>0
 \qquad\text{and}\qquad
 N_0<\infty,
\]
depending only on $\beta_0$, such that the following holds.

For $\beta\ge\beta_0$ and $\lambda\ge0$, let $\E_\lambda$ and
$\Var_\lambda$ denote expectation and variance under the disorder law
with density proportional to
\[
 e^{\lambda X_{N,\beta}}.
\]
Then, whenever
\[
 N\ge N_0,
 \qquad
 A_0N^{-2/15}\le\lambda\le\lambda_*/2,
\]
one has
\begin{equation}\label{eq:uniform-energy-tilted}
 \begin{gathered}
 c_0N\lambda^5
 \le
 \E_\lambda X_{N,\beta}-NE_\beta(0)
 \le
 C_0N\lambda^5,\\[2mm]
 c_0N\lambda^4
 \le
 \Var_\lambda(X_{N,\beta})
 \le
 \frac{N-1}{2}.
 \end{gathered}
\end{equation}
The constants are uniform for all $\beta\ge\beta_0$.

At zero temperature, let $\E_\lambda$ and $\Var_\lambda$ instead
refer to the disorder law tilted by $e^{\lambda G_N}$. Then, with
the same constants and on the same tilt window,
\[
 \begin{gathered}
 c_0N\lambda^5
 \le
 \E_\lambda G_N-Ne_*
 \le
 C_0N\lambda^5,\\[2mm]
 c_0N\lambda^4
 \le
 \Var_\lambda(G_N)
 \le
 \frac{N-1}{2}.
 \end{gathered}
\]
The variance upper bound requires no lower bound on the tilt:
at finite temperature,
\[
 \Var_\lambda(X_{N,\beta})\le\frac{N-1}{2}
 \qquad
 (N\ge1,\ 0\le\lambda\le\beta),
\]
and at zero temperature,
\[
 \Var_\lambda(G_N)\le\frac{N-1}{2}
 \qquad
 (N\ge1,\ 0\le\lambda<\infty).
\]
\end{corollary}

\begin{proof}
The two bounds on the tilted mean and the lower bound on the tilted
variance follow directly from
\zcref{prop:sixth-window-consequences}, applied to the common
sixth-power moment window supplied by
\zcref{lem:uniform-energy-sixth-moment}. Exponential tilting by
$e^{\lambda(Y_N-NA_N)}$ and by $e^{\lambda Y_N}$ gives the same
probability law, so the centering by $NE_\beta(0)$ or $Ne_*$ does not
affect the tilted measure.

It remains to prove the variance upper bound. In the off-diagonal SK
model, write the Gaussian feature vector associated with
$\sigma\in\{-1,1\}^N$ as
\[
 (a_\sigma)_{ij}
 =
 \frac{\sigma_i\sigma_j}{\sqrt N},
 \qquad i<j.
\]
Then
\[
 |a_\sigma|^2
 =
 \frac1N\binom N2
 =
 \frac{N-1}{2}.
\]
The gauge transformations
\[
 g_{ij}\mapsto\tau_i\tau_jg_{ij},
 \qquad
 \tau\in\{-1,1\}^N,
\]
are orthogonal, preserve the log partition function, and act
transitively on the spin labels. Hence
\zcref{prop:tilted-transitive-variance}, with $\kappa=\beta$ and
\[
 R^2=\frac{N-1}{2},
\]
gives
\[
 \Var_\lambda(X_{N,\beta})
 \le
 \frac{N-1}{2}
 \qquad
 (0\le\lambda\le\beta).
\]
Applying the zero-temperature assertion of the same proposition to
\[
 G_N=\max_\sigma a_\sigma\cdot g
\]
gives
\[
 \Var_\lambda(G_N)\le\frac{N-1}{2}
\]
for every finite $\lambda\ge0$. This proves all the stated bounds.
\end{proof}

The upper bound on the mean deficit also controls fluctuations around
the finite-system mean and normalized logarithmic moments at every
negative tilt. We obtain these consequences by applying
\zcref{prop:mean-moment-consequences} with the mean-deficit exponent
$a=2/3$. The corresponding quadratic fluctuation scale is
\[
 Q_N=N^{1-4a/5}=N^{7/15},
\]
and the crossover tilt is
\[
 t_N=N^{-a/5}=N^{-2/15}.
\]

\begin{corollary}\zlabel{cor:uniform-energy-concentration}
The conclusions of \zcref{prop:mean-moment-consequences} hold
uniformly for $\beta\ge\beta_0$ with
\[
 X_N=X_{N,\beta},
 \qquad
 A_N=E_\beta(0),
 \qquad
 a=\frac23,
 \qquad
 Q_N=N^{7/15},
 \qquad
 t_N=N^{-2/15}.
\]
They also hold, with the same constants and sufficiently large
$N$ threshold, for
\[
 X_N=G_N,
 \qquad
 A_N=e_*.
\]

In particular, for every $\lambda<0$ and all sufficiently large $N$,
\begin{equation}\label{eq:ground-negative-replicas}
 0
 \le
 e_*-\frac1{N\lambda}\log\E e^{\lambda G_N}
 \le
 C\left[
   N^{-2/3}+|\lambda|N^{-8/15}
 \right].
\end{equation}
Uniformly for $\beta\ge\beta_0$, the analogous finite-temperature
bound is
\[
 0
 \le
 E_\beta(0)
 -\frac1{N\lambda}\log\E e^{\lambda X_{N,\beta}}
 \le
 C\left[
   N^{-2/3}+|\lambda|N^{-8/15}
 \right].
\]

The centered upper-tail estimate has a quadratic regime and a
$6/5$-power regime. Their crossover occurs for a total deviation
\[
 u\asymp N^{1-a}=N^{1/3}.
\]
\end{corollary}

\begin{proof}
We verify the two hypotheses of
\zcref{prop:mean-moment-consequences} with $a=2/3$.

The required mean-deficit estimate is
\eqref{eq:abs-endpoint-mean}:
\[
 0
 \le
 E_\beta(0)-\frac{\E X_{N,\beta}}{N}
 \le
 CN^{-2/3},
\]
uniformly for $\beta\ge\beta_0$. Its zero-temperature limit gives
\[
 0
 \le
 e_*-\frac{\E G_N}{N}
 \le
 CN^{-2/3}.
\]

For the positive-moment hypothesis, fix
$0<\lambda\le\lambda_*$, where $\lambda_*$ lies in the common range
of the scalar upper and lower bounds. 
Independence of the diagonal completion $D_N$, together with
$\Var(D_N)=1/2$, gives
\[
 \psi^c_{N,\beta}(\lambda)
 =
 \frac1{N\lambda}\log\E e^{\lambda X_{N,\beta}}
 +\frac{\lambda}{4N}.
\]
Hence
\[
 \begin{aligned}
 \frac1{N\lambda}\log\E e^{\lambda X_{N,\beta}}
 &=
 \psi^c_{N,\beta}(\lambda)-\frac{\lambda}{4N}\\
 &\le
 E_\beta(\lambda)\\
 &\le
 E_\beta(0)+C_s\lambda^5,
 \end{aligned}
\]
where the first inequality uses
$e_{N,\beta}(\lambda)\ge0$ and the second uses
\eqref{eq:gs-scalar-trial-cost}. Thus the required fifth-order
positive-moment bound holds on a fixed interval, uniformly in
$\beta\ge\beta_0$.

At zero temperature, the same estimate follows directly by letting
$\beta\to\infty$ at fixed $N$ and $\lambda$: the soft-maximum bound
gives
\[
 X_{N,\beta}\longrightarrow G_N,
\]
while
\[
 E_\beta(0)\longrightarrow e_*,
\]
and Gaussian domination passes the exponential moments to the limit.
Therefore
\[
 \frac1{N\lambda}\log\E e^{\lambda G_N}
 \le
 e_*+C_s\lambda^5.
\]

Finally, both $X_{N,\beta}$ and $G_N$ are convex functions of the
off-diagonal Gaussian disorder. Their Lipschitz constants are at most
\[
 \sqrt{\frac{N-1}{2}}\le\sqrt{\frac N2}.
\]
Thus the Lipschitz hypothesis of
\zcref{prop:mean-moment-consequences} holds with the common constant
$L=1/\sqrt2$. All of that proposition's conclusions therefore apply
with
\[
 a=\frac23,
 \qquad
 Q_N=N^{1-4a/5}=N^{7/15},
 \qquad
 t_N=N^{-a/5}=N^{-2/15},
\]
uniformly in $\beta\ge\beta_0$ and also at zero temperature.

In particular,
\[
 N^{-4a/5}=N^{-8/15},
\]
which gives the negative-tilt bounds displayed above. The two
centered upper-tail exponents meet at $
 u=N^{1-a}=N^{1/3}$, 
giving the stated crossover scale.
\end{proof}

\subsection{Fixed-temperature consequences}
\zlabel{subsec:fixed-temperature-rescaling}

We now return from the rescaled variables to the original
positive-temperature normalization. Fix $\beta>1$, and apply the
preceding estimates with $\beta_0=\beta$. All constants in this
subsection may therefore depend on this fixed value of $\beta$.

The change of variables
\[
 \lambda=\beta s
\]
gives
\begin{equation}\label{eq:fixed-energy-rescaling}
 P_\beta(s)=\beta E_\beta(\beta s),\qquad
 \phi^c_{N,\beta}(s)
   =\beta\psi^c_{N,\beta}(\beta s),\qquad
 F_{N,\beta}=\beta X_{N,\beta}.
\end{equation}
Thus every estimate proved above on the fixed interval
$0<\lambda\le\lambda_0$ transfers to
\[
 0<s\le s_\beta,
 \qquad
 s_\beta\le\min\left\{\frac12,\frac{\lambda_0}{\beta}\right\}.
\]
Since $\beta$ is fixed, powers of $\beta$ arising from this rescaling
are absorbed into constants depending only on $\beta$.

\begin{proposition}\zlabel{prop:scalar-gap}
For every fixed $\beta>1$, there exist constants
\[
 s_\beta>0,\qquad c_\beta>0,\qquad C_\beta<\infty
\]
such that
\begin{equation}\label{eq:scalar-gap}
 c_\beta s^5
 \le
 P_\beta(s)-P_\beta
 \le
 C_\beta s^5,
 \qquad 0<s\le s_\beta.
\end{equation}

Let $\alpha$ be any unconstrained minimizing cumulative profile for
the scalar Parisi functional at inverse temperature $\beta$, and define
the profile obtained by imposing the floor $s$ by
\[
 \widehat a_s(q)=\max\{s,\alpha(q)\}.
\]
Then there exist constants
\[
 q_0>0,\qquad k>0,\qquad K<\infty,
\]
depending only on $\beta$, such that, for
$0<s\le s_\beta$,
\begin{equation}\label{eq:scalar-trial}
 0
 \le
 \cP_\beta(\widehat a_s)-P_\beta(s)
 \le
 C_\beta s^5,
\end{equation}
and
\[
 kq
 \le
 \widehat a_s(q)
 \le
 s+Kq,
 \qquad 0\le q\le q_0.
\]
In particular, the constrained scalar minimum differs from the
unconstrained minimum by exactly order $s^5$, while the explicit trial
profile $\widehat a_s$ realizes this order uniformly on the stated
small-$s$ interval.
\end{proposition}

\begin{proof}
Fix $\beta>1$ and apply
\zcref{lem:gs-scalar-root,prop:scalar-lower-via-fisher} with the
temperature threshold chosen to be $\beta$. Let $\alpha$ be an
unconstrained minimizing cumulative profile in the original
positive-temperature normalization. Under the rescaling
\[
 \gamma=\beta\alpha,
\]
the rescaled minimizer is
\[
 \gamma_\beta=\beta\alpha.
\]
For $\lambda=\beta s$, its floor trial is
\[
 \gamma_{\beta,\beta s}^{\mathrm{tr}}
 =
 \max\{\beta s,\gamma_\beta\}
 =
 \beta\max\{s,\alpha\}
 =
 \beta\widehat a_s.
\]

The two floor estimates give, for $0<\beta s\le\lambda_*$,
\[
 c_0(\beta s)^5
 \le
 E_\beta(\beta s)-E_\beta(0)
 \le
 C_0(\beta s)^5.
\]
Using
\[
 P_\beta(s)=\beta E_\beta(\beta s),
 \qquad
 P_\beta=\beta E_\beta(0),
\]
we obtain
\[
 c_0\beta^6s^5
 \le
 P_\beta(s)-P_\beta
 \le
 C_0\beta^6s^5.
\]
Since $\beta$ is fixed, the factors of $\beta$ are absorbed into
constants $c_\beta,C_\beta>0$. After taking
$s_\beta\le\lambda_*/\beta$, this proves
\eqref{eq:scalar-gap}.

The same rescaling gives the estimate for the explicit trial. Indeed,
\zcref{lem:gs-scalar-root} gives
\[
 E_\beta(\beta s)
 \le
 \cP_\beta^{\mathrm{en}}
   (\gamma_{\beta,\beta s}^{\mathrm{tr}})
 \le
 E_\beta(0)+C_0(\beta s)^5.
\]
Since
\[
 \cP_\beta(\widehat a_s)
 =
 \beta\,
 \cP_\beta^{\mathrm{en}}
   (\gamma_{\beta,\beta s}^{\mathrm{tr}}),
\]
we have
\[
 0
 \le
 \cP_\beta(\widehat a_s)-P_\beta(s)
 \le
 C_\beta s^5,
\]
after enlarging $C_\beta$.

Finally, the near-zero bounds for $\gamma_\beta$ give
\[
 k_0q\le\gamma_\beta(q)\le C_0q
\]
on a fixed interval $[0,q_0]$. Dividing by $\beta$ yields
\[
 \frac{k_0}{\beta}q
 \le
 \alpha(q)
 \le
 \frac{C_0}{\beta}q.
\]
Since
\[
 \alpha\le\widehat a_s=\max\{s,\alpha\}\le s+\alpha,
\]
we conclude that
\[
 kq\le\widehat a_s(q)\le s+Kq,
 \qquad 0\le q\le q_0,
\]
with
\[
 k=\frac{k_0}{\beta},
 \qquad
 K=\frac{C_0}{\beta}.
\]
All constants depend only on the fixed inverse temperature $\beta$.
\end{proof}

\begin{proposition}\zlabel{prop:sharp-comparison}
For every fixed $\beta>1$, there exist constants
\[
 C_\beta<\infty,
 \qquad
 0<s_\beta\le\frac12,
\]
such that, for every $N\ge1$ and $0<s\le s_\beta$,
\begin{equation}\label{eq:sharp-comparison}
 0
 \le
 P_\beta(s)-\phi^c_{N,\beta}(s)
 \le
 C_\beta\left[
   N^{-2/3}
   +N^{-3/4}s^{-1/2}
   +\frac{\log(2+N)}{Ns}
 \right].
\end{equation}
\end{proposition}

\begin{proof}
Fix $\beta>1$ and apply
\zcref{prop:abs-selfconsistent-comparison} with the temperature
threshold chosen to be $\beta$. Set
\[
 \lambda=\beta s.
\]
By \eqref{eq:fixed-energy-rescaling},
\[
 P_\beta(s)-\phi^c_{N,\beta}(s)
 =
 \beta\left[
   E_\beta(\beta s)-\psi^c_{N,\beta}(\beta s)
 \right].
\]
Hence \eqref{eq:abs-comparison} gives
\[
 \begin{aligned}
 P_\beta(s)-\phi^c_{N,\beta}(s)
 &\le
 C\beta\left[
   N^{-2/3}
   +N^{-3/4}(\beta s)^{-1/2}
   +\frac{\log(2+N)}{N\beta s}
 \right]\\
 &=
 C\left[
   \beta N^{-2/3}
   +\beta^{1/2}N^{-3/4}s^{-1/2}
   +\frac{\log(2+N)}{Ns}
 \right].
\end{aligned}
\]
Choose
\[
 s_\beta
 \le
 \min\left\{
   \frac12,\frac{\lambda_0}{\beta}
 \right\},
\]
so that \eqref{eq:abs-comparison} applies whenever 
$0<s\le s_\beta$. Since $\beta$ is fixed, the factors $\beta$ and
$\beta^{1/2}$ may be absorbed into a constant $C_\beta$. The lower
bound follows from the general comparison
$P_\beta(s)\ge\phi^c_{N,\beta}(s)$.
\end{proof}

\begin{proof}[Proof of \zcref{thm:main}]
Fix $\beta>1$ and apply \zcref{thm:uniform-energy-powers} with
$\beta_0=\beta$. Since
\[
 F_{N,\beta}=\beta X_{N,\beta},
\]
we have
\[
 V_{N,\beta}
 =
 \Var(F_{N,\beta})
 =
 \beta^2\Var(X_{N,\beta}),
\]
while
\[
 P_\beta-p_{N,\beta}
 =
 \beta\left(
 E_\beta(0)-\frac{\E X_{N,\beta}}{N}
 \right).
\]
Thus the four bounds in
\eqref{eq:uniform-energy-four-bounds}, after absorbing the fixed
powers of $\beta$ into constants depending only on $\beta$, give
\eqref{eq:fixed-beta-four-bounds}. The upper bounds extend to every
$N\ge1$ after increasing the corresponding constants, while the
lower bounds hold beyond a threshold depending only on $\beta$.
\end{proof}

\begin{remark}
The same rescaling gives a corresponding variance--mean relation.
Indeed, \zcref{lem:short-variance-mean} gives
\[
 E_\beta(0)-\frac{\E X_{N,\beta}}{N}
 \ge
 c\left(
 \frac{\Var(X_{N,\beta})+1/2}{N}
 \right)^{5/4}.
\]
Multiplying by $\beta$ and using
\[
 \Var(X_{N,\beta})+\frac12
 =
 \beta^{-2}\left(
   V_{N,\beta}+\frac{\beta^2}{2}
 \right)
\]
yields, after absorbing the fixed factor $\beta^{-3/2}$ into the
constant,
\begin{equation}\label{eq:fixed-beta-variance-mean}
 P_\beta-p_{N,\beta}
 \ge
 c_\beta
 \left(
   \frac{V_{N,\beta}+\beta^2/2}{N}
 \right)^{5/4}.
\end{equation}
Thus the independent diagonal completion remains explicit in the
original free-energy normalization through the term $\beta^2/2$.
\end{remark}

\begin{lemma}\zlabel{lem:sixth-moment}
Fix $\beta>1$, and define
\[
 \ell_N(s)
 =
 \log\E\exp\{s(F_{N,\beta}-NP_\beta)\}.
\]
There exist constants
\[
 a_\beta,b_\beta,\kappa_\beta,s_\beta>0
 \qquad\text{and}\qquad
 N_0(\beta)<\infty,
\]
depending only on $\beta$, such that, for every $N\ge N_0(\beta)$,
\begin{equation}\label{eq:sixth-moment}
 a_\beta Ns^6
 \le
 \ell_N(s)
 \le
 b_\beta Ns^6
 \qquad
 \left(
   \kappa_\beta N^{-2/15}\le s\le s_\beta
 \right).
\end{equation}
After increasing $b_\beta$ if necessary, the upper bound
\[
 \ell_N(s)\le b_\beta Ns^6
\]
holds for every $N\ge1$ and every $0<s\le s_\beta$.
\end{lemma}

\begin{proof}
Return to the variables
\[
 X_{N,\beta}=\frac{F_{N,\beta}}{\beta},
 \qquad
 E_\beta(0)=\frac{P_\beta}{\beta},
 \qquad
 \lambda=\beta s.
\]
Then
\[
 \begin{aligned}
 \ell_N(s)
 &=
 \log\E
 \exp\left\{
   s(F_{N,\beta}-NP_\beta)
 \right\}\\
 &=
 \log\E
 \exp\left\{
   \beta s\bigl(X_{N,\beta}-NE_\beta(0)\bigr)
 \right\}\\
 &=
 \ell_{N,\beta}^{\mathrm{en}}(\beta s).
 \end{aligned}
\]
Apply \zcref{lem:uniform-energy-sixth-moment} with the temperature
threshold chosen to be $\beta$. If its constants are
$a_0,b_0,A_0,\lambda_*$, then
\[
 a_0N(\beta s)^6
 \le
 \ell_N(s)
 \le
 b_0N(\beta s)^6
\]
whenever
\[
 A_0N^{-2/15}\le\beta s\le\lambda_*.
\]
Thus \eqref{eq:sixth-moment} holds with, for example,
\[
 a_\beta=a_0\beta^6,\qquad
 b_\beta=b_0\beta^6,\qquad
 \kappa_\beta=\frac{A_0}{\beta},\qquad
 s_\beta=\frac{\lambda_*}{\beta}.
\]
Decreasing $s_\beta$ if necessary to remain in the fixed-temperature
tilt range gives the stated result. The all-$N$ upper bound follows
from the corresponding all-$N$ assertion in
\zcref{lem:uniform-energy-sixth-moment}.
\end{proof}

\begin{proof}[Proof of \zcref{thm:upper-tails}]
We apply \zcref{prop:sixth-window-consequences} with
\[
 Y_N=F_{N,\beta},
 \qquad
 A_N=P_\beta,
 \qquad
 \tau=\frac{2}{15}.
\]
The required sixth-power moment window is exactly
\zcref{lem:sixth-moment}: its lower bound begins at a constant
multiple of $N^{-2/15}$, while its upper bound holds throughout a
fixed interval $0<s\le s_\beta$.

It remains to verify convexity of the normalized logarithmic moment.
As a function of the Gaussian disorder, $F_{N,\beta}$ is convex and
Lipschitz. Hence \zcref{lem:short-laplace-convexity} implies that
\[
 s\longmapsto
 \frac1s\log\E e^{sF_{N,\beta}}
\]
is convex on $(0,\infty)$. Subtracting the constant $NP_\beta$ shows
that
\[
 s\longmapsto\frac{\ell_N(s)}{s}
\]
is convex as required. Gaussian Lipschitz bounds also ensure finite
exponential moments of every order.

Therefore \zcref{prop:sixth-window-consequences} applies. Since
\[
 5\tau=\frac23,
\]
the matching two-sided upper-tail estimate begins at deviations per
spin of order $N^{-2/3}$, and its exponent is $Nx^{6/5}$, as asserted
in \zcref{thm:upper-tails}.
\end{proof}

\begin{corollary}\zlabel{cor:fixed-beta-tilted}
Fix $\beta>1$. There exist constants
\[
 c_\beta,C_\beta,\kappa_\beta,s_\beta>0
 \qquad\text{and}\qquad
 N_0(\beta)<\infty
\]
such that the following holds. Let $\E_s$ and $\Var_s$ denote
expectation and variance under the disorder law tilted by
$e^{sF_{N,\beta}}$. Then, whenever
\[
 N\ge N_0(\beta),
 \qquad
 \kappa_\beta N^{-2/15}\le s\le s_\beta/2,
\]
one has
\begin{equation}\label{eq:fixed-beta-tilted}
 \begin{gathered}
 c_\beta Ns^5
 \le
 \E_sF_{N,\beta}-NP_\beta
 \le
 C_\beta Ns^5,\\[2mm]
 c_\beta Ns^4
 \le
 \Var_s(F_{N,\beta})
 \le
 \frac{\beta^2(N-1)}{2}.
 \end{gathered}
\end{equation}
Additionally, the variance upper bound holds for all $s\in[0,1]$. 
\end{corollary}

\begin{proof}
Recall that
\[
 F_{N,\beta}=\beta X_{N,\beta}.
\]
Tilting the original free energy by $e^{sF_{N,\beta}}$ is therefore
exactly the same as tilting $X_{N,\beta}$ by
\[
 e^{\lambda X_{N,\beta}},
 \qquad
 \lambda=\beta s.
\]
Apply \zcref{cor:uniform-energy-tilted} with the temperature threshold
chosen to be $\beta$.

On its common tilt window,
\[
 c_0N\lambda^5
 \le
 \E_\lambda X_{N,\beta}-NE_\beta(0)
 \le
 C_0N\lambda^5,
\]
and
\[
 c_0N\lambda^4
 \le
 \Var_\lambda(X_{N,\beta})
 \le
 \frac{N-1}{2}.
\]
Since
\[
 P_\beta=\beta E_\beta(0),
\]
multiplying the mean estimate by $\beta$ and substituting
$\lambda=\beta s$ gives
\[
 c_0\beta^6Ns^5
 \le
 \E_sF_{N,\beta}-NP_\beta
 \le
 C_0\beta^6Ns^5.
\]
Similarly,
\[
 \Var_s(F_{N,\beta})
 =
 \beta^2\Var_\lambda(X_{N,\beta}),
\]
so
\[
 c_0\beta^6Ns^4
 \le
 \Var_s(F_{N,\beta})
 \le
 \frac{\beta^2(N-1)}{2}.
\]
Because $\beta$ is fixed, the powers of $\beta$ are absorbed into
constants $c_\beta$ and $C_\beta$.

If the interval in $\lambda$ has endpoints $A_0N^{-2/15}$ and
$\lambda_*/2$, then in the variable $s=\lambda/\beta$ these become 
\[
 \kappa_\beta N^{-2/15},
 \qquad
 \frac{s_\beta}{2},
\]
with
\[
 \kappa_\beta=\frac{A_0}{\beta},
 \qquad
 s_\beta=\frac{\lambda_*}{\beta}.
\]
Finally, the variance cap in
\zcref{cor:uniform-energy-tilted} holds for
$0\le\lambda\le\beta$, which under $\lambda=\beta s$ is exactly $
 0\le s\le1$. 
\end{proof}

\begin{corollary}\zlabel{cor:fixed-beta-concentration}
Fix $\beta>1$. The conclusions of
\zcref{prop:mean-moment-consequences} hold with
\[
 X_N=F_{N,\beta},\qquad
 A_N=P_\beta,\qquad
 a=\frac23,\qquad
 Q_N=N^{7/15},\qquad
 t_N=N^{-2/15},
\]
with constants and the sufficiently large $N$ threshold allowed to
depend on $\beta$.

In particular, for every $s<0$ and all sufficiently large $N$,
\begin{equation}\label{eq:fixed-beta-negative-replicas}
 0
 \le
 P_\beta-\frac1{Ns}\log\E e^{sF_{N,\beta}}
 \le
 C_\beta\left[
   N^{-2/3}+|s|N^{-8/15}
 \right].
\end{equation}
Here $N^{-2/3}$ is the mean-deficit contribution, while the term
$|s|N^{-8/15}$ comes from the quadratic bound on the centered
logarithmic moment.

Moreover, for every $u\ge0$,
\[
 \Prob\{F_{N,\beta}-\E F_{N,\beta}\ge u\}
 \le
 \exp\left[
  -c_\beta\min\left\{
    \frac{u^2}{N^{7/15}},
    \frac{u^{6/5}}{N^{1/5}}
  \right\}
 \right].
\]
\end{corollary}

\begin{proof}
We verify the hypotheses of
\zcref{prop:mean-moment-consequences} with $a=2/3$.

The required mean-deficit estimate follows from
\eqref{eq:fixed-beta-four-bounds}:
\[
 0
 \le
 P_\beta-\frac{\E F_{N,\beta}}{N}
 \le
 C_\beta N^{-2/3}.
\]

For the positive-moment hypothesis, take
$0<s\le s_\beta$ in the common scalar comparison interval. Since the
independent diagonal completion has variance $\beta^2/2$,
\[
 \frac1{Ns}\log\E e^{sF_{N,\beta}}
 =
 \phi^c_{N,\beta}(s)-\frac{\beta^2s}{4N}.
\]
The general comparison and the scalar trial bound give
\[
 \phi^c_{N,\beta}(s)
 \le
 P_\beta(s)
 \le
 P_\beta+C_\beta s^5.
\]
Consequently
\[
 \frac1{Ns}\log\E e^{sF_{N,\beta}}
 \le
 P_\beta+C_\beta s^5,
 \qquad
 0<s\le s_\beta.
\]
Thus the two quantitative hypotheses of
\zcref{prop:mean-moment-consequences} hold with constants depending
only on the fixed inverse temperature $\beta$.

Finally, $F_{N,\beta}$ is a convex function of the off-diagonal
Gaussian disorder. Its Gaussian feature vectors have squared norm
$\beta^2(N-1)/2$, so its Lipschitz constant is at most
\[
 \beta\sqrt{\frac{N-1}{2}}
 \le
 \frac{\beta}{\sqrt2}\sqrt N.
\]
Hence the Lipschitz hypothesis holds with
$L=\beta/\sqrt2$. We may therefore apply
\zcref{prop:mean-moment-consequences} with
\[
 a=\frac23.
\]
Its associated scales are
\[
 Q_N=N^{1-4a/5}=N^{7/15},
 \qquad
 t_N=N^{-a/5}=N^{-2/15}.
\]
Moreover,
\[
 N^{-4a/5}=N^{-8/15},
\]
which gives \eqref{eq:fixed-beta-negative-replicas}. The two
centered upper-tail rates cross at $
 u=N^{1-a}=N^{1/3}$.
\end{proof}

\appendix
\section{Finite-sample response and curvature estimates}
\zlabel{app:response-estimates}

The scalar lower floor bound and the reflected-regularization argument
both rely on a common set of finite-$N$ response identities. We collect
those identities here. The main quantities are the second variation of
the interpolating functional, a pairing between a profile perturbation
and the Hessian response, and a localized Fisher-information term.
For arbitrary scalar competitors, the Fisher-information term
is controlled through a Stieltjes integral involving the profile
measure; for profiles with bounded density, it can instead be
estimated directly by that density.

Fix
\[
 N\ge1,\qquad \beta>0,\qquad 0\le t<1,\qquad
 r=1-t,\qquad 0\le\lambda\le\beta,
\]
and set
\[
 J_{ij}=\sqrt{\frac{t}{N}}\,g_{ij}\quad(i<j),
 \qquad
 J_{ii}=0,
\]
where the variables $(g_{ij})_{i<j}$ are independent standard
Gaussians. Let $\gamma$ be a deterministic nondecreasing
right-continuous cumulative profile satisfying
\[
 \lambda\le\gamma(q)\le\beta,\qquad 0\le q<1,
 \qquad
 \gamma(0-)=0,\qquad
 \gamma(1)=\beta.
\]
For this profile, let $u=u_\gamma$ solve
\[
 \partial_q u
 +\frac r2\bigl(\Delta u+\gamma(q)|\nabla u|^2\bigr)=0,
\]
with terminal condition
\[
 u(1,y;J)
 =
 \frac1\beta
 \log\sum_{\sigma\in\{-1,1\}^N}
 \exp\left\{
   \beta\left(\frac12\sigma^TJ\sigma+y\cdot\sigma\right)
 \right\}.
\]
We write
\[
 f_{\beta,J}(y)=u(1,y;J)
\]
for this terminal function.

Let $P_g$ denote the product standard Gaussian law of the
off-diagonal coordinates $(g_{ij})_{i<j}$, and let $\E_g$ denote
expectation with respect to $P_g$. The outer logarithmic moment
changes the matrix law to
\[
 Q_\gamma(\dd g)
 =
 \frac{e^{\lambda u_\gamma(0,0;J)}}
 {\E_g e^{\lambda u_\gamma(0,0;J)}}\,P_g(\dd g).
\]
When $\lambda=0$, this is simply the original Gaussian matrix law.

Conditional on the matrix $J$, let $Y_q$ solve
\[
 \dd Y_q
 =
 r\gamma(q)m_q\,\dd q+\sqrt r\,\dd B_q,
 \qquad
 Y_0=0,
\]
where $B$ is standard $N$-dimensional Brownian motion and
\[
 m_q=\nabla_yu_\gamma(q,Y_q;J),
 \qquad
 H_q=D_y^2u_\gamma(q,Y_q;J).
\]
When it is useful to distinguish the Hessian as a function of the
field variable from its value along the diffusion, we write
\[
 H(q,y;J)=D_y^2u_\gamma(q,y;J),
 \qquad
 H_q=H(q,Y_q;J).
\]

The prescribed root mass is $\lambda\delta_0$. Thus the remaining
free measure is
\[
 \dd\gamma-\lambda\delta_0,
\]
which may contain additional mass
$\gamma(0)-\lambda$ at zero. We write $\E_\gamma$ for expectation
under the joint law obtained by first sampling the matrix from
$Q_\gamma$ and then sampling the diffusion conditional on that
matrix. If the current profile is $\gamma_\theta$, we write
$\E_\theta$. Unless another law is explicitly indicated, all process
expectations below use this convention.

Define the deterministic response functions
\[
 b_\gamma(q)
 =
 \frac1N\E_\gamma\Tr H_q,
 \qquad
 \zeta_\gamma(q)
 =
 \frac1N\E_\gamma\Tr H_q^2.
\]

The completed interpolation functional is
\[
 \mathfrak F_t^c(\gamma)
 =
 \frac1{N\lambda}
 \log\E_g e^{\lambda u_\gamma(0,0;J)}
 -\frac r2\int_0^1q\gamma(q)\,\dd q
 +\frac{\lambda t}{4N}.
\]
At $\lambda=0$, the first term is interpreted as
\[
 \frac1N\E_g u_\gamma(0,0;J).
\]
The final deterministic term is the contribution of the independent
diagonal Gaussian completion and has zero derivative with respect to
the profile.

The response identities and the pointwise early-time Hessian bound
below hold for every $\beta>0$. The assumption $\beta\ge1$ is needed
only for the averaged curvature estimate, whose proof uses that the
profile eventually reaches height one. For an admissible segment
$\gamma_\theta=\gamma_0+\theta v$, differentiating
$\mathfrak F_t^c(\gamma_\theta)$ produces a nonnegative term in the
second variation. The results below estimate this term and relate
it to the perturbation $v$ through the Hessian response.

\begin{proposition}\zlabel{prop:gs-response-tools}
Let $\gamma_0$ and $\gamma_1$ be two admissible profiles with the same
outer tilt $\lambda$, and set
\[
 \gamma_\theta=(1-\theta)\gamma_0+\theta\gamma_1
 =\gamma_0+\theta v,
 \qquad
 v=\gamma_1-\gamma_0,
 \qquad 0\le\theta\le1.
\]
Additional mass at zero may differ between $\gamma_0$ and $\gamma_1$.
Derivatives with respect to $\theta$ at $\theta=0,1$ are interpreted
one-sided.

For fixed $\theta$, define the Legendre transform in the field variable
by
\[
 \Phi_\theta(q,m;J)
 =
 u_{\gamma_\theta}(q,\cdot;J)^*(m),
 \qquad m\in(-1,1)^N.
\]
In the magnetization variable $m$, set
\[
 h=(D_m^2\Phi_\theta)^{-1},
 \qquad
 B=D_m^2(\partial_\theta\Phi_\theta),
 \qquad
 \dot h=\left.\partial_q h\right|_m.
\]
If
\[
 m=\nabla_yu_{\gamma_\theta}(q,y;J),
\]
then Legendre duality gives
\[
 D_m^2\Phi_\theta(q,m;J)
 =
 \bigl(D_y^2u_{\gamma_\theta}(q,y;J)\bigr)^{-1},
\]
and hence
\[
 h(q,m;J)=H(q,y;J).
\]
Thus $h$ and $H$ represent the same Hessian at dual magnetization and
field variables. The derivative $\dot h$ is taken almost everywhere in $q$ at fixed
$m$.

Under the joint law defining $\E_\theta$, evaluate $h$, $B$, and
$\dot h$ at the random magnetization $m_q$, and define
\[
 e_\theta(q)
 =
 \frac1N\E_\theta\Tr(hBhBh),
 \qquad
 \mathcal J_\theta(q)
 =
 \frac1N\E_\theta
 \Tr(\dot h\,h^{-1}\dot h).
\]
Since $h$ is symmetric positive definite and $B$ and $\dot h$ are
symmetric,
\[
 \Tr(hBhBh)=\|h^{1/2}Bh\|_{\mathrm F}^2,
 \qquad
 \Tr(\dot h\,h^{-1}\dot h)
 =\|h^{-1/2}\dot h\|_{\mathrm F}^2,
\]
so $e_\theta(q)$ and $\mathcal J_\theta(q)$ are nonnegative. Moreover,
\begin{equation}\label{eq:shared-response}
 \frac{\dd^2}{\dd\theta^2}
 \mathfrak F_t^c(\gamma_\theta)
 =
 r\int_0^1 e_\theta(q)\,\dd q
 +
 \frac{\lambda}{N}
 \Var_{Q_{\gamma_\theta}}
 \left(
   \partial_\theta
   u_{\gamma_\theta}(0,0;J)
 \right).
\end{equation}
In particular,
\[
 \frac{\dd^2}{\dd\theta^2}
 \mathfrak F_t^c(\gamma_\theta)
 \ge
 r\int_0^1e_\theta(q)\,\dd q.
\]
\end{proposition}

We next control the weighted pairing between the profile perturbation
$v$ and the Hessian response $\zeta_\theta$. The first estimate bounds
this pairing by the second-variation term $e_\theta$ together with a
localized Fisher-information term. The second estimate controls that
Fisher-information term by the cutoff derivative and the curvature
integrated against the profile measure.

\begin{lemma}\zlabel{lem:response-fisher}
Let the profiles and response quantities be as in
\zcref{prop:gs-response-tools}. For every compactly supported Lipschitz
function $\varphi:(0,1)\to\mathbb R$,
\begin{align}
 r\left|\int_0^1v\zeta_\theta\varphi\,\dd q\right|
 &\le
 \left(\int_0^1e_\theta\,\dd q\right)^{1/2}
 \left[
  \left(\int_0^1b_\theta|\varphi'|^2\,\dd q\right)^{1/2}
  +2\left(\int_0^1\mathcal J_\theta\varphi^2\,\dd q\right)^{1/2}
 \right],
 \label{eq:shared-pairing}\\
 \int_0^1\varphi^2\mathcal J_\theta\,\dd q
 &\le
 \int_0^1|\varphi'|^2b_\theta\,\dd q
 +r\int_{(0,1)}\varphi^2\zeta_\theta\,\dd\gamma_\theta.
 \label{eq:shared-fisher}
\end{align}
\end{lemma}

To apply these estimates, we will bound $b_\theta$ from above near
$q=0$ and obtain an averaged lower bound on the Hessian under the joint
law defining $\E_\theta$.

\begin{lemma}\zlabel{lem:response-curvature}
Let $\gamma$ be any admissible profile in the preceding setup. If
$q<T\le1$ and
\[
 \gamma(s)\le A,\qquad q\le s<T,
\]
then, for every $y\in\mathbb R^N$ and every $1\le i\le N$,
\begin{equation}\label{eq:shared-early-diagonal}
 H_{ii}(q,y;J)
 \le
 A+[r(T-q)]^{-1/2}.
\end{equation}

Assume in addition that $\beta\ge1$. For every $q<1$ and $U\ge0$,
\begin{equation}\label{eq:shared-averaged-curvature}
 \gamma(q)\le U
 \quad\Longrightarrow\quad
 \E_\gamma H_q\succeq c_UI,
\end{equation}
where
\[
 c_U
 =
 \frac{e^{-16c_0M_U-2M_U^2}}{2(1+c_0)},
 \qquad
 c_0=\sqrt{\pi/2},
 \qquad
 M_U=\max\{1,U\}.
\]
Consequently, under the same hypothesis,
\[
 \zeta_\gamma(q)\ge c_U^2.
\]
\end{lemma}

We prove the response and curvature estimates first for step profiles.
For such profiles, the required differentiations can be carried out on
each interval on which $\gamma$ is constant. We then establish bounds
uniform in the step partition, with $N$ and $\beta$ fixed, and pass to
general admissible profiles by approximation. The auxiliary analytic
bounds used in this passage may depend on $N$ and $\beta$, whereas the
constants displayed in the lemma are independent of the approximating
partition.

\subsection{Spin-label offsets and curvature bounds}

To compare the response to perturbations of the spin energies with
curvature in the field variables, we perturb the terminal energy
associated with each spin label separately. The resulting Hessian in
these offset variables has nonpositive off-diagonal entries and zero
row sums. These properties give the matrix and directional estimates
used below. The Gaussian increments after the current time remain
averaged through the usual Cole--Hopf recursion.

\begin{lemma}\zlabel{lem:spin-offset-hessian}
Let $\gamma$ be a step profile. For each
$\sigma\in\{-1,1\}^N$, replace the terminal energy $E_\sigma$ by
$E_\sigma+z_\sigma$. Fix $q$ and $y$, and let $U(z)$ denote the value
obtained after integrating the Gaussian increments between $q$ and the
terminal time. Set
\[
 a=\gamma(q),
\]
with $a=\beta$ when $q=1$, and define
\[
 p=\nabla_zU,
 \qquad
 K=D_z^2U.
\]
Then $p$ is a probability vector,
\[
 K\1=0,
\]
and
\begin{equation}\label{eq:spin-offset-cone}
 K_{\sigma\tau}\le-a p_\sigma p_\tau
 \quad(\sigma\ne\tau),
 \qquad
 a(\operatorname{diag}p-pp^T)\preceq K
 \preceq\beta(\operatorname{diag}p-pp^T).
\end{equation}
Moreover, if $s_\sigma\in\{-1,1\}$ and $|\xi_\sigma|\le1$ for every
$\sigma$, then
\begin{equation}\label{eq:spin-offset-direction}
 |s^TK\xi|\le s^TKs,
 \qquad
 a[1-(p\cdot s)^2]
 \le s^TKs
 \le\beta[1-(p\cdot s)^2].
\end{equation}
\end{lemma}

\begin{proof}
At the terminal time,
\[
 U(z)
 =
 \frac1\beta
 \log\sum_\sigma e^{\beta(E_\sigma+z_\sigma)},
\]
so
\[
 p_\sigma
 =
 \frac{e^{\beta(E_\sigma+z_\sigma)}}
 {\sum_\tau e^{\beta(E_\tau+z_\tau)}},
 \qquad
 K=\beta(\operatorname{diag}p-pp^T).
\]
Thus all assertions hold at the terminal time.

Consider one backward recursion step
\[
 U=\frac1a\log\E e^{aU'},
\]
and define
\[
 \E_a F
 =
 \frac{\E[Fe^{aU'}]}{\E e^{aU'}},
\]
where $\E$ integrates the Gaussian increment in this step. When
$a=0$, $\E_a$ denotes ordinary expectation. Differentiation gives
\[
 p=\E_a p',
 \qquad
 K=\E_aK'+a\Cov_a(p'),
\]
where
\[
 p'=\nabla_zU',
 \qquad
 K'=D_z^2U'.
\]
In particular, if $p'$ is a probability vector, then so is $p$.

Suppose the profile height at the next step is $b\ge a$. Using the
inductive bound
\[
 K'_{\sigma\tau}\le-bp'_\sigma p'_\tau
 \qquad(\sigma\ne\tau),
\]
we obtain
\[
 \begin{aligned}
 K_{\sigma\tau}
 &\le
 -b\,\E_a(p'_\sigma p'_\tau)
 +a\Bigl(
   \E_a(p'_\sigma p'_\tau)-p_\sigma p_\tau
 \Bigr) \\
 &=
 -(b-a)\E_a(p'_\sigma p'_\tau)
 -a p_\sigma p_\tau
 \le
 -a p_\sigma p_\tau.
 \end{aligned}
\]
Adding the same constant to every offset adds that constant to $U$.
Hence
\[
 K\1=0.
\]
Also, if $K'\succeq0$, then
\[
 K=\E_aK'+a\Cov_a(p')\succeq0.
\]

For the upper matrix bound, the inductive hypothesis gives
\[
 K'
 \preceq
 \beta(\operatorname{diag}p'-p'p'^T).
\]
Therefore
\[
 \begin{aligned}
 K
 &\preceq
 \beta\E_a(\operatorname{diag}p'-p'p'^T)
 +a\Cov_a(p') \\
 &=
 \beta(\operatorname{diag}p-pp^T)
 -(\beta-a)\Cov_a(p') \\
 &\preceq
 \beta(\operatorname{diag}p-pp^T).
 \end{aligned}
\]

Now set
\[
 w_{\sigma\tau}=-K_{\sigma\tau}\ge0
 \qquad(\sigma\ne\tau).
\]
Since $K\1=0$,
\[
 x^TKy
 =
 \frac12\sum_{\sigma\ne\tau}
 w_{\sigma\tau}
 (x_\sigma-x_\tau)(y_\sigma-y_\tau).
\]
The bound
\[
 w_{\sigma\tau}\ge a p_\sigma p_\tau
\]
then gives
\[
 K\succeq a(\operatorname{diag}p-pp^T),
\]
which proves \eqref{eq:spin-offset-cone}.

For the first inequality in \eqref{eq:spin-offset-direction}, only
pairs with $s_\sigma\ne s_\tau$ contribute, and for such pairs
\[
 |(s_\sigma-s_\tau)(\xi_\sigma-\xi_\tau)|
 \le4
 =(s_\sigma-s_\tau)^2.
\]
Hence
\[
 |s^TK\xi|\le s^TKs.
\]
Finally, applying the two matrix bounds in
\eqref{eq:spin-offset-cone} to $s$ and using
\[
 s^T(\operatorname{diag}p-pp^T)s
 =
 1-(p\cdot s)^2
\]
gives the remaining two inequalities.
\end{proof}

\begin{proof}[Proof of \zcref{lem:response-curvature}]
Changing the field coordinate $y_i$ adds the offset
$z_\sigma=y_i\sigma_i$ to the energy of spin $\sigma$. Hence, with
$s_\sigma=\sigma_i$, one has
\[
 H_{ii}=s^TKs,
 \qquad
 p\cdot s=m_i.
\]
The directional bounds in \zcref{lem:spin-offset-hessian} therefore give
\begin{equation}\label{eq:response-diagonal-basic}
 \gamma(q)(1-m_i^2)
 \le H_{ii}
 \le \beta(1-m_i^2)
 \le\beta.
\end{equation}

\emph{Early diagonal bound.}
Restart the diffusion from $(q,y)$ with $J$ fixed, and let
$\E_{q,y}$ denote expectation over this conditional diffusion. Along
the diffusion,
\[
 \dd m_s=\sqrt r\,H_s\,\dd B_s,
 \qquad
 \dd H_s=-r\gamma(s)H_s^2\,\dd s
          +\sqrt r\,D_yH_s\,\dd B_s.
\]
Thus, for $q\le s\le T$,
\[
 H_{ii}(q,y)-\E_{q,y}H_{ii}(s,Y_s)
 =
 r\int_q^s
 \gamma(a)\E_{q,y}(H_a^2)_{ii}\,\dd a.
\]
On the other hand, It\^o isometry gives
\[
 \E_{q,y}(m_{s,i}-m_{q,i})^2
 =
 r\int_q^s\E_{q,y}(H_a^2)_{ii}\,\dd a.
\]
If $\gamma\le A$ on $[q,T)$, then
\[
 H_{ii}(q,y)-\E_{q,y}H_{ii}(s,Y_s)
 \le
 A\,\E_{q,y}(m_{s,i}-m_{q,i})^2
 \le A,
\]
where the last inequality uses that $(m_{s,i})_{s\ge q}$ is a
martingale with values in $[-1,1]$.

The cross It\^o isometry also gives
\[
 \sqrt r\int_q^T\E_{q,y}H_{ii}(s,Y_s)\,\dd s
 =
 \E_{q,y}
 \bigl[(m_{T,i}-m_{q,i})(B_{T,i}-B_{q,i})\bigr].
\]
By Cauchy--Schwarz and the preceding martingale bound,
\[
 \sqrt r\int_q^T\E_{q,y}H_{ii}(s,Y_s)\,\dd s
 \le\sqrt{T-q}.
\]
Since
\[
 \E_{q,y}H_{ii}(s,Y_s)\ge H_{ii}(q,y)-A,
\]
we obtain
\[
 \sqrt r\,(T-q)\bigl(H_{ii}(q,y)-A\bigr)
 \le\sqrt{T-q},
\]
which proves \eqref{eq:shared-early-diagonal}.

\emph{Averaged curvature.}
Assume now that $\beta\ge1$ and fix $q<1$ with $\gamma(q)\le U$.
We first obtain a lower bound at a time $T\ge q$ for which
$\gamma(T)\ge1$. Refining the step partition if necessary, we may
take $q$ and $T$ to be partition points.

Let $s_j$ be the partition times up to $T$, and let $a_j$ be the
constant value of $\gamma$ on $[s_j,s_{j+1})$. Write
\[
 u_j=u(s_j,Y_{s_j};J).
\]
On such an interval, $e^{a_ju}$ solves the backward heat equation, so
the drifted transition has density
\[
 e^{a_j(u_{j+1}-u_j)}
\]
relative to the corresponding Gaussian increment. Multiplying these
transition densities with the matrix density $Q_\gamma$ gives, for
$T>0$, a density proportional to $e^{V_T}$, where
\begin{equation}\label{eq:response-prefix-left-height}
 V_T
 =
 \gamma(T-)u(T,Y_T;J)
 -
 \int_{[0,T)}
 u(s,Y_s;J)
 \bigl(\dd\gamma(s)-\lambda\delta_0(\dd s)\bigr).
\end{equation}
For $T=0$,
\[
 V_0=\lambda u(0,0;J).
\]
The left limit $\gamma(T-)$ appears because the transition ending at
$T$ uses the profile height immediately before $T$. Thus the explicit coefficient of $u(T,Y_T;J)$ is $\gamma(T-)$,
including when $\dd\gamma$ has an atom at $T$.

Fix a site $i$. Let $\mathcal H_{i,T}$ be generated by all matrix
coordinates not incident to $i$ and by all field increments through
time $T$ at sites other than $i$. Under the product Gaussian reference law, the matrix coordinates
incident to $i$, together with the independent standard Gaussian
variables generating the field increments at site $i$, form a
standard Gaussian vector $G$, independent of $\mathcal H_{i,T}$. Let $\E_0$ denote expectation over
$G$ with $\mathcal H_{i,T}$ fixed. The conditional density under the
$\gamma$-law is then
\[
 \rho_{\gamma,T}(G)
 =
 \frac{e^{V_T(G)}}{\E_0e^{V_T}}.
\]

For every $s\le T$, the coefficient vector of a fixed spin
configuration in the coordinates $G$ has squared norm
\[
 \frac{t(N-1)}N+rs\le1.
\]
Since the gradient with respect to the spin-label offsets is a
probability vector, the map
\[
 G\longmapsto u(s,Y_s;J)
\]
is therefore $1$-Lipschitz. The measure
$\dd\gamma-\lambda\delta_0$ has mass
$\gamma(T-)-\lambda$ on $[0,T)$, and hence
\begin{equation}\label{eq:response-V-lipschitz}
 \operatorname{Lip}(V_T)
 \le2\gamma(T-)-\lambda
 \quad(T>0),
 \qquad
 \operatorname{Lip}(V_0)\le\lambda.
\end{equation}

Negating the coordinates in $G$ and simultaneously flipping spin
$i$ leaves the conditioned variables unchanged. The Gaussian
increments after $T$ at site $i$ can be negated inside their
integrals. Consequently,
\[
 V_T(-G)=V_T(G),
 \qquad
 m_i(T,-G)=-m_i(T,G).
\]

We also need the dependence of $m_i(T)$ on $G$. For a matrix
coordinate $g_{ij}$, take
\[
 s_\sigma=\sigma_i,
 \qquad
 \xi_\sigma=\sigma_i\sigma_j
\]
in \eqref{eq:spin-offset-direction}. This gives
\[
 |\partial_{g_{ij}}m_i(T)|
 \le\sqrt{\frac tN}\,H_{ii}(T).
\]
For a field increment at site $i$ over an interval of length
$\Delta s$, the corresponding bound is
\[
 |\partial_Gm_i(T)|
 \le\sqrt{r\Delta s}\,H_{ii}(T).
\]
Summing over these coordinates yields
\begin{equation}\label{eq:response-gaussian-gradient}
 |\nabla_Gm_i(T)|^2
 \le
 \left[\frac{t(N-1)}N+rT\right]H_{ii}(T)^2
 \le H_{ii}(T)^2.
\end{equation}

For a Lipschitz function $f$ of a standard Gaussian vector,
\begin{equation}\label{eq:response-gaussian-l1}
 \E_0|f-\E_0f|
 \le c_0\E_0|\nabla_Gf|,
 \qquad
 c_0=\sqrt{\pi/2}.
\end{equation}
Indeed, interpolate between two independent Gaussian vectors by a
rotation through angle $\pi/2$. The expected absolute projection of a
standard Gaussian vector onto a fixed vector is
$\sqrt{2/\pi}$ times its norm, and integration over the rotation
angle gives \eqref{eq:response-gaussian-l1}.

Let $K_0>0$ satisfy
\[
 \operatorname{Lip}(V_T)\le K_0,
\]
and set
\[
 R=4c_0K_0.
\]
Choose an even Lipschitz function
$\eta:\mathbb R\to[0,1]$ such that
\[
 \eta=1\ \text{on }[-R,R],
 \qquad
 \eta=0\ \text{outside }[-2R,2R],
 \qquad
 \operatorname{Lip}(\eta)\le R^{-1},
\]
and put
\[
 \chi(G)=\eta\bigl(V_T(G)-\E_0V_T\bigr).
\]
Applying \eqref{eq:response-gaussian-l1} to $V_T$ and then Markov's
inequality gives
\[
 \E_0\chi\ge\frac34.
\]

Suppose now that $\gamma(T)\ge1$. By
\eqref{eq:response-diagonal-basic},
\[
 1-|m_i(T)|
 \le1-m_i(T)^2
 \le H_{ii}(T).
\]
Since $\chi$ is even in $G$ and $m_i(T)$ is odd,
\[
 \E_0[\chi m_i(T)]=0.
\]
Applying \eqref{eq:response-gaussian-l1} to $\chi m_i(T)$ and using
\eqref{eq:response-gaussian-gradient} gives
\[
 \begin{aligned}
 \E_0[\chi|m_i(T)|]
 &\le
 c_0\E_0|\nabla_G(\chi m_i(T))|\\
 &\le
 c_0\E_0[\chi H_{ii}(T)]
 +\frac{c_0K_0}{R}.
 \end{aligned}
\]
Therefore
\[
 \frac34
 \le\E_0\chi
 \le
 (1+c_0)\E_0[\chi H_{ii}(T)]
 +\frac{c_0K_0}{R}.
\]
Since $R=4c_0K_0$,
\begin{equation}\label{eq:response-cutoff-H}
 \E_0[\chi H_{ii}(T)]
 \ge\frac1{2(1+c_0)}.
\end{equation}

The Gaussian exponential-moment bound for a $K_0$-Lipschitz function
gives
\[
 \E_0e^{V_T-\E_0V_T}\le e^{K_0^2/2}.
\]
On the support of $\chi$,
\[
 V_T-\E_0V_T\ge-2R=-8c_0K_0,
\]
and hence
\[
 \rho_{\gamma,T}
 \ge
 e^{-8c_0K_0-K_0^2/2}.
\]
Combining this with \eqref{eq:response-cutoff-H} yields
\begin{equation}\label{eq:response-cutoff-curvature}
 \begin{aligned}
 \E_\gamma[H_{ii}(T)\mid\mathcal H_{i,T}]
 &=\E_0[\rho_{\gamma,T}H_{ii}(T)]\\
 &\ge
 e^{-8c_0K_0-K_0^2/2}
 \E_0[\chi H_{ii}(T)]\\
 &\ge
 \frac{e^{-8c_0K_0-K_0^2/2}}{2(1+c_0)}.
 \end{aligned}
\end{equation}

We now choose $T$. If $\gamma(q)\ge1$, take $T=q$. Since
$\gamma(q)\le U$,
\[
 \operatorname{Lip}(V_q)\le2M_U,
 \qquad
 M_U=\max\{1,U\},
\]
so \eqref{eq:response-cutoff-curvature} applies with
$K_0=2M_U$.

If $\gamma(q)<1$, let $T>q$ be the first subsequent partition time
such that
\[
 \gamma(T)\ge1.
\]
Such a time exists because $\gamma(1)=\beta\ge1$. By the choice of
$T$,
\[
 \gamma(T-)\le1,
\]
so \eqref{eq:response-V-lipschitz} gives
\[
 \operatorname{Lip}(V_T)\le2.
\]
Thus \eqref{eq:response-cutoff-curvature} applies with $K_0=2$. The
curvature bound at $T$ uses the right value $\gamma(T)\ge1$, whereas
the density estimate uses the left value $\gamma(T-)\le1$.

In both cases,
\[
 \E_\gamma H_{ii}(T)
 \ge
 c_U,
 \qquad
 c_U
 =
 \frac{e^{-16c_0M_U-2M_U^2}}{2(1+c_0)}.
\]
Finally, taking expectations in the Hessian evolution gives
\[
 \E_\gamma H_q
 =
 \E_\gamma H_T
 +
 r\int_q^T
 \gamma(s)\E_\gamma H_s^2\,\dd s
 \succeq
 \E_\gamma H_T.
\]
Flipping spin $i$ together with the signs of the incident couplings
and the $i$th field path preserves the joint law and changes the sign
of $H_{ij}$ for $j\ne i$. Hence the averaged off-diagonal entries
vanish. Since the diagonal lower bound holds for every site,
\[
 \E_\gamma H_q\succeq c_UI.
\]
Jensen's inequality then gives
\[
 \zeta_\gamma(q)
 =
 \frac1N\E_\gamma\Tr H_q^2
 \ge
 \frac1N\Tr\bigl[(\E_\gamma H_q)^2\bigr]
 \ge c_U^2.
\]

This proves the two bounds for step profiles. Their extension to
general admissible profiles is given in
\zcref{subsec:response-localization}.
\end{proof}

\subsection{Response and Fisher estimates}

We first derive the response identity that converts the profile
perturbation $v$ into a lower bound for the second variation. In the
magnetization variable, the profile-dependent term in the Legendre
equation is
\[
 \frac{r\gamma_\theta(q)}2|m|^2.
\]
Differentiating in $\theta$ and then twice in $m$ therefore produces
the term $rv(q)I$. We evaluate the resulting equations along the
magnetization process associated with $\gamma_\theta$. For matrices
$A$ and $C$ of the same size, write
\[
 A:C=\Tr(A^TC).
\]

\begin{proof}[Proof of \zcref{prop:gs-response-tools}]
At fixed $N,\beta,J$ and $q<1$, the Hessian
$D_y^2u_{\gamma_\theta}(q,y;J)$ is positive definite and
\[
 y\longmapsto \nabla_yu_{\gamma_\theta}(q,y;J)
\]
maps $\mathbb R^N$ onto $(-1,1)^N$. To see surjectivity, the control
formula gives
\[
 u_{\gamma_\theta}(q,y;J)-\sum_{i=1}^N|y_i|=O_{N,\beta,J,q}(1)
\]
uniformly in $y$. Hence, for every $m\in(-1,1)^N$,
\[
 y\longmapsto u_{\gamma_\theta}(q,y;J)-m\cdot y
\]
is coercive. Its unique minimizer satisfies
$m=\nabla_yu_{\gamma_\theta}(q,y;J)$ and defines the inverse gradient
map.

Differentiating the Legendre equation at fixed $m$ gives
\[
 \partial_q\Phi_\theta
 =
 \frac r2\Tr h+\frac{r\gamma_\theta}2|m|^2,
 \qquad
 \partial_\theta h=-hBh.
\]
Differentiating the first identity in $\theta$ and taking two
derivatives in $m$ yields
\[
 \partial_q B+\frac r2D_m^2\Tr(h^2B)=rvI.
\]
Thus the profile perturbation enters the Hessian response exactly
through the term $rvI$.

Set
\[
 z=\partial_\theta^2\Phi_\theta.
\]
Since
\[
 \partial_\theta^2h
 =
 2hBhBh-h(D_m^2z)h
\]
and $\partial_\theta^2\gamma_\theta=0$, differentiating the Legendre
equation twice in $\theta$ gives
\[
 \partial_q z+\frac r2h^2:D_m^2z
 =
 r\Tr(hBhBh),
 \qquad
 z(1,\cdot)=0.
\]
At $q=0$ and $y=0$, global spin-flip symmetry gives
\[
 m_0=\nabla_yu_{\gamma_\theta}(0,0;J)=0.
\]
Legendre duality therefore gives
\[
 z(0,0)
 =
 -\partial_\theta^2u_{\gamma_\theta}(0,0;J).
\]
Applying It\^o's formula to $z(q,m_q)$ under the law associated with
$\gamma_\theta$, whose generator is
\[
 \partial_q+\frac r2h^2:D_m^2,
\]
and using $z(1,\cdot)=0$, we obtain
\[
 \frac1N\E_{Q_{\gamma_\theta}}
 \partial_\theta^2u_{\gamma_\theta}(0,0;J)
 =
 r\int_0^1e_\theta(q)\,\dd q.
\]

For $\lambda>0$, differentiating the outer logarithmic moment twice
gives
\[
 \begin{aligned}
 \frac{\dd^2}{\dd\theta^2}
 \left[
   \frac1{N\lambda}
   \log\E_g e^{\lambda u_{\gamma_\theta}(0,0;J)}
 \right]
 &=
 \frac1N\E_{Q_{\gamma_\theta}}
 \partial_\theta^2u_{\gamma_\theta}(0,0;J)\\
 &\quad+
 \frac\lambda N
 \Var_{Q_{\gamma_\theta}}
 \left(
   \partial_\theta u_{\gamma_\theta}(0,0;J)
 \right).
 \end{aligned}
\]
The profile penalty in $\mathfrak F_t^c$ is affine in $\theta$, and
the diagonal-completion term is independent of $\theta$. Combining
these observations with the preceding identity proves
\eqref{eq:shared-response}. When $\lambda=0$, the outer logarithmic
moment is interpreted as expectation, giving the same identity without
the variance term.
\end{proof}

\begin{proof}[Proof of \zcref{lem:response-fisher}]
We first prove the pairing estimate. Set
\[
 \mathcal A(q)
 =
 \frac1N\E_\theta\Tr(h^2B),
 \qquad
 \mathcal L_m
 =
 \partial_q+\frac r2h^2:D_m^2.
\]
The response equation for $B$ gives
\[
 \begin{aligned}
 \mathcal L_m\Tr(h^2B)
 &=
 \Tr((\dot h h+h\dot h)B)
 +\Tr\left(
   h^2\left[
     \partial_qB+\frac r2D_m^2\Tr(h^2B)
   \right]
 \right)\\
 &=
 rv\Tr(h^2)
 +\Tr((\dot h h+h\dot h)B).
 \end{aligned}
\]
Dynkin's formula along the magnetization process therefore yields
\begin{equation}\label{eq:response-A-derivative}
 \mathcal A'
 =
 rv\zeta_\theta
 +
 \frac1N\E_\theta
 \Tr((\dot h h+h\dot h)B).
\end{equation}

To bound $\mathcal A$ and the final term in
\eqref{eq:response-A-derivative}, diagonalize $h$ with eigenvalues
$h_i>0$. Since $B$ and $\dot h$ are symmetric, 
\[
 \Tr(hBhBh)
 =
 \sum_{i,j}
 \frac{h_ih_j(h_i+h_j)}2\,B_{ij}^2,
\]
while
\[
 \Tr(\dot h\,h^{-1}\dot h)
 =
 \sum_{i,j}
 \frac{h_i+h_j}{2h_ih_j}\,\dot h_{ij}^2.
\]
Cauchy--Schwarz gives
\[
 |\Tr(h^2B)|^2
 \le
 \Tr(h)\Tr(hBhBh),
\]
and
\[
 \left|
 \Tr((\dot h h+h\dot h)B)
 \right|
 \le
 2
 \bigl[\Tr(hBhBh)\bigr]^{1/2}
 \bigl[\Tr(\dot h\,h^{-1}\dot h)\bigr]^{1/2}.
\]
Taking expectations and applying Cauchy--Schwarz once more gives
\begin{equation}\label{eq:response-A-bounds}
 |\mathcal A|^2
 \le b_\theta e_\theta,
 \qquad
 \left|
 \frac1N\E_\theta
 \Tr((\dot h h+h\dot h)B)
 \right|
 \le
 2\sqrt{\mathcal J_\theta e_\theta}.
\end{equation}

Multiply \eqref{eq:response-A-derivative} by $\varphi$ and integrate.
Since $\varphi$ is compactly supported in $(0,1)$,
\[
 r\int_0^1v\zeta_\theta\varphi\,\dd q
 =
 -\int_0^1\mathcal A\varphi'\,\dd q
 -
 \frac1N\int_0^1
 \varphi\,
 \E_\theta\Tr((\dot h h+h\dot h)B)\,\dd q.
\]
Using \eqref{eq:response-A-bounds} and Cauchy--Schwarz in $q$ yields
\[
 r\left|\int_0^1v\zeta_\theta\varphi\,\dd q\right|
 \le
 \left(\int_0^1e_\theta\,\dd q\right)^{1/2}
 \left[
  \left(\int_0^1b_\theta|\varphi'|^2\,\dd q\right)^{1/2}
  +
  2\left(\int_0^1\mathcal J_\theta\varphi^2\,\dd q\right)^{1/2}
 \right],
\]
which is \eqref{eq:shared-pairing}.

\emph{Fisher estimate.}
Let
\[
 n=\Tr h,
 \qquad
 \mathcal P(q)=\frac1N\E_\theta\Tr\dot h.
\]
The Legendre equation gives
\begin{align}
 \dot h
 &=
 -\frac r2
 h\bigl(D_m^2n+2\gamma_\theta I\bigr)h,
 \label{eq:response-hdot}\\
 \dot n+\frac r2h^2:D_m^2n
 &=
 -r\gamma_\theta\Tr(h^2).
 \label{eq:response-ndot}
\end{align}
For a step profile, $\gamma_\theta$ is constant between successive
jumps. Differentiating \eqref{eq:response-ndot} with respect to $q$
at fixed $m$ on such an interval gives
\[
 \begin{aligned}
 \mathcal L_m\dot n
 &=
 -\frac r2
 (\dot h h+h\dot h):
 \bigl(D_m^2n+2\gamma_\theta I\bigr)\\
 &=
 \Tr\bigl[
   (\dot h h+h\dot h)
   h^{-1}\dot h h^{-1}
 \bigr]\\
 &=
 2\Tr(\dot h\,h^{-1}\dot h),
 \end{aligned}
\]
where \eqref{eq:response-hdot} was used in the second line.
Dynkin's formula therefore gives
\[
 \mathcal P'=2\mathcal J_\theta
\]
between jumps.

At a jump of $\gamma_\theta$, the quantities $h$ and $D_m^2n$ are
continuous, while \eqref{eq:response-hdot} gives
\[
 [\dot h]
 =
 -r[\gamma_\theta]h^2.
\]
Combining the interval evolution with these jumps yields the
distributional identity
\begin{equation}\label{eq:response-P-balance}
 \dd\mathcal P
 =
 2\mathcal J_\theta\,\dd q
 -
 r\zeta_\theta\,\dd\gamma_\theta.
\end{equation}
Moreover,
\[
 |\Tr\dot h|^2
 \le
 \Tr(h)\Tr(\dot h\,h^{-1}\dot h),
\]
so Cauchy--Schwarz in expectation gives
\begin{equation}\label{eq:response-P-bound}
 |\mathcal P|^2
 \le
 b_\theta\mathcal J_\theta.
\end{equation}

Multiply \eqref{eq:response-P-balance} by $\varphi^2$ and integrate.
Compact support of $\varphi$ gives
\[
 2\int_0^1\varphi^2\mathcal J_\theta\,\dd q
 =
 -2\int_0^1\varphi\varphi'\mathcal P\,\dd q
 +
 r\int_{(0,1)}
 \varphi^2\zeta_\theta\,\dd\gamma_\theta.
\]
Set
\[
 X=\int_0^1\varphi^2\mathcal J_\theta\,\dd q,
 \qquad
 Y=\int_0^1|\varphi'|^2b_\theta\,\dd q.
\]
By \eqref{eq:response-P-bound},
\[
 \left|
 \int_0^1\varphi\varphi'\mathcal P\,\dd q
 \right|
 \le\sqrt{XY}.
\]
Hence
\[
 2X
 \le
 2\sqrt{XY}
 +
 r\int_{(0,1)}
 \varphi^2\zeta_\theta\,\dd\gamma_\theta
 \le
 X+Y
 +
 r\int_{(0,1)}
 \varphi^2\zeta_\theta\,\dd\gamma_\theta,
\]
which proves \eqref{eq:shared-fisher} for step profiles.

The approximation in \zcref{subsec:response-localization} extends
both estimates to general admissible profiles. 
\end{proof}

\subsection{Localization and general profiles}
\zlabel{subsec:response-localization}

We first justify the inverse-Hessian calculations by a bounded
localization argument and then pass the estimates proved for step
profiles to general admissible profiles. Both steps use bounds that
are uniform over the step partitions when $N$ and $\beta$ are fixed.

For the inverse Hessian, put
\[
 L(J,y)=\frac N2\|J\|_{\rm op}+\sqrt N\,|y|.
\]
Let $p_{J,y}(\sigma)$ be the terminal Gibbs probability of
$\sigma\in\{-1,1\}^N$. Since
\[
 p_{J,y}(\sigma)\ge 2^{-N}e^{-2\beta L(J,y)},
\]
for every $\xi\in\R^N$,
\[
 \begin{aligned}
 \xi^TH(1,y;J)\xi
 &=\beta\inf_{a\in\R}\sum_\sigma
       p_{J,y}(\sigma)(\xi\cdot\sigma-a)^2\\
 &\ge
 \beta e^{-2\beta L(J,y)}
 \inf_{a\in\R}2^{-N}\sum_\sigma(\xi\cdot\sigma-a)^2\\
 &=\beta e^{-2\beta L(J,y)}|\xi|^2.
 \end{aligned}
\]
Hence
\[
 H(1,y;J)\succeq
 \beta e^{-2\beta L(J,y)}I.
\]
The backward Hessian identity gives
\[
 H(q,y;J)\succeq\E_{q,y}H(1,Y_1;J).
\]
Since the drift of $Y$ has norm at most $\beta\sqrt N$,
\[
 |Y_1|
 \le
 |y|+\beta\sqrt N+\sqrt r\,|B_1-B_q|.
\]
Using Jensen's inequality and
$\E|B_1-B_q|\le\sqrt N$, we obtain
\[
 H(q,y;J)\succeq
 \beta\exp\left\{
 -\beta N\|J\|_{\rm op}
 -2\beta\sqrt N\,|y|
 -2N(\beta^2+\beta)
 \right\}I.
\]
Thus powers of $H^{-1}$ grow at most exponentially in
$\|J\|_{\rm op}+|y|$, uniformly over the step partition.

To control the remaining derivatives, write
\[
 \mathcal L_\theta
 =
 \partial_q+\frac r2\Delta_y
 +r\gamma_\theta m\cdot\nabla_y.
\]
For a step profile,
\[
 w=\partial_\theta u,
 \qquad
 k=\partial_\theta^2u
\]
satisfy
\[
 \mathcal L_\theta w
 =
 -\frac r2v|m|^2,
 \qquad
 \mathcal L_\theta k
 =
 -2rv\,m\cdot\nabla_yw
 -r\gamma_\theta|\nabla_yw|^2,
 \qquad
 w(1)=k(1)=0.
\]
The terminal spin cumulants are bounded at fixed $N,\beta$, while
\[
 |m|\le\sqrt N,
 \qquad
 \|H\|_{\rm op}\le N\beta.
\]
Successive spatial differentiation of these equations and backward
Gronwall therefore give partition-uniform bounds for every spatial
derivative of $\nabla_yu$, $w$, and $k$ used above. Combining these
bounds with the lower bound on $H$ shows that the corresponding
derivatives in magnetization variables are bounded by
\[
 C\exp\bigl\{C(\|J\|_{\rm op}+|y|)\bigr\},
 \qquad
 C=C_{N,\beta}.
\]

The terminal comparison also gives
\[
 |u_\gamma(0,0;J)-u_\gamma(0,0;0)|
 \le
 \frac N2\|J\|_{\rm op},
\]
and hence
\[
 \frac{\dd Q_\gamma}{\dd P_g}
 \le
 C_{N,\beta}
 e^{(\beta N/2)\|J\|_{\rm op}}.
\]
The Gaussian exponential moments of the matrix and Brownian
supremum then make the preceding derivative bounds integrable.
They provide the domination needed to remove the localization and
justify the differentiations under expectation used above.

We next approximate general profiles. Choose common step
approximations $\gamma_{0,n}$ and $\gamma_{1,n}$ that preserve the
outer tilt and endpoint total mass, and set
\[
 \gamma_{\theta,n}
 =
 (1-\theta)\gamma_{0,n}+\theta\gamma_{1,n},
 \qquad
 v_n=\gamma_{1,n}-\gamma_{0,n},
\]
and
\[
 \delta_n
 =
 \|\gamma_{0,n}-\gamma_0\|_1
 +
 \|\gamma_{1,n}-\gamma_1\|_1.
\]
Then
\[
 \sup_{\theta\in[0,1]}
 \|\gamma_{\theta,n}-\gamma_\theta\|_1
 +
 \|v_n-v\|_1
 \le2\delta_n.
\]
The difference equations for $u,w,k$, together with the same spatial
estimates, give, for every derivative order $j$ needed above,
\[
 \sup_{\theta,q,J}
 \left(
 \|u_n-u\|_{C_y^j}
 +\|w_n-w\|_{C_y^j}
 +\|k_n-k\|_{C_y^j}
 \right)
 \le
 C_{j,N,\beta}\delta_n.
\]
Here $C_y^j$ denotes the supremum norm of spatial derivatives through
order $j$. The source terms in the corresponding difference equations
have time $L^1$ norm at most $C\delta_n$, which gives the displayed
bound by backward Gronwall. The uniform convergence of $w_n$ and
$k_n$ then identifies them with the first and second
$\theta$-derivatives of the limiting solution.

It remains to pass the probability laws to the limit. For fixed
$\theta$, write
\[
 \gamma_n=\gamma_{\theta,n},
 \qquad
 \gamma=\gamma_\theta,
\]
and set
\[
 \Delta u_n(J)
 =
 u_{\gamma_n}(0,0;J)-u_\gamma(0,0;J),
 \qquad
 d_n
 =
 \frac{rN}{2}\|\gamma_n-\gamma\|_1.
\]
The profile variation formula gives
\[
 |\Delta u_n(J)|\le d_n,
\]
and therefore
\[
 \frac{\dd Q_{\gamma_n}}{\dd Q_\gamma}(g)
 =
 \frac{e^{\lambda\Delta u_n(J)}}
 {\E_{Q_\gamma}e^{\lambda\Delta u_n(J)}}
 \in
 [e^{-2\lambda d_n},e^{2\lambda d_n}].
\]
Thus the density ratios $\dd Q_{\gamma_n}/\dd Q_\gamma$ converge
uniformly to one. Coupling the
conditional diffusions with the same Brownian motion and using the
spatial convergence above gives convergence of the corresponding
joint laws.

The same convergence yields
\[
 B_n\longrightarrow B
\]
and, at almost every $q$,
\[
 \dot h_n\longrightarrow\dot h,
\]
since monotonicity gives
$\gamma_{\theta,n}(q)\to\gamma_\theta(q)$ at every continuity point
of $\gamma_\theta$. The exponential bounds above provide an
integrable majorant for the products appearing in the response and
Fisher identities. Hence the response identity
\eqref{eq:shared-response}, the pairing estimate
\eqref{eq:shared-pairing}, and the distributional Fisher identity
pass to the limit.

For the Stieltjes term, $\zeta_n$ converges uniformly in $q$, while
$\dd\gamma_{\theta,n}$ converges weakly with fixed total mass.
Therefore
\[
 \int\varphi^2\zeta_{\theta,n}\,\dd\gamma_{\theta,n}
 \longrightarrow
 \int\varphi^2\zeta_\theta\,\dd\gamma_\theta
\]
for the Lipschitz functions $\varphi$ used in
\zcref{lem:response-fisher}. Finally, to pass the averaged-curvature
bound at a fixed $q$, choose the approximating partitions to contain
$q$ and to satisfy
\[
 \gamma_{\theta,n}(q)=\gamma_\theta(q).
\]
Bounded convergence then gives
\eqref{eq:shared-averaged-curvature} for the limiting profile.

\bibliographystyle{plain}
\bibliography{references}

\end{document}